\documentclass[12pt]{article}
\usepackage[margin=1in]{geometry}
\usepackage{amsmath,amsthm,amsfonts,amssymb}
\usepackage[algo2e]{algorithm2e}
\usepackage{graphicx,enumerate}
\graphicspath{ {./figures/} }
\usepackage[title]{appendix}

\usepackage[
            CJKbookmarks=true, 
            bookmarksnumbered=true,
            bookmarksopen=true,
            colorlinks=true,
            citecolor=red,
            linkcolor=blue,
            anchorcolor=red,
            urlcolor=blue
            ]{hyperref}

\usepackage{tikz}
\usepackage{tikz-qtree}

\usetikzlibrary{calc}
\usetikzlibrary{decorations.pathreplacing,decorations.markings}

\tikzstyle{dot}=[circle,fill,black,inner sep=1pt]

\tikzset{
  on each segment/.style={
    decorate,
    decoration={
      show path construction,
      moveto code={},
      lineto code={
        \path [#1]
        (\tikzinputsegmentfirst) -- (\tikzinputsegmentlast);
      },
      curveto code={
        \path [#1] (\tikzinputsegmentfirst)
        .. controls
        (\tikzinputsegmentsupporta) and (\tikzinputsegmentsupportb)
        ..
        (\tikzinputsegmentlast);
      },
      closepath code={
        \path [#1]
        (\tikzinputsegmentfirst) -- (\tikzinputsegmentlast);
      },
    },
  },
  mid arrow/.style={postaction={decorate,decoration={
        markings,
        mark=at position .5 with {\arrow[#1]{stealth}}
      }}},
  early arrow/.style={postaction={decorate,decoration={
        markings,
        mark=at position .2 with {\arrow[#1]{stealth}}
      }}},
}

\def\alternatecolorred{%
    \pgfkeysalso{red}%
    \global\let\alternatecolor\alternatecolorblue 
}
\def\alternatecolorblue{%
    \pgfkeysalso{blue}%
    \global\let\alternatecolor\alternatecolorred 
}

\newcommand{\altred}{\let\alternatecolor\alternatecolorred 
\tikzset{every edge/.append code = {%
    \global\let\currenttarget\tikztotarget 
    \pgfkeysalso{append after command={(\currenttarget)}}
			\alternatecolor
}}
}
\newcommand{\altblue}{\let\alternatecolor\alternatecolorblue 
\tikzset{every edge/.append code = {%
    \global\let\currenttarget\tikztotarget 
    \pgfkeysalso{append after command={(\currenttarget)}}
			\alternatecolor
}}
}

\tikzstyle{vertexdot}=[circle, draw, fill=black, minimum size=3,inner sep=0pt]

\newtheorem{theorem}{Theorem}
\newtheorem{lemma}{Lemma}
\newtheorem{proposition}{Proposition}
\newtheorem{corollary}{Corollary}
\newtheorem{definition}{Definition}

\newtheorem{conjecture}{Conjecture}

\newtheorem{remark}{Remark}

\usepackage{color}
\usepackage{subfigure}
\usepackage{algorithm}
\usepackage{algorithmic}

\usepackage{xspace,prettyref}

\newrefformat{eq}{(\ref{#1})}
\newrefformat{chap}{Chapter~\ref{#1}}
\newrefformat{sec}{Section~\ref{#1}}
\newrefformat{alg}{Algorithm~\ref{#1}}
\newrefformat{fig}{Fig.~\ref{#1}}
\newrefformat{tab}{Table~\ref{#1}}
\newrefformat{rmk}{Remark~\ref{#1}}
\newrefformat{clm}{Claim~\ref{#1}}
\newrefformat{def}{Definition~\ref{#1}}
\newrefformat{cor}{Corollary~\ref{#1}}
\newrefformat{lmm}{Lemma~\ref{#1}}
\newrefformat{prop}{Proposition~\ref{#1}}
\newrefformat{app}{Appendix~\ref{#1}}
\newrefformat{hyp}{Hypothesis~\ref{#1}}
\newrefformat{thm}{Theorem~\ref{#1}}
\newrefformat{ass}{Assumption~\ref{#1}}
\newrefformat{conj}{Conjecture~\ref{#1}}

\newcommand{\ok}{\bar{k}}

\DeclareMathAlphabet{\varmathbb}{U}{bbold}{m}{n}

\renewcommand{\hat}{\widehat}
\renewcommand{\tilde}{\widetilde}

\usepackage{bbm}

\newcommand{\ER}{Erd\H{o}s-R\'{e}nyi\xspace}

\pgfdeclarelayer{background}
\pgfdeclarelayer{foreground}
\pgfsetlayers{background,main,foreground}

\usepackage{mathtools}
\DeclarePairedDelimiter\ceil{\lceil}{\rceil}
\DeclarePairedDelimiter\floor{\lfloor}{\rfloor}

\theoremstyle{plain}

\newtheorem*{theorem*}{Theorem}
\newtheorem*{proposition*}{Proposition}

\usepackage{setspace}

\author{
{\sf David Gamarnik}\thanks{MIT; e-mail: {\tt gamarnik@mit.edu}. }
\and
{\sf Ilias Zadik}\thanks{MIT; e-mail: {\tt izadik@mit.edu}, NYU; email: {\tt zadik@nyu.edu}}
}
\title{The Landscape of the Planted Clique Problem:\\ Dense subgraphs and the Overlap Gap Property}

\begin{document}


\maketitle

\begin{spacing}{1}





\begin{abstract} 
We study the computational-statistical gap of the planted clique problem, where a clique of size $k$ is planted in an \ER graph $G(n,\frac{1}{2})$. The goal is to recover the planted clique vertices by observing the graph. It is known that the clique can be recovered as long as $k \geq \left(2+\epsilon\right)\log n $ for any $\epsilon>0$, but no polynomial-time algorithm is known for this task unless $k=\Omega\left(\sqrt{n} \right)$. Following a statistical-physics inspired point of view, as a way to understand the nature of this computational-statistical gap, we study the landscape of the ``sufficiently dense" subgraphs of $G$ and their overlap with the planted clique. 

Using the first moment method, we present evidence of a phase transition for the presence of the Overlap Gap Property (OGP) at $k=\Theta\left(\sqrt{n}\right)$. OGP is a concept originating in spin glass theory and known to suggest algorithmic hardness when it appears. We further prove the presence of the OGP when $k$ is a small positive power of $n$, and therefore for an exponential-in-$n$ part of the gap, by using a conditional second moment method. As our main technical tool, we establish the first, to the best of our knowledge, concentration results for the $K$-densest subgraph problem for the \ER model $G\left(n,\frac{1}{2}\right)$ when $K=n^{0.5-\epsilon}$ for arbitrary $\epsilon>0$. Our methodology throughout the paper, is based on a certain form of overparametrization, which is conceptually aligned with a large body of recent work in learning theory and optimization.
\end{abstract}
\tableofcontents

\section{Introduction}

We study the planted clique model, first introduced in \cite{JerrumClique}, in which  one observes an $n$-vertex undirected graph $G$ sampled in two stages; in the first stage, the graph is sampled according to an \ER graph $G\left(n,\frac{1}{2}\right)$ and in the second stage, $k$ out of the $n$ vertices are chosen uniformly at random and all the edges between these $k$ vertices are deterministically added (if they did not already exist due to the first stage sampling). 
The $k$-vertex subgraph chosen in the second stage 
is called  the \textit{planted clique} $\mathcal{PC} $. 
The inference task of interest is to recover $\mathcal{PC}$ from observing $G$.

It is known that as long as $k \geq (2+\epsilon) \log_2 n$, the graph $G$ will have only $\mathcal{PC}$ as a $k$-clique in $G$ w.h.p. (see e.g. \cite{BollBook}). In particular under this assumption, $\mathcal{PC}$ is recoverable w.h.p. by the brute-force algorithm which checks every $k$-vertex subset of whether they induce a $k$-clique or not. Note that the exhaustive algorithm requires $\binom{n}{k}$ time to terminate, making it not polynomial-time for the values of $k$ of interest. For any $k \geq (2+\epsilon) \log_2 n$, a relatively simple quasipolynomial-time algorithm with termination time $n^{O\left(\log n\right)}$  recovers $\mathcal{PC}$ w.h.p. as $n \rightarrow + \infty$ (see e.g. the discussion in \cite{VitalyClique} and references therein). Note that a quasipolynomial-time termination time outperforms the termination time of the exhaustive search for $k=\omega\left( \log n\right)$.

The first polynomial-time (greedy) recovery algorithm of $\mathcal{PC}$ came out of the observation in \cite{KuceraClique} according to which when $k \geq C \sqrt{n \log n}$ for some sufficiently large $C>0$, the $k$-highest degree nodes in $G$ are the vertices of $\mathcal{PC}$ w.h.p. A fundamental work \cite{alon1998finding} proved that a polynomial-time algorithm based on spectral methods recovers $\mathcal{PC}$ when $k \geq c \sqrt{n}$ for any fixed $c>0$ (see also   \cite{FeigeClique}, \cite{DeshpandeMontanari}, \cite{PeresClique}  and references therein.) Furthermore, in the regime $k/\sqrt{n} \rightarrow 0$, various computational barriers have been established for the success of certain classes of polynomial-time algorithms, such as the Sum of Squares Hierarchy \cite{SamClique}, the Metropolis Process \cite{JerrumClique} and statistical query algorithms \cite{VitalyClique}. Nevertheless, no general algorithmic barrier such as NP-hardness has been proven for recovering $\mathcal{PC}$ when $k/\sqrt{n} \rightarrow 0$.  The absence of polynomial-time algorithms together with the absence of a complexity theoretic explanation in the regime where $k \geq (2+\epsilon)\log n$ and $ k/\sqrt{n} \rightarrow 0$ gives rise to arguably one of the most celebrated and well-studied computational-statistical gaps in the literature, known as the \textit{planted clique problem.} 

\paragraph{Computational gaps.}
Computational gaps between what existential or brute-force methods promise and what computationally efficient algorithms achieve is an ubiquitous phenomenon in the analysis of algorithmic tasks in random structures. Such gaps arise for example in the study of several \textit{``non-planted"} models like the maximum-independent-set problem in sparse random graphs \cite{gamarnik2014limits}, \cite{coja2011independent}, the largest submatrix problem of a random Gaussian matrix \cite{gamarnik2016finding}, the $p$-spin models~\cite{montanari2018optimization}, \cite{subag2018following}, \cite{gamarnik2019overlap}, 
diluted $p$-spin-model~\cite{GamarnikSpin}, and the  random $k$-SAT model\cite{mezard2005clustering}, \cite{achlioptas2008algorithmic}. Recently, such computational gaps started appearing in \textit{``planted"} inference algorithmic tasks in statistics literature such as the high dimensional linear regression problem \cite{gamarnikzadik}, \cite{gamarnikzadik2}, the tensor principal component analysis (PCA) \cite{AukoshPCA}, \cite{berthet2013complexity} the stochastic block model (see \cite{Abbe17},  \cite{BreslerCom} and references therein) and, of course, the planted clique problem described above. Towards the fundamental study of such computational gaps roughly speaking the following two methods have been considered.

%
%
%
%

\begin{itemize}
\item[(1)] \textbf{Computational gaps: Average-Case Complexity Theory and the central role of Planted Clique} \\
None of the above gaps have been proven to be an NP-hard algorithmic task. Nevertheless, in correspondence with the well-studied worst-case NP-Completeness complexity theory (see e.g. \cite{KarpNP}), some very promising attempts have been made towards building a similar theory for planted inference algorithmic tasks (see e.g.  \cite{berthet2013complexity}, \cite{CaiGaps}, \cite{BertherRIP}, \cite{BreslerCom} and references therein). The goal of this line of research is to show that for two conjecturally computationally hard statistical tasks the existence of a polynomial-time algorithm for one task implies a polynomial-time recovery algorithm for the other. In particular, (computational hardness of) the latter task reduces to (computational hardness of) the former. Notably, the \textit{planted clique problem} seem to play a central role in these developments, similar to the role the boolean-satisfiability problem played in the development of the worst-case NP-completeness theory. Specifically in the context of statistical reduction, multiple statistical tasks in their conjecturally hard regime such as Sparse-PCA \cite{berthet2013complexity}, submatrix localization \cite{CaiGaps}, RIP certification \cite{BertherRIP}, rank-1 Submatrix Detection, Biclustering \cite{BreslerCom} have been proven to reduce to the planted clique problem in the regime $k/\sqrt{n} \rightarrow 0$. 

\item[(2)] \textbf{Computational Gaps: A Spin Glass Perspective (Overlap Gap Property) }\\
For several of the above-mentioned computational gaps, an inspiring connection have been drawn between the geometry of their solution space, appropriately defined, and their algorithmic difficulty. Specifically it has been repeatedly observed that the appearance of a certain disconnectivity property in the solution space called \textit{Overlap Gap Property (OGP)}, originated in spin glass theory, coincides with the conjectured algorithmic hard phase for the problem. Furthermore, it has also been seen that at the absence of this property even simply local improvement based algorithms such as for example
greedy algorithms can exploit the smooth geometry and succeed.

The connection between algorithmic performance and OGP was initially suggested in the study of the celebrated  random $k$-SAT model, independently in \cite{mezard2005clustering}, \cite{AchlioptasCojaOghlanRicciTersenghi}. Then it has been also 
suggested for other both ``non-planted" models such as maximum independent set in random graphs \cite{gamarnik2014limits}, \cite{rahman2014local}, 
but also models with ``planted signal'', such as high dimensional linear regression \cite{gamarnikzadik}, \cite{gamarnikzadik2} and tensor PCA \cite{AukoshPCA}. Despite the fundamental nature of the planted clique problem in the development of average-case complexity theory, OGP has not been studied for the planted clique problem. The study of OGP in the context of the planted clique problem is the main focus of this work.

We next provide some intuition on what OGP is in the context of ``non-planted" problems. Motivated by the study of concentration of the associated Gibbs measures  \cite{TalagrandBook} for low enough temperature, the OGP concerns the geometry of the near (optimal) solutions. It has been observed that any two ``near-optimal" solutions for many such modes exhibit the disconnectivity property stating that that their overlap, measured as a rescaled Hamming distance, is either very large or very small, which we call \textit{the Overlap Gap Property (OGP)}~\cite{AchlioptasCojaOghlanRicciTersenghi}, \cite{achlioptas2008algorithmic}, \cite{montanari2011reconstruction}, \cite{coja2011independent}, \cite{gamarnik2014limits}, \cite{rahman2014local}, \cite{GamarnikSpin}
 \cite{gamarnik2014performance}. For example, the independent
sets achieving nearly maximal size in sparse random graph exhibit the OGP~\cite{gamarnik2014limits}. An interesting rigorous link also appears between OGP and the power of local algorithms. For example OGP has been used
in~\cite{gamarnik2014limits} to establish a fundamental barrier on the power of a class of local algorithms called i.i.d. factors
for finding nearly largest independent sets in sparse random graphs (see also \cite{rahman2014local} for a tighter later result). Similar negative results have been established in the context of the random NAE-K-SAT problem for the Survey propagation algorithm~\cite{gamarnik2014performance}, random NAE-K-SAT for the Walksat algorithm \cite{coja2016walksat},  the max-cut problem in random hypergraphs for the family of i.i.d. factors \cite{GamarnikSpin}, and for Approximate Message 
Passing algorithms in the context of $p$-spin models~\cite{gamarnik2019overlap}.
As mentioned  above, when OGP disappears, the situation appears to change and for many of these problems algorithms succeed~\cite{achlioptas2008algorithmic},\cite{achlioptas2002two},\cite{gamarnikzadik2}. 
Because of this connection it is conjectured that the onset of the phase transition point for the presence of OGP corresponds to the onset of algorithmic hardness.

It is worth mentioning that other properties such as the shattering property and the condensation, which have been extensively studied in the context of random constraint satisfaction problems, such as random K-SAT, are topological properties of the solution space which have been linked with algorithmic difficulty (see e.g. \cite{achlioptas2008algorithmic}, \cite{krzakala2007gibbs} for appropriate definitions).  We would like to importantly point out that neither of them is identical with OGP. OGP implies for trivial reasons the shattering property but the other implication does not hold.  For example, consider the model of random linear equations \cite{Coja17}, where recovery can be obtained efficiently via the Gaussian elimination when the system is satisfiable. In \cite{Coja17} it is established that OGP never appears as the overlaps concentrate on a single point but shattering property does hold in a part of the satisfiability regime. Furthermore, OGP is also not the same with condensation. For example, in the solution space of random $K$-SAT, OGP appears for multioverlaps when 
the ratio of clauses to variables is close to $2^K\log 2 /K$ (up to poly-$\log K$ factors)~\cite{gamarnik2014performance} which is far below condensation which appears around ratio $2^K \log 2$ \cite{krzakala2007gibbs}. It should be noted that in random k-SAT the onset of the apparent algorithmic hardness also occurs around $2^K \log 2/K$  \cite{gamarnik2014performance}, \cite{hetterich2016analysing}. The exact connection between each of these properties and algorithmic hardness is an ongoing and fascinating research direction. 

A more recent development is the study of solution space properties such as OGP and its connections to  algorithmic hardness, in the context of
problems with ``planted signal'', for example for the high dimensional linear regression problem \cite{gamarnikzadik}, \cite{gamarnikzadik2}. 
The strategy followed in those papers is comprised of two steps. First the statistical inference task is reduced  into an average-case optimization task associated with a natural empirical risk objective. Then, as a second step, a geometric analysis of the region of feasible solutions is performed and the OGP (or the lack of it) is established. Interestingly, in this line of work the ``overlaps" considered are between the ``near-optimal" solutions of the optimization task and the planted structure itself. In the present paper we follow a similar path to identify the OGP phase transition point for the planted clique problem.

\end{itemize}

\subsection*{Contribution and Discussion}

In this paper we analyze the presence of OGP for the planted clique problem. We first consider 
the  goal of inferring the planted clique as as a problem of optimizing a certain ``empirical risk" objective and then 
we perform the OGP analysis on the landscape of near-optimal solutions. The first natural choice for the empirical 
risk is the  log-likelihood of the recovery problem which assigns to any $k$-subset $C \subseteq V(G)$ 
the risk value $-\log  \mathbb{P}\left(\mathcal{PC}=C | G\right)$.  
A relatively straightforward analysis of this choice implies that when $k \geq \left(2+\epsilon \right) \log_2n$ 
the only $k$-subset obtaining a non-trivial log-likelihood is the planted clique itself, since there are no other 
cliques of size $k$ in the graph w.h.p. as $n \rightarrow + \infty$. In particular, this perspective of studying the near-optimal solutions and OGP fails to provide anything fruitful.

\paragraph{ The Dense Subgraphs Landscape and OGP.} We adopt instead 
the ``relaxed" \textit{$k$-Densest-Subgraph} objective, which given  a graph $G$  assigns to any $k$-subset $C \subseteq V(G)$ the empirical risk $-|\mathrm{E}[C]|$, that is we consider  $$\mathcal{D}(G):\max_{C \subseteq V(G), |C|=k} |E[C]| ,$$ where by $\mathrm{E}[C]$ we refer to the set of edges in the induced subgraph defined by $C$. Notice that $\mathcal{D}(G)$ is equivalent to maximizing the log-likelihood of a similar recovery problem, the planted $k$-dense subgraph problem where the edges of $\mathcal{PC}$ are only placed with some specific probability $1>p >1/2$ and the rest of the edges are still drawn with probability $\frac{1}{2}$ as before (see e.g.  \cite{BreslerCom} and references therein). Also, notice that, interestingly,  $ \mathcal{D}(G)$ does not depend on the value of $p$; that is it is universal for all values of $p \in (\frac{1}{2},1)$. Now the planted clique model we are interested in can be seen as the extreme case of the planted $k$-dense subgraph problem when $p \rightarrow 1^-$. Furthermore, in this work we analyze the \textit{overparametrized} version of $\mathcal{D}(G)$, $\ok$-densest-subgraph problem, where for some parameter $\ok \geq k$ the focus is on 
\begin{equation}
\label{opt2}\mathcal{D}_{\ok,k}(G): \max_{C \subseteq V(G), |C|=\ok} |E[C]|,
\end{equation}
 while importantly \textit{the planted clique in $G$ remains of size $k$}. In this work we study the following question:

\vspace{0.2cm}

\begin{center} \textit{How much do  near-optimal solution of $\mathcal{D}_{\ok,k}(G)$ intersect the planted clique $\mathcal{PC}$?}
\end{center}
\vspace{0.2cm}

The Overlap Gap Property \textit{($\ok$-OGP) }for the $\ok$-Densest subgraph problem would mean that near-optimal solution of $\mathcal{D}_{\ok,k}(G)$ (sufficiently dense $\ok$-subgraphs of $G$) have \textit{either a  large or small intersection with the planted clique} (see Definition \ref{OGP2} below for more details).

To study the presence of $\ok$-OGP we focus on the monotonicity of the overlap-restricted optimal values for $z=\floor{\frac{\bar{k}k}{n}},\floor{\frac{\ok k}{n}}+1,\ldots,k;$ $$d_{\bar{k},k}(G)(z)=\max_{C \subseteq V(G), |C|=\bar{k}, \mathrm{overlap}(C)=z} |E[C]|,$$ where $\mathrm{overlap}(C):=|C \cap \mathcal{PC}|.$  Note that we define the overlaps beginning from $\floor{\frac{\bar{k}k}{n}}$ as this level of overlap with $\mathcal{PC}$ is trivially obtained from a uniformly at random chosen $\bar{k}$-vertex subgraph.

\paragraph{Monotonicity and the OGP.} It is not hard to see that the monotonicity  of $d_{\ok,k}(G)(z)$ might be linked with the presence or absence of $\ok$-OGP. For example, assume that as an implication of non-monotonicity,
 for some realization of $G$ the curve $d_{\ok,k}$ satisfies that for some  $z^* \in (\floor{\frac{\bar{k}k}{n}},k)$, \begin{equation} \label{TypeM} d_{\ok,k}(G)(z^*) <\min \{d_{\ok,k}(G)(0),d_{\ok,k}(G)(k) \}. \end{equation}
 Then $\ok$-OGP holds. Indeed, choosing any $\mathcal{T}>0$ with $$d_{\ok,k}(G)(z^*) <\mathcal{T} <\min \{d_{\ok,k}(G)(0),d_{\ok,k}(G)(k) \}$$ we notice that (1) since $\mathcal{T} >  d_{\bar{k},k}(G)(z^*) $ any ``dense" $\bar{k}$-subgraph with at least $\mathcal{T}$ edges cannot overlap at exactly $z^*$ vertices with $\mathcal{PC}$ and (2) since $\mathcal{T} <\min \{d_{\ok,k}(G)(0),d_{\bar{k},k}(G)(k) \}$ there exist both zero and full overlap ``dense" $\bar{k}$-subgraphs with that many edges. On the other hand, when the curve is monotonic with respect to overlap $z$, $\bar{k}$-OGP does not hold for a similar reason. Furthermore, note, that when the curve is monotonically increasing the near-optimal solutions of $\mathcal{D}_{\bar{k},k}(G)$ are expected to nearly entirely contain $\mathcal{PC}$  and hence are expected to be  \textit{relevant} for recovery. At the same time,   when it is monotonically decreasing the near-optimal solutions of $\mathcal{D}_{\bar{k},k}(G)$ are expected to have nearly empty intersection with $\mathcal{PC}$, and hence are expected to be  \textit{irrelevant} for recovery.  

\paragraph{Monotonicity of the First Moment Curve.} Using an optimized union-bound argument (first moment method) we obtain a deterministic upper bound function (we call it \textit{first moment curve}) $\Gamma_{\bar{k},k}(z)$ such that for all overlap values $z$, \begin{equation} \label{Upper} d_{\bar{k},k}(G)(z) \leq \Gamma_{\bar{k},k}(z),\end{equation} which is also provably tight, up-to-lower order terms, at the end-point $z=0$ (Proposition \ref{unionkkk}). For this reason, with the hope that $\Gamma_{\bar{k},k}(z)$ provides a tight upper bound in (\ref{Upper}), we perform a monotonicity analysis of $\Gamma_{\bar{k},k}(z)$ instead.

We discover that when $k=o\left(\sqrt{n} \right)$, and when $\bar{k}$ is relatively small satisfying $k \leq \bar{k} \leq k^2$ (thus including $\bar{k}=k$) $\Gamma_{\bar{k},k}$ is \textit{non-monotonic} satisfying a relation similar to (\ref{TypeM}) for some $z^*$. Then we show that  for relatively large $\bar{k}$ satisfying $k^2 \leq \bar{k}$ this function  is in fact \textit{decreasing}. On the other hand, when $k=\omega\left(\sqrt{n}\right)$, for relatively small $\bar{k}$ satisfying $k \leq \bar{k} \leq n^2/k^2$,  $\Gamma_{\bar{k},k}$ is \textit{non-monotonic} satisfying a relation similar to (\ref{TypeM}) for some $z^*$, while for relatively large $\bar{k}$ satisfying $n^2/k^2 \leq \bar{k}$ it is \textit{increasing}. In particular, an interesting phase transition is taking place at the critical size $k=\sqrt{n}$ and high overparametrization $\bar{k}$. A summary is produced in Table 1. Theorem \ref{FM} and the discussion that follows provide exact details of the above statements. 
\begin{table}[t]
\centering 
\begin{tabular}{|c | c | c| } 
\hline
& Low Overparametrization  $\bar{k}$ & High Overparametrization $\bar{k}$ \\ [0.5ex] 
\hline 
$k=o\left(\sqrt{n}\right)$  & $\Gamma_{\bar{k},k}$  non-monotonic & $\Gamma_{\bar{k},k}$ monotonically decreasing \\ 
\hline
$k=\omega\left(\sqrt{n}\right)$ &   $\Gamma_{\bar{k},k}$  non-monotonic  & $\Gamma_{\bar{k},k}$ monotonically increasing  \\ [1ex] 
\hline 
\end{tabular}
\label{table:nonlin} 
\caption{The monotonicity phase transitions of $\Gamma_{\bar{k},k}$ at $k=\sqrt{n}$ and varying $\bar{k}$.}
\end{table}

\begin{figure}[!htb]
     \begin{center}
        \subfigure{%
           \label{fig:first}
           \includegraphics[width=0.9\textwidth]{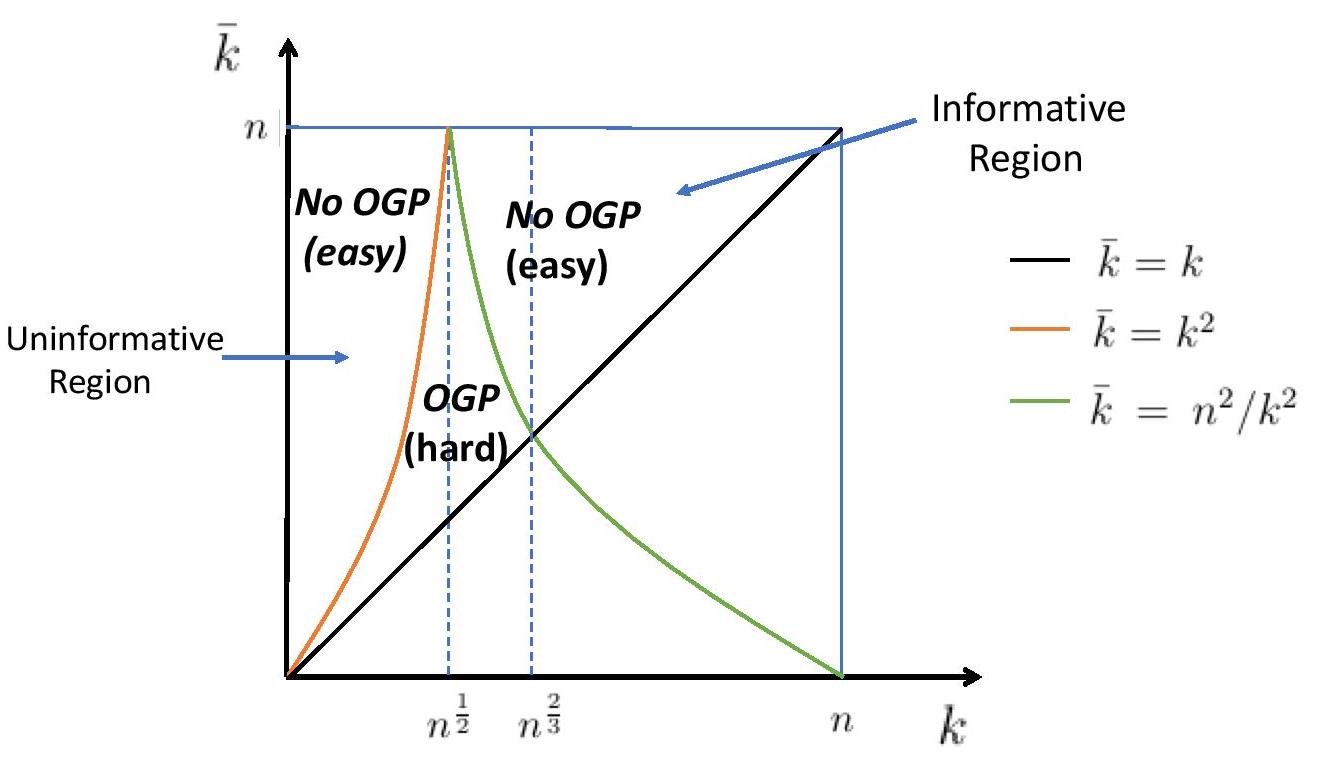}
}

    \end{center}
    \caption{ \small{This is a pictorial representation (phase diagram) of the conjectured appearance of the Overlap Gap Property (OGP) for the space of $\bar{k}$-dense subgraphs in an instance of the planted random graph model $G\left(n,k,\frac{1}{2}\right)$. }}
   \label{fig:Phase}
\end{figure}

Assuming the tightness of $\Gamma_{\bar{k},k}$ in (\ref{Upper}) we arrive at a \textit{conjecture} (Conjecture \ref{mainconj})  regarding the $\bar{k}$-OGP of the landscape which we pictorially describe in Figure \ref{fig:Phase}. In the (apparently) algorithmically hard regime $k=o\left(\sqrt{n} \right)$ the landscape is either exhibiting $\bar{k}$-OGP or is uniformative. On the other hand, in the algorithmically tractable regime $k=\omega\left(\sqrt{n} \right)$ for appropriately large $\bar{k}$ there is no  $\bar{k}$-OGP, the landscape is informative and the optimal solutions of $\mathcal{D}_{\bar{k},k}(G)$ have almost full overlap with $\mathcal{PC}$. Of course this is only a prediction for the monotonicity of $d_{\bar{k},k}(G)$, as the function $\Gamma_{\bar{k},k}$ corresponds only to an upper bound. For this reason we establish results proving parts of the picture suggested by the monotonicity of $\Gamma_{k,\bar{k}}$.

\paragraph{Phase Diagram: Figure \ref{fig:Phase}.}
In Figure \ref{fig:Phase} we see a pictorial representation (phase diagram) of the conjectured appearance of the Overlap Gap Property (OGP) for the space of $\bar{k}$-dense subgraphs in an instance of the planted random graph model $G\left(n,k,\frac{1}{2}\right)$ (Conjecture \ref{mainconj}). For reasons of clarity, the equations describing the curves appearing in the Figure do not take into account logarithmic-in-$n$ terms. We focus on what we call as the overparametrized regime where $\bar{k} \geq k$ (that is we focus on the region above the black line). In this regime, we observe three phases where the boundaries of the phases are defined by the $\bar{k}$-axis, the orange curve ($\bar{k}=k^2$), the green curve ($\bar{k}=n^2/k^2)$ and the black curve ($\bar{k}=k$).  In the upper left phase defined the $\bar{k}$-axis and the orange curve there is no OGP and the problem of recovering the $\bar{k}$-densest subgraph is expected to be algorithmically easy. Yet, this is the regime where the $\bar{k}$-densest subgraph is predicted to have empty intersection with the planted clique, granting it uninformative for recovery purposes of $\mathcal{PC}$. For this reason this region is called ``uniformative region". In the middle phase defined by the orange curve and the green curve there is OGP and the problem is expected to be algorithmically hard. Finally, in the upper right phase defined the black curve and the green curve there is no OGP and the problem of recovering the $\bar{k}$-densest subgraph is predicted to be algorithmically easy. Furthermore, in this phase the $\bar{k}$-densest subgraph is predicted to be fully overlapping with the planted clique, granting it informative for recovery purposes. For this reason we call this region an "informative region". Note the revealing phase transition at $k=\sqrt{n}$, below of which for the small values of $\bar{k}$ there is OGP and for higher values there is no OGP, yet the densest subgraph relaxation is uninformative for recovering the clique. On the other hand, for $k>\sqrt{n}$, for smaller values of $\bar{k}$ there is still OGP but for higher values there is no OGP and the densest subgraph relaxation is informative for recovering the clique. This phase transition is also described in Table \ref{table:nonlin}. Interestingly, assuming the focus was solely on the phase diagram defined by the curve $\bar{k}=k$ the OGP phase transition is predicted to take place at $k=n^{\frac{2}{3}}$ which is far above the predicted algorithmic threshold $k=\sqrt{n}$. For this reason, we consider the use of overparametrization of fundamental importance for studying the landscape of the dense subgraph of $G\left(n,k,\frac{1}{2}\right)$.

\paragraph{Overlap Gap Property for $k \leq n^{0.0917}$.} We establish that under the 
assumption $k \leq  \bar{k} = n^{C}$, for some $0<C<C^*=\frac{1}{2}-\frac{\sqrt{6}}{6} \sim 0.0917..$ (note $k=o \left(\sqrt{n} \right)$)
indeed $\bar{k}$-OGP holds for $\mathcal{D}_{\bar{k},k}\left(G\right)$. The result holds for all values of $C<C^*$ (up-to-$\log$ factors) where the curve $\Gamma_{k,\bar{k}}$ for $k \leq \bar{k}=n^C$  is proven non-monotonic (Theorem \ref{OGP}). Specifically, we establish that for some constants $0<D_1<D_2$ any $\bar{k}$-subgraph of $G$ which is \textit{``sufficiently dense"}  either intersects $\mathcal{PC}$ in \textit{at most $D_1 \sqrt{ \frac{\bar{k}}{\log \frac{n}{\bar{k}}}}$ nodes } or in \textit{at least $D_2 \sqrt{ \frac{\bar{k}}{\log \frac{n}{\bar{k}}}}$ nodes}.  Our proof is based on a delicate second moment method argument for dense subgraphs of \ER graphs. We believe that the second moment method argument can be further improved to extend the result to the case $C^*=0.5-\epsilon$ for arbitrary $\epsilon>0$. We leave this as an important  open question.

\paragraph{Overlap Gap Property implies failure of an MCMC family.} We prove that the conjectured existence of the Overlap Gap Property (as stated in Conjecture \ref{mainconj} and partially established in Theorem \ref{OGP}) provides a rigorous barrier for a family of natural MCMC methods for the densest subgraph problem. The family of MCMC methods consists of running reversible nearest-neighbor dynamics on the space of $\bar{k}$-vertex subgraphs where \begin{itemize} 
\item two $\bar{k}$-subgraphs are considered connected if their set of vertices have Hamming distance two,
\item the associated Gibbs measure to each $\bar{k}$-vertex subgraph $C$ is given by 
\begin{align*} \label{Gibbs} \pi_{\beta}\left(C\right) \propto \exp \left( \beta |\mathrm{E}\left[C\right]|\right),\end{align*} 
\item the  inverse temperature $\beta$ scales at least polylogarithmically with $n$.
\end{itemize} The result applies in the cases where either $k<\sqrt{n} $ and $k \leq \bar{k} \leq k^2$ or $\sqrt{n}<k<n^{\frac{2}{3}}$ and $k \leq \bar{k} \leq n^2/k^2$ (up to logarithmic-in-$n$ terms). 

To be precise, we show that under the above assumptions and appropriate initialization of the dynamics with a $\bar{k}$-subgraph of low overlap size with planted clique, the time required for the dynamics to reach non-trivial overlap with the planted clique is at least $\exp \left( \Omega \left( \beta \frac{\bar{k}}{\left( \log \frac{n}{\bar{k}}\right)} \right)\right)$ (Theorem \ref{thmslow}). We prove the failure of MCMC result by first establishing that under the above assumptions the landscape of $\bar{k}$-dense subgraphs contains ``free-energy wells" of exponential-in-$\bar{k}$ depth (Proposition \ref{Prop1few}), 
which is a provable barrier for local search methods for spin glass models and tensor PCA (see \cite{AukoshPCA} and references therein).

We would like to highlight the following  surprising corollary of this result in the special, but arguably natural choice $\bar{k}=k$.  In this 
case following  Conjecture~\ref{mainconj}, the OGP is expected to hold  all the way up to $k< n^{\frac{2}{3}}$. Our result, therefore, implies that 
under Conjecture~\ref{mainconj} the MCMC method  requires exponential in $n^{O(1)}$ time to recover the planted clique whenever $k \leq n^{\frac{2}{3}}$. Note that $k=n^{\frac{2}{3}}$ is significantly bigger then the conjectured computational threshold $k=O(\sqrt{n})$. 
This failure of MCMC methods for any $k<n^{\frac{2}{3}}$ should be compared with the pioneering work by Jerrum~\cite{JerrumClique} which started the literature on the planted clique problem. In this work, Jerrum proves the failure of Metropolis process towards recovering the planted clique when $k=o(\sqrt{n})$. 
As explained in the Introduction, this failure is treated as the first indication of the now widely-accepted algorithmic hardness of the planted clique problem 
when $k=o(\sqrt{n})$. Nevertheless, importantly, Jerrum does not prove the success of the Metropolis process when $k>\sqrt{n}$. Furthermore, in the Conclusion section of \cite{JerrumClique} he expresses his belief that the failure of recovering the planted clique if proven when $k>n^{\frac{1}{2}+\epsilon},$ for some $ \epsilon>0$, ``would represent a severe indictment of the Metropolis process as a heuristic search technique for large cliques in a random graph". Here we provide strong evidence that a similar family of MCMC methods (not including, though, the Metropolis process) does it indeed fail for $k>n^{\frac{1}{2}+\epsilon},$ for any $ \frac{1}{6}>\epsilon>0$.  Interestingly though, we believe that the performance of the MCMC methods can be salvaged 
all the way down to $k=\sqrt{n}$, by considering an appropriate overparametrization level $\bar{k}$ of the vertex size of the subgraphs. We leave this
as another interesting open question.

\paragraph{The use of Overparametrization.}  
Choosing $\bar{k}>k$ is paramount in  all of the results described here. If we have opted for the arguably more natural choice $\bar{k}=k$, and focused solely on $k$-vertex subgraphs, we would see the monotonicity of the curve $\Gamma_{\bar{k},k}$ exhibiting a phase transition at the peculiar threshold $k=n^{\frac{2}{3}}$ (see Remark \ref{remark} and Figure \ref{fig:Phase}). A significant inspiration for the  overparametrization idea is derived from it's recent success on ``smoothening" bad local behavior in landscapes arising  in the context of deep learning \cite{ShamirNN}, \cite{AfonsoNeural}, \cite{ TengyuNeural}  and beyond  (e.g. \cite{ArianEM} in the context of learning mixtures of Gaussians). We consider this to be a novel conceptual contribution to this line of research on computational-statistical gaps with 
various potential extensions.

\paragraph{$n^{0.5-\epsilon}$-Dense Subgraphs of $G\left(n,\frac{1}{2}\right)$.} Proposition \ref{unionkkk} and Theorem \ref{OGP} are based on a new result on the $K$-Densest subgraph of a vanilla \ER model $G_0 $ sampled from $G \left(n,\frac{1}{2} \right)$;
$$d_{\mathrm{ER},K}(G_0)=\max_{C \subseteq V(G_0), |C|=K} |E[C]|,$$ for any  $K<n^{\frac{1}{2}-\epsilon}$ where $\epsilon>0$. 
The study of $d_{\mathrm{ER},K}(G_0)$ is a natural question in random graph theory which, to the best of our knowledge, remains not well-understood even for moderately large values of $K=K_n$.  For small enough values of $K$, specifically $K<2 \log_2n$, it is well-known $d_{\mathrm{ER},K}(G_0)=\binom{K}{2}$ w.h.p. as $n \rightarrow + \infty$ (originally established in \cite{Grimmet75}). On the other hand when $K=n$, trivially $d_{\mathrm{ER},K}(G_0)$ follows $\mathrm{Binom}\left(\binom{K}{2},\frac{1}{2}\right)$ and hence for any $\alpha_K=\omega\left(1\right)$,  $d_{\mathrm{ER},K}(G_0)=\frac{1}{2}\binom{K}{2}+O\left( K \alpha_K\right)$ w.h.p. as $n \rightarrow + \infty$. If we choose for the sake of argument $\alpha_K=\log \log K$ the following natural question can be posed; \begin{center} \textit{How $d_{\mathrm{ER},K}(G_0)$ transitions from  $\binom{K}{2}$ for $K<2 \log_2 n$ to $\frac{1}{2}\binom{K}{2}+O\left( K \log \log K \right)$ for $K=n$?} \end{center}

A recent result in the literature studies the case $K=C \log n$ for $C>2$ \cite{Bollobas18} and establishes (it is an easy corollary of the main result of the aforementioned paper), \begin{equation} \label{Bela} d_{\mathrm{ER},K}(G_0)=h^{-1} \left( \log 2-\frac{2 \left(1+o\left(1\right)\right)}{C} \right) \binom{k}{2},\end{equation} w.h.p. as $n \rightarrow + \infty$. Here $\log $ is natural logarithm and $h^{-1}$ is the inverse of the (rescaled) binary entropy $h: [\frac{1}{2},0] \rightarrow [0,1]$ is defined by \begin{equation} \label{BinEnt} h(x)=-x \log x-(1-x)\log x. \end{equation} Notice that $\lim_{C \rightarrow + \infty }h^{-1} \left( \log 2-\frac{2 \left(1+o\left(1\right)\right)}{C} \right)  = \frac{1}{2}$ which means that  the result from \cite{Bollobas18} agrees with the first order behavior of $d_{\mathrm{ER},K}(G_0)$ at `` very large" $K$ such as $K=n$.  The proof from  \cite{Bollobas18} is based on a careful and elegant application of the second moment method, where special care is made to control the way ``sufficiently dense" subgraphs overlap.

We study the behavior of $d_{\mathrm{ER},K}(G_0)$ for any $K<n^{\frac{1}{2}-\epsilon}$, for $\epsilon>0$. Specifically, we build and improve on the second moment method technique from \cite{Bollobas18} and establish tight results for the first and second order behavior of $d_{\mathrm{ER},K}(G_0)$ when $K$ is a power of $n$ strictly less than $\sqrt{n}$. Specifically in Theorem \ref{denseeee} we show that for any $ K= n^{C}$ for $C \in (0,\frac{1}{2})$ there exists some positive constant $\beta=\beta(C) \in (0,\frac{3}{2})$ such that \begin{equation}\label{US} d_{\mathrm{ER},K}(G_0)=h^{-1} \left(\log 2-\frac{ \log  \binom{n}{K}  }{\binom{K}{2}} \right)\binom{K}{2} - O\left(K^{\beta}\sqrt{ \log n} \right)\end{equation} w.h.p. as $n \rightarrow + \infty$. 

Note that as our result is established when $K$ is a power $n$, it does not apply in the logarithmic regime. Nevertheless, it is in agreement with the result of of  \cite{Bollobas18} since for $K=C \log n$, $$ \frac{ \log  \binom{n}{K}  }{\binom{K}{2}}=(1+o\left(1\right)) \frac{K \log \left(\frac{n}{K}\right)}{\frac{K^2}{2}}=(1+o\left(1\right)) \frac{2}{C},$$ that is the argument in $h^{-1}$ of (\ref{US}) converges to the argument in $h^{-1}$ of  (\ref{Bela}) at this scaling.

By Taylor expanding $h^{-1}$ around $\log 2$ that is $h^{-1}\left(\log 2-t\right)=\frac{1}{2}+\frac{1}{\sqrt{2}}\sqrt{t}+o\left(\sqrt{t}\right)$ for $t=o\left(1\right)$ (Lemma \ref{EntTaylor}), using our result we can identify the second order  behavior of $d_{\mathrm{ER},K}(G_0)$
 \begin{align*}    d_{\mathrm{ER},K}(G_0) =\frac{K^2}{4}+\frac{K^{\frac{3}{2}}\sqrt{\log  \left(\frac{n}{K}\right) }}{2} + o\left(K^{\frac{3}{2}}\right),
\end{align*} w.h.p. as $n \rightarrow + \infty$. See  Corollary \ref{CorD} for the exact statement. Note that the second order behavior is of different order in $K$ that in the extreme case $K=n$. We leave the analysis of the behavior of $d_{\mathrm{ER},K}(G_0)$ in the regime for $K$ between $n^{\frac{1}{2}}$ and $n$ as an intruiguing open question.

\textit{Connection with $\bar{k}$-OGP.} While our result (\ref{US}) holds for any $K=\Theta\left(n^C\right)$, $0<C<\frac{1}{2}$, 
 we were  able to use this result to prove $\bar{k}$-OGP only when $C<0.0917$. This is  because in order 
 to establish $\bar{k}$-OGP using our non-monotonicity arguments  we needed  the error term in (\ref{US}) to be $o\left(K\right)$, which from our result can only be established if $C<0.0917$. The reason is that in order to transfer the non-monotonicity of the first moment curve $\Gamma_{\bar{k},k}(z)$ 
 to the non-monotonicity of the actual curve $d_{\bar{k},k}(G)$ we needed the error term in our approximation gap between $d_{\ok,k}\left(G\right)(z)$ and $\Gamma_{\ok,k}(z)$, to be sufficiently small so that the non-monotonicity behavior of $\Gamma_{\ok,k}(z)$ transfers to the non-monotonicity behavior of $d_{\ok,k}\left(G\right)(z)$  . We quantify the non-monotonicity behavior of a function via its ``depth", that is for the first moment curve via $$\min \{\Gamma_{\ok,k}\left(0\right),\Gamma_{\ok,k}\left(k\right)\}-\min_{z \in [0,k]} \Gamma_{\ok,k}(z).$$ The latter ``depth" quantity can be proven to grow with order similar to $\Omega\left(K\right)=\Omega\left(\bar{k}\right)$ leading to the necessary order for the error term to make the argument go through.

\paragraph{Notations.}
Throughout the paper we use standard order of magnitude notations. Specifically,
for any real-valued sequences $\{a_n\}_{n \in \mathbb{N}}$ and $\{b_n\}_{n \in \mathbb{N}}$, $a_n=\Theta\left(b_n\right)$
if there exists an absolute constant $c>0$ such that $\frac{1}{c}\le |\frac{a_n}{ b_n}| \le c$; $a_n =\Omega\left(b_n\right)$ or $b_n = O\left(a_n \right)$ if there exists  an absolute constant $c>0$ such that $|\frac{a_n}{ b_n}| \ge c$; $a_n =\omega \left(b_n \right)$ or $b_n = o\left(a_n\right)$ if $\lim_n| \frac{a_n}{ b_n}| =0$. 
For an undirected graph $G$ on $n$ vertices we denote by $V(G)$ the set of its vertices and $E[G]$ the set of its edges. 
Throughout the paper we denote by $h$ the (rescaled) binary entropy given by (\ref{BinEnt}) and for $\gamma \in (\frac{1}{2},1)$, we define \begin{equation} \label{alpha} r(\gamma,\frac{1}{2}):=\log 2-h(\gamma). \end{equation}

\section{Main Results}
\label{Main}
\subsection{The Planted Clique Model and Overlap Gap Property}
We start with formally defining the Planted Clique Model.
 
\paragraph{The Planted Clique Model.} Sample an $n$ vertex undirected graph $G_0$ according to the \ER $G(n,\frac{1}{2})$ distribution. Then choose $k$ out of $n$ vertices of $G_0$ uniformly at random and connect all pairs of these vertices with  undirected edges, creating a planted clique $\mathcal{PC}$ of size $k$. We denote the resulting $n$-vertex undirected graph by $G\left(n,k,\frac{1}{2}\right)$ or $G$ for simplicity.

\paragraph{The Recovery Goal.} Given  $G$ recover the vertices of the planted clique $\mathcal{PC}$.

\subsection{The $\bar{k}$-Densest Subgraph Problem for $\bar{k} \geq k=|\mathcal{PC}|$} We study the landscape of the sufficiently dense subgraphs in $G$.
Besides $n,k$ we introduce an additional parameter $\bar{k} \in \mathbb{N}$ with $k \leq \bar{k} \leq n$. The dense subgraphs we consider are of vertex size $\bar{k}$. We study overlaps between the sufficiently dense $\bar{k}$-dense subgraphs and the planted clique $\mathcal{PC}$. Specifically we focus on the $\bar{k}$-densest subgraph problem on $G$, $\mathcal{D}_{\bar{k},k}(G)$ defined in (\ref{opt2}).
We define the $\bar{k}$-Overlap Gap Property of $\mathcal{D}_{\bar{k},k}(G)$ as follows
   
\begin{definition}[$\bar{k}$-OGP]\label{OGP2}
$\mathcal{D}_{\bar{k}.k}\left(G\right)$ exhibits the $\bar{k}$-Overlap Gap Property ($\bar{k}$-OGP)  if there exists $\zeta_{1,n},\zeta_{2,n} \in [k]$ with $ \zeta_{1,n}<\zeta_{2,n}$ and $0<r_n<\binom{k}{2}$ such that;
\begin{itemize}
\item[(1)]  There exists $\bar{k}$-subsets $A,A' \subseteq V(G)$ with $|A \cap \mathcal{PC}| \leq \zeta_{1,n},$ \\$|A' \cap \mathcal{PC}| \geq \zeta_{2,n}$ and $\min \{ |\mathrm{E}\left[A\right] |, |\mathrm{E}\left[A'\right] | \} \geq r_n$. 
\item[(2)]    For any $\bar{k}$-subset $A \subset V(G)$ with $|\mathrm{E}\left[A\right] |\geq r_n$ it holds,\\ either $|A \cap \mathcal{PC}| \leq \zeta_{1,n}$ or $|A \cap \mathcal{PC}|\geq \zeta_{2,n}.$
\end{itemize}
\end{definition}Here, the first part of the definition ensures that there are sufficiently dense $\bar{k}$-subgraphs of $G$ with both ``low" and ``high" overlap with $\mathcal{PC}$. The second condition ensures that any sufficiently dense $\bar{k}$-subgraph of $G$ will have either ``low" overlap or ``high" overlap with $\mathcal{PC}$, implying gaps in the realizable overlap sizes.

To study $\bar{k}$-OGP we study the following curve. For every $z \in \{\floor{\frac{k\bar{k}}{n}},\floor{\frac{k\bar{k}}{n}}+1,\ldots,k\}$ let
\begin{equation}
\label{opt2}\mathcal{D}_{\bar{k},k}(G)(z): \max_{C \subseteq V(G),|C|=\bar k, |C \cap \mathcal{PC}|=z }|E[C]|.
\end{equation}with optimal value denoted by $d_{\bar{k},k}(G)(z)$. In words,  $d_{\bar{k},k}(G)(z)$ corresponds to the number of edges of the densest $\bar{k}$-vertex subgraph with vertex-intersection with the planted clique of cardinality $z$. Notice that, as explained in the previous section, we restrict ourselves to overlap at least $k\bar{k}/n$ since this level of intersection with $\mathcal{PC}$ is achieved simply by sampling uniformly at random a $\bar{k}$-vertex subgraph of $G$.

\subsection{Monotonicity of the First Moment Curve $\Gamma_{\bar{k},k}$}

The following deterministic curve will be of distinct importance in what follows.
\begin{definition}[First moment curve] \label{GammaDfn}
The first moment curve is the real-valued function $\Gamma_{\bar{k},k}: \{\floor{\frac{k\bar{k}}{n}},\floor{\frac{k\bar{k}}{n}}+1,\ldots,k\}\rightarrow \mathbb{R}_{>0}$, where for $z=\bar{k}=k$, $$\Gamma_{\bar{k},k}(k)=\binom{k}{2}$$ and otherwise $$\Gamma_{\bar{k},k}(z)=\binom{z}{2}+h^{-1} \left(\log 2-\frac{ \log \left(\binom{k}{z} \binom{n-k}{\bar{k}-z} \right) }{\binom{\bar{k}}{2}-\binom{z}{2}} \right)\left(\binom{\bar{k}}{2}-\binom{z}{2}\right),$$ for $z\in \{\floor{\frac{k\bar{k}}{n}},\floor{\frac{k\bar{k}}{n}}+1,\ldots,k\}$, 
\end{definition}Here the function $h^{-1}$ is the inverse function of $h$, which is defined in (\ref{BinEnt}). We establish the following proposition relating $d_{k,\bar{k}}(G)(z)$ and $\Gamma_{\bar{k},k}(G)(z)$.

\begin{proposition}\label{unionkkk}

Let $k,\bar{k},n \in \mathbb{N}$ with $k \leq \bar{k} \leq n$.
\begin{itemize}

\item[(1)] With high probability as $n \rightarrow + \infty$ for every $z\in \{\floor{\frac{k\bar{k}}{n}},\floor{\frac{k\bar{k}}{n}}+1,\ldots,k\}$
$$d_{\bar{k},k}(G)(z) \leq \Gamma_{\bar{k},k}(z).$$
\item[(2)] Suppose $\left( \log n \right)^5 \leq k \leq \bar{k} =\Theta\left(n^{C}\right)$ for $C \in (0,\frac{1}{2})$. For any $\beta \in (0,\frac{3}{2})$ with $\beta=\beta(C) >\frac{3}{2}-\left( \frac{5}{2}-\sqrt{6}\right)\frac{1-C}{C}$,
\begin{align}  \label{Corr}
\Gamma_{\bar{k},k}\left(0\right)-O\left(\left(\bar{k}\right)^{\beta} \sqrt{  \log n} \right) \leq  d_{\bar{k},k}(G)(0),
\end{align} 
with high probability as $n \rightarrow + \infty$.
\end{itemize}
\end{proposition} The bounds stated in Proposition \ref{unionkkk} are based on the first and second moment methods. The proof of Proposition \ref{unionkkk} is in Section \ref{25}.

\begin{remark}
Under the assumptions of Part (2) of Proposition \ref{unionkkk} we have \begin{align*}\Gamma_{\bar{k},k}\left(0\right)&=\frac{1}{2}\binom{\bar{k}}{2}+\left(\frac{1}{\sqrt{2}}+o\left(1\right)\right) \sqrt{\binom{\bar{k}}{2} \log \left[ \binom{n-k}{\bar{k}}\right]}\\
&=\frac{\left(\bar{k}\right)^2}{4}+\frac{\left(\bar{k}\right)^{\frac{3}{2}}\sqrt{\log  \left(\frac{\left(n-k\right)e}{\bar{k}}\right) }}{2} + o\left(\left(\bar{k}\right)^{\frac{3}{2}}\sqrt{ \log n}\right).\end{align*} Here we have used Taylor expansion for $h^{-1}$ around $\log 2$: $h^{-1}\left(\log 2-t\right)=\frac{1}{2}+\left(\frac{1}{\sqrt{2}}+o\left(1\right)\right)\sqrt{t}$ (Lemma \ref{EntTaylor}) for $t=\frac{ \log \left(\binom{n-k}{\bar{k}} \right) }{\binom{\bar{k}}{2}}=O\left( \frac{\log n}{\bar{k}}\right)=o\left(1\right)$ and Stirling's approximation. The above calculation shows that the additive error term in (\ref{Corr}) can change the value of $\Gamma_{\bar{k},k}(0)$ only at the third  order term.
\end{remark}
We explain here how Part (1) of Proposition \ref{unionkkk} is established with a goal to provide intuition for the first moment curve definition. 
Fix some $z \in \{\floor{\frac{\bar{k}k}{n}},\floor{\frac{\bar{k}k}{n}}+1,\ldots,k\}$.
For $\gamma \in (0,1)$ we consider the counting random variable for the number of subgraphs with $\ok$ vertices, $z$ vertices common with the planted clique and at least $\binom{z}{2}+\gamma \left( \binom{\bar{k}}{2}-\binom{z}{2} \right)$ edges; $$Z_{\gamma,z}:=|\{A \subseteq V(G) : |A|=\bar{k}, |A \cap \mathcal{PC}|=z, |E[A]| \geq \binom{z}{2}+\gamma \left( \binom{\bar{k}}{2}-\binom{z}{2} \right) \}|.$$ Notice that first moment method, or simply Markov's inequality, yields $$\mathbb{P}\left[ Z_{\gamma,z} \geq 1 \right] \leq \mathbb{E}\left[Z_{\gamma,z} \right].$$ In particular, if for some $\gamma>0$ it holds $\mathbb{E}\left[Z_{\gamma,z} \right]=o(1)$ we conclude that $Z_{\gamma,z}=0$ whp and in particular all dense subgraphs have at most $\binom{z}{2}+\gamma \left( \binom{\bar{k}}{2}-\binom{z}{2} \right)$ edges, that is $$d_{\bar{k},k}(G)(z) \leq  \binom{z}{2}+\gamma \left( \binom{\bar{k}}{2}-\binom{z}{2} \right),$$ w.h.p. as $n \rightarrow + \infty$. Therefore the pursuit of finding the tightest upper bound using this technique, consists of finding the $ \min \gamma : \mathbb{E}\left[Z_{\gamma,z} \right]=o(1)$.

Note that for any subset $A \subset V(G)$ the number of its induced edges follows a shifted Binomial distribution $\binom{z}{2}+\mathrm{Bin}\left(\binom{\bar{k}}{2}-\binom{z}{2},\frac{1}{2} \right)$. In particular, we have
\begin{align*}
\mathbb{E}\left[Z_{\gamma,z} \right] = \binom{k}{z} \binom{n-k}{\bar{k}-z} \mathbb{P}\left[ \mathrm{Bin}\left(\binom{\bar{k}}{2}-\binom{z}{2},\frac{1}{2} \right) \geq  \gamma \left( \binom{\bar{k}}{2}-\binom{z}{2} \right) \right].
\end{align*} From this point on, standard identities connecting the tail of the Binomial distribution with the binary entropy function $h$ (see for example Lemma \ref{Bin} below) yield the optimal choice to be $$\gamma :=h^{-1} \left(\log 2-\frac{ \log \left(\binom{k}{z} \binom{n-k}{\bar{k}-z} \right) }{\binom{\bar{k}}{2}-\binom{z}{2}} \right),$$ which yields Part (1) if Proposition \ref{unionkkk}. More details are in Section \ref{25}. The part (2) follows from a much more elaborate second moment method, the discussion of which we defer to Subsection \ref{DenseER} and Section  \ref{17}.

 We now study the monotonicity property of the first moment curve. We establish the following proposition which proves that for appropriate choice of the overparametrization level of $\bar{k}$, the first moment curve $\Gamma_{\bar{k},k}(G)$ exhibits a monotonicity phase transitions at the predicted algorithmic threshold $k=\Theta \left(\sqrt{n}\right)$.

\begin{theorem}[Monotonicity Phase Transition at $k=\sqrt{n}$]\text{               } \label{FM} Let $k,\bar{k},n \in \mathbb{N}$ with $n \rightarrow + \infty$ and $\epsilon>0$ an arbitrarily small constant. Suppose $k \leq \bar{k} \leq n$ and  $\left(\log n\right)^5 \leq \bar{k}=o\left(n\right)$. There exist a sufficiently large constant $C_0=C_0\left(\epsilon\right)>0$ such that for the discretized interval $\mathcal{I}=\mathcal{I}_{C_0}=\mathbb{Z} \cap \left[ \floor{C_0\frac{\bar{k}k}{n}},\left(1-\epsilon \right)k\right]$ the following are true for all sufficiently large $n$,
\begin{itemize}
\item[(1)] if $ k=o\left(\sqrt{n}\right)$ then
\begin{itemize}

\item[(1i)] for any $\bar{k}=o\left(\frac{k^2}{ \log \left(\frac{n}{k^2}\right)} \right),$ the function $\Gamma_{\bar{k},k}(z),z \in \mathcal{I}_{C_0}$ is non-monotonic (Figure 1(a)).
\item[(1ii)] for any $\bar{k}=\omega\left(\frac{k^2}{ \log \left(\frac{n}{k^2}\right)} \right),$ the function $\Gamma_{\bar{k},k}(z), z \in \mathcal{I}_{C_0}$ is decreasing (Figure 1(b)).
\end{itemize}
\item[(2)] if $k=\omega\left(\sqrt{n}\right)$ then
 \begin{itemize}
\item[(2i)] for any $\bar{k}=o\left(\frac{n^2}{k^2\log \left(\frac{k^2}{n}\right)}\right),$ the function $\Gamma_{\bar{k},k}(z), z \in \mathcal{I}_{C_0}$ is non-monotonic (Figure 2(a)).
\item[(2ii)] for any $\bar{k}=\omega \left(\frac{n^2}{k^2\log \left(\frac{k^2}{n}\right) } \right),$  the function $\Gamma_{\bar{k},k}(z), z  \in \mathcal{I}_{C_0}$ is increasing (Figure 2(b)).
\end{itemize}
\end{itemize}

Furthermore, in the regime when the function is non-monotonic there are constants $0<D_1<D_2$ and $E>0$ such that for $u_1:=D_1  \ceil{\sqrt{\frac{\bar{k}}{\log \left(\frac{n}{\bar{k}}\right)}}}$ and $u_2:=D_2 \ceil{\sqrt{\frac{\bar{k}}{\log \left(\frac{n}{\bar{k}}\right)}}}$ and large enough $n$ the following are true.
\begin{itemize}
\item[(a)]  $\floor{ C_0\frac{\bar{k}k}{n}}<u_1<u_2<\left(1-\epsilon\right)k$ and 

\item[(b)]
\begin{align}\label{slack2}
  \max_{z \in \mathcal{I} \cap \left[u_1,u_2\right]} \Gamma_{\bar{k},k}(z) +\Omega\left( \frac{\bar{k}}{ \log \left(\frac{n}{\bar{k}}\right)} \right) \leq \Gamma_{\bar{k},k}(\floor{C_0 \frac{\bar{k}k}{n}}) \leq \Gamma_{\bar{k},k}\left(\left(1-\epsilon \right) k\right).
\end{align}
\item[(c)] $\Gamma_{\bar{k},k}\left(z\right)$ is an increasing function for $z \in \left[E \sqrt{\bar{k} \log \frac{n}{\bar{k}}},\left(1-\epsilon\right)k\right]$.
\end{itemize} 

\end{theorem}
The proof of  Theorem~\ref{FM} can be found in Section \ref{4710}.

\begin{remark}\label{remark}
In the special case where $\bar{k}=k$, it is straightforward to check from Theorem \ref{FM} that $\Gamma_{\bar{k},k}$ exhibits a monotonicity phase transition at $k=\Theta\left(n^{\frac{2}{3}}\right)$ and does not exhibit the monotonicity phase transition at $k=\Theta\left(\sqrt{n}\right)$. 
\end{remark}

\begin{remark}
Note that the monotonicity analysis in Theorem \ref{FM} is performed in the slightly ``shrinked" interval  $\mathcal{I}_{C_0}=\mathbb{Z} \cap \left[ \floor{C_0\frac{\bar{k}k}{n}},\left(1-\epsilon \right)k\right]$ for arbitrarily small $\epsilon>0$ and some constant $C_0=C_0(\epsilon)>0$. The restriction is made purely for technical reasons as it allows for an easier analysis of the curve's monotonicity behavior. We leave the monotonicity analysis near the endpoints as an open question.
\end{remark}

\begin{figure}[t]
     \begin{center}
        \subfigure["Low" overparametrization $\bar{k}=k=700$.]{%
           \label{fig:first}
            \includegraphics[width=0.7\textwidth]{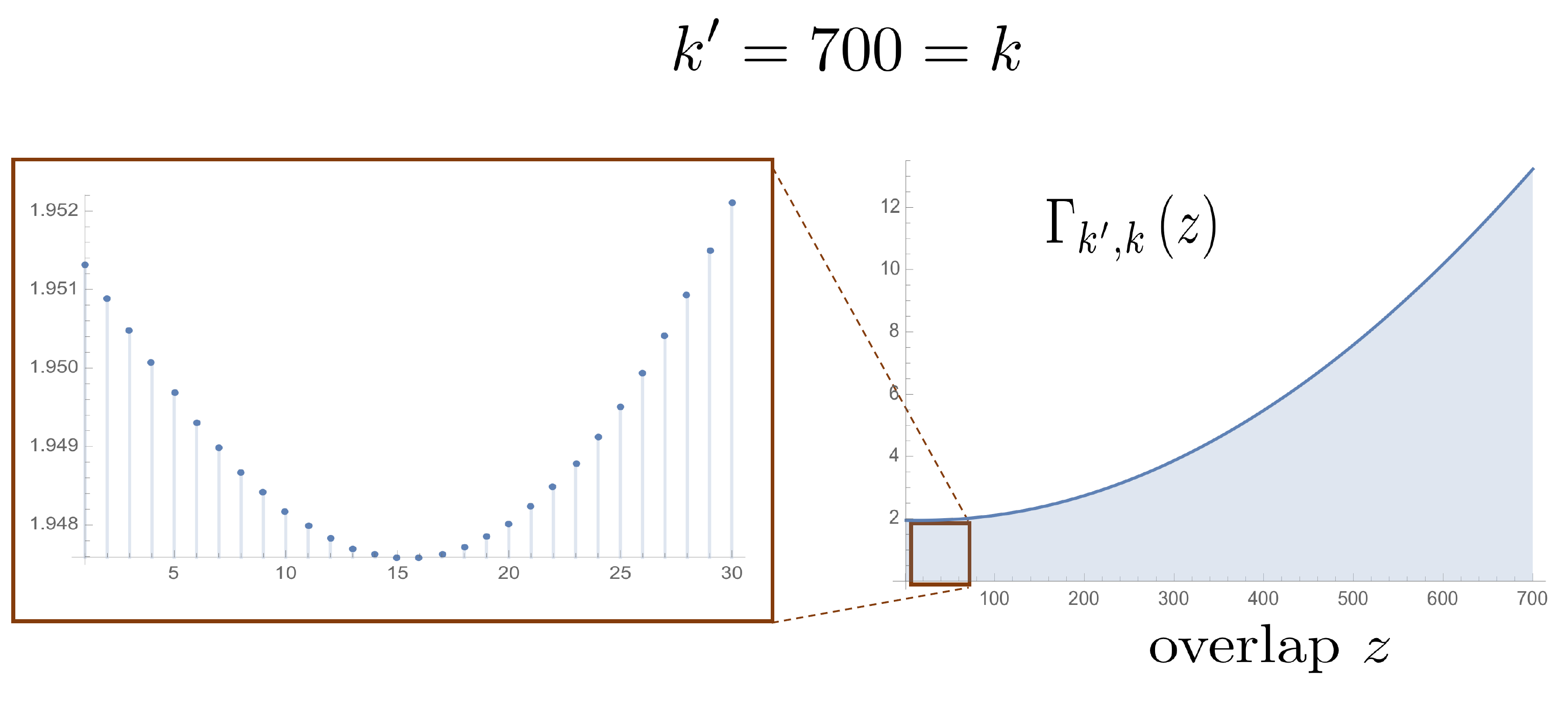}
        }%
\qquad
        \subfigure[``High" overparametrization $\bar{k}=2k^2=980000$.]{%
           \label{fig:second}
           \includegraphics[width=0.36\textwidth]{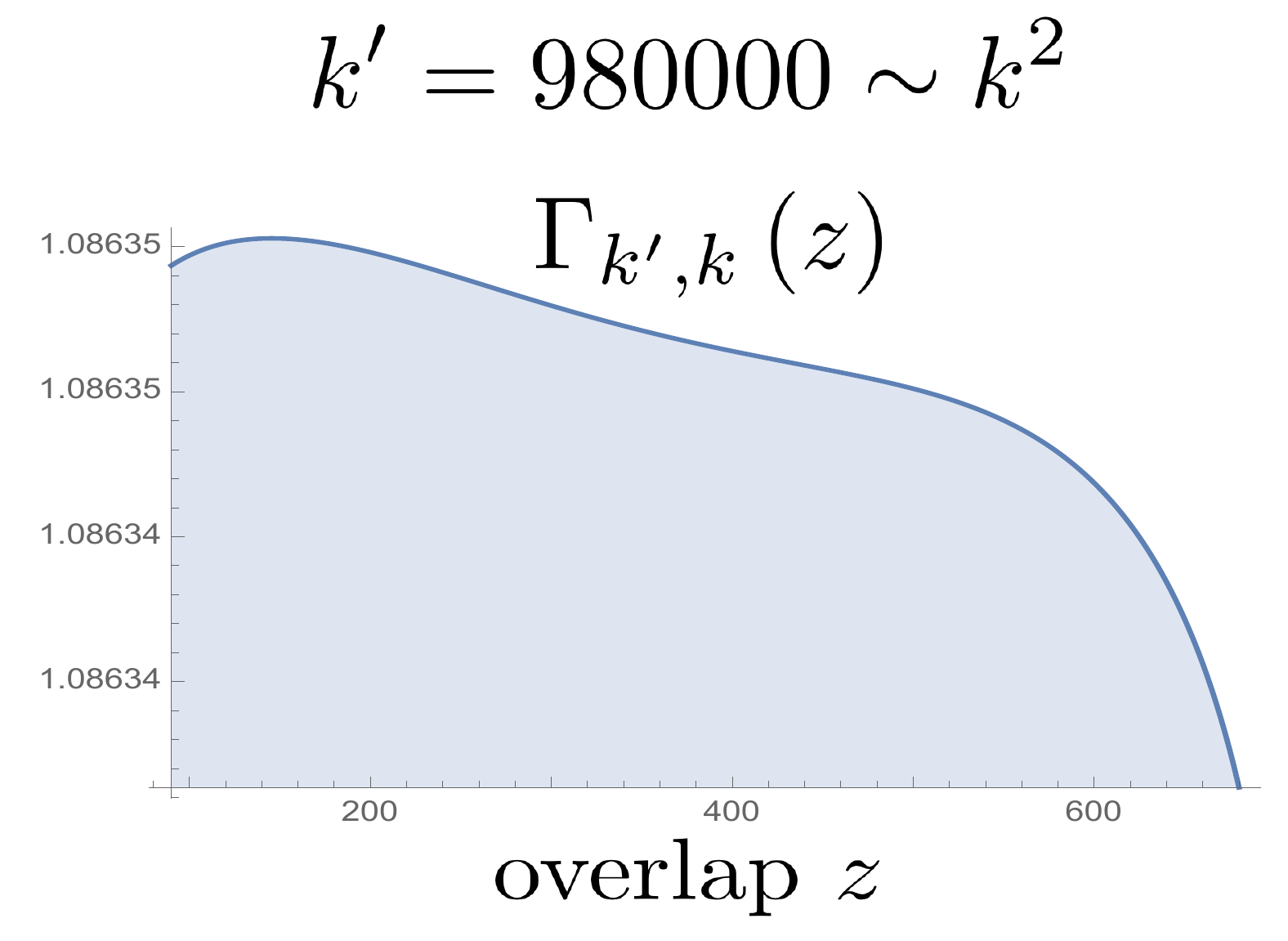}
        }\\ 
    \end{center}
    \caption{ %
       The behavior $\Gamma_{\bar{k},k}$ for $n=10^7$ nodes, planted clique of size $k=700 \ll \floor{\sqrt{n}}=3162$ and ``high" and ``low" values of $\bar{k}$.  We approximate $\Gamma_{\bar{k},k}(z)$ using the Taylor expansion of $h^{-1}$  by $\tilde{\Gamma}_{\bar{k},k}(z)=\frac{1}{2}\left(\binom{k}{2}+\binom{z}{2}\right)+\frac{1}{\sqrt{2}}\sqrt{\left( \binom{k}{2}-\binom{z}{2}  \right) \log \left[ \binom{k}{z}\binom{n-k}{\bar{k}-z}\right]}$. To capture the monotonicity behavior, we renormalize and plot $\left(\bar{k}\right)^{-\frac{3}{2}}\left(\tilde{\Gamma}_{\bar{k},k}(z)-\frac{1}{2}\binom{\bar{k}}{2}\right)$ versus the overlap sizes $z \in [\floor{\frac{\bar{k}k}{n}},k]$.
     }%
   \label{fig:subfigures}
\end{figure}

Theorem \ref{FM} suggests that there are four regimes of interest for the pair $(k,\bar{k})$ and the monotonicity behavior 
of $\Gamma_{\bar{k},k}(z)$. We explain here the implication of Theorem \ref{FM} under the assumption 
that \textit{$\Gamma_{\bar{k},k}(z)$ is a tight approximation of $d_{\bar{k},k}(G)(z)$.}

 Let us focus first on the regime where the size of the planted clique is $k=o \left( \sqrt{n}\right)$. Assume first that the level of overparametrization is relatively small, namely $\bar{k}=o\left(\frac{k^2}{ \log \left(\frac{n}{k^2}\right)} \right)$, including the case $\bar{k}=k$. In that case the curve is non-monotonic and (\ref{slack2}) holds (the case of Figure 1(a)). Now this implies that $\bar{k}$-OGP appears for the model. The reason is that under the tightness assumption, (\ref{slack2}) translates to $$\max_{z \in \mathcal{I} \cap \left[u_1,u_2\right]} d_{\bar{k},k}(G)(z) +\Omega\left( \frac{\bar{k}}{ \log \left(\frac{n}{\bar{k}}\right)} \right) \leq d_{\bar{k},k}(G)(\floor{C_0 \frac{\bar{k}k}{n}}) \leq d_{\bar{k},k}(G)\left(\left(1-\epsilon\right)k\right).$$
We then conclude  that for sufficiently small constant $c>0$ any $\bar{k}$-vertex subgraph with number of edges at least $ d_{\bar{k},k}(G)(\floor{C_0 \frac{\bar{k}k}{n}})-c \frac{\bar{k}}{ \log \left(\frac{n}{\bar{k}}\right)} $ must have either at most $u_1$ intersection with $\mathcal{PC}$ or at least $u_2$ intersection with $\mathcal{PC}$ and there exist subgraphs with both at most $u_1$ and at least $u_2$ intersection with $\mathcal{PC}$ dense subgraphs with at least that many edges.

\begin{figure}[t]\label{fig:2}
     \begin{center}
        \subfigure["Low" overparametrization $\bar{k}=k=4000$.]{%
           \label{fig:first}
            \includegraphics[width=0.65\textwidth]{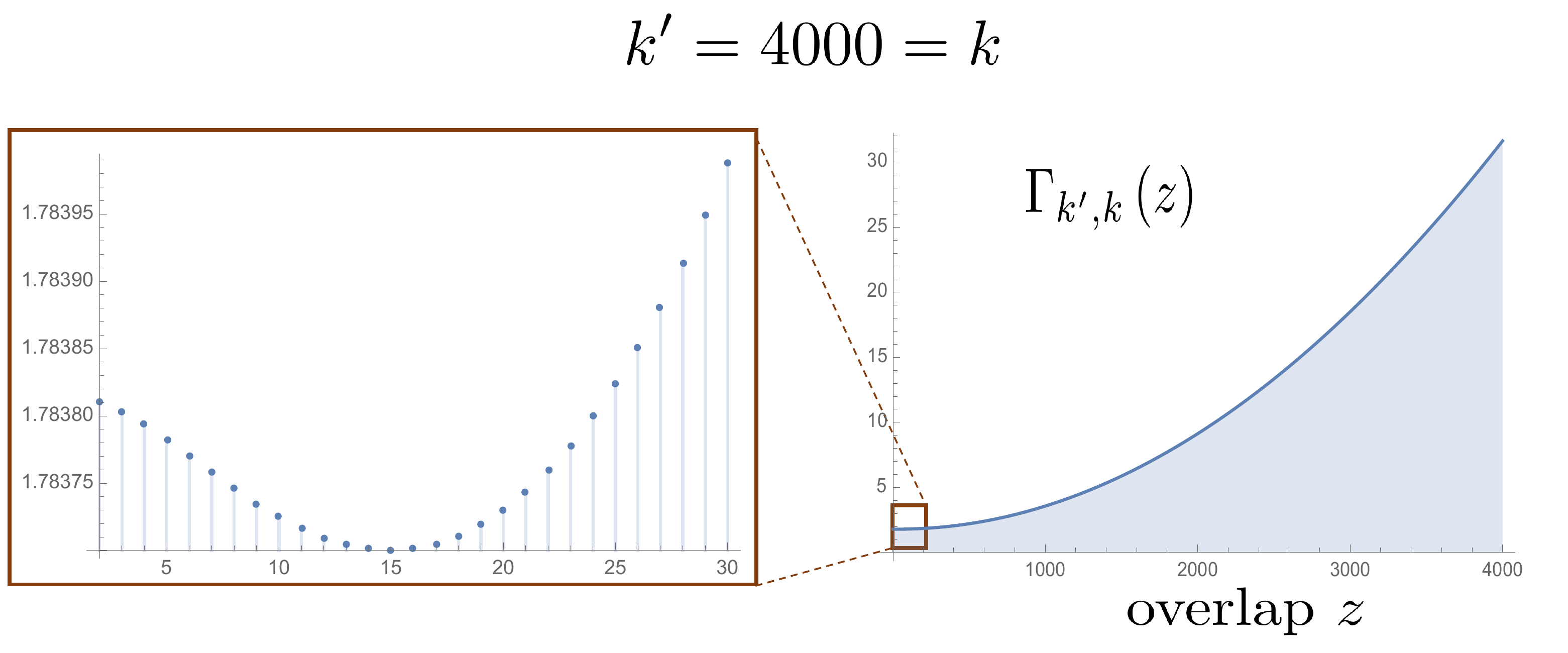}
        }%
\qquad
        \subfigure[``High" overparametrization $\bar{k}=n^2/k^2=6250000$.]{%
           \label{fig:second}
           \includegraphics[width=0.4\textwidth]{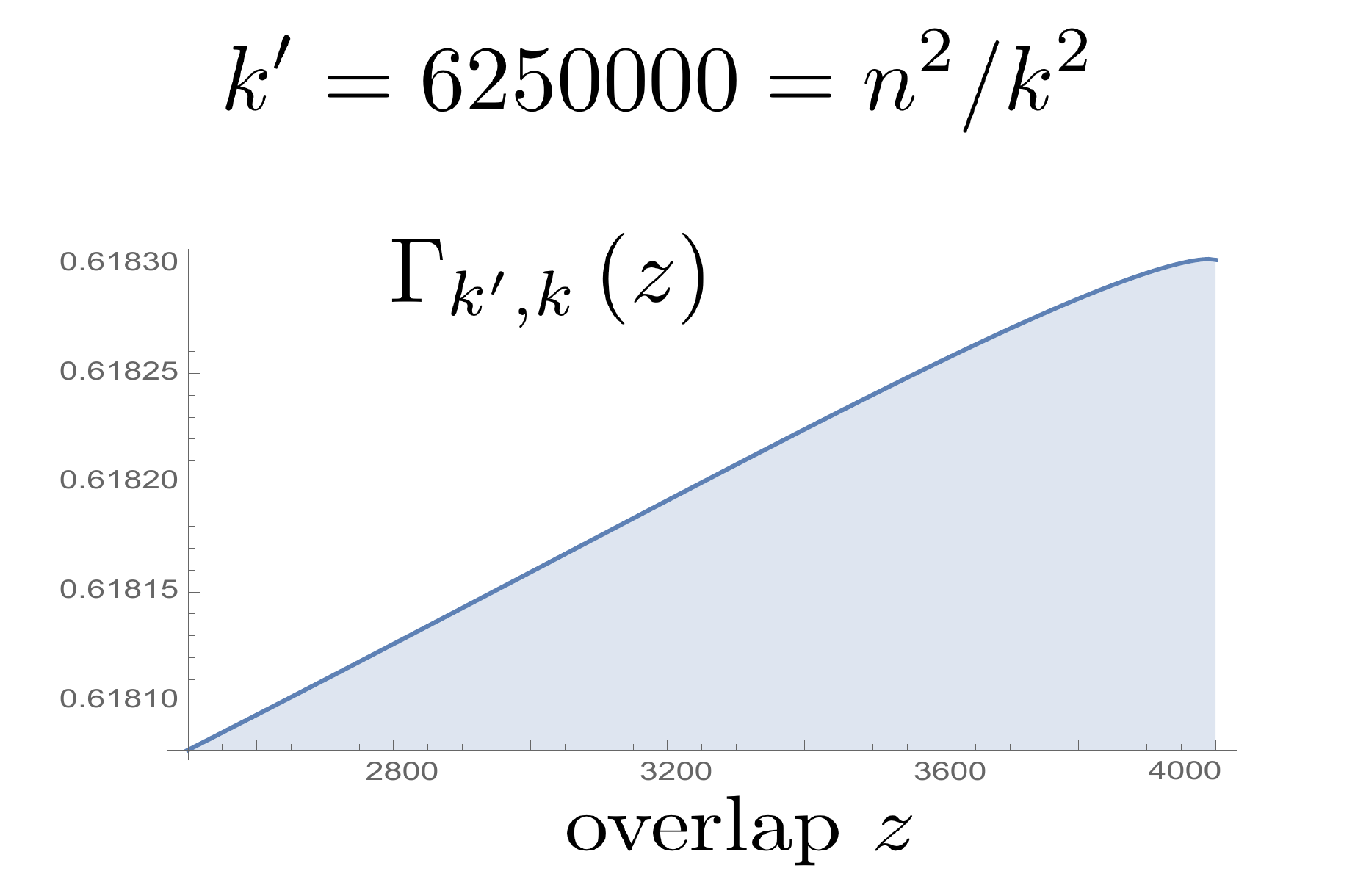}
        }\\ 
    \end{center}
    \caption{
       The behavior $\Gamma_{\bar{k},k}$ for $n=10^7$ nodes, planted clique of size $k=4000 \gg \floor{\sqrt{n}}=3162$ and ``high" and ``low" values of $\bar{k}$.  The rest of the plotting details are identical with that of Figure 1.
     }%
   \label{fig:subfigures}
\end{figure}

Now assume that overparametrization is relatively large, that is $\bar{k}=\omega\left(\frac{k^2}{ \log \left(\frac{n}{k^2}\right)}\right)$. Then the function $\Gamma_{\bar{k},k}(z)$ is decreasing (the case of Figure 1(b)). This is a regime where $\bar{k}$-OGP disappears but the higher overlap $z$ with $\mathcal{PC}$ implies smaller value of $d_{\bar{k},k}(G)(z)$. In particular, in that case one can hope to perhaps 
efficiently find a sufficiently dense subgraphs but they have almost zero intersection with $\mathcal{PC}$. Thus in this case  
the landscape of the dense subgraphs is \textit{uninformative}.
In conclusion, when $k=o\left(\sqrt{n}\right)$ (and again under the tightness assumption)
 either the landscape of the dense subgraphs is either \textit{uninformative or it exhibits $\bar{k}$-OGP}.

Now suppose $k=\omega \left( \sqrt{n}\right)$. Assume first that the overparametrization is relatively small, that is $\bar{k}=o\left(\frac{n^2}{k^2\log \left(\frac{k^2}{n}\right)}\right)$.  In that regime the curve is non-monotonic (see Figure 2(a)). Then, as in the previous case, the 
$\bar{k}$-OGP holds  for the model.

Finally, assuming that the overparametrization is relatively large, that is $\bar{k}=\omega\left(\frac{n^2}{k^2\log \left(\frac{k^2}{n}\right)}\right)$, the function $\Gamma_{\bar{k},k}(z)$ is increasing (see Figure 2(b)). Under the tightness assumption, it then follows that $\bar{k}$-OGP disappears and higher overlap $z$ with $\mathcal{PC}$ implies higher $d_{\bar{k},k}(G)(z)$. This is an \textit{informative} 
case where one can hope to find  sufficiently dense subgraphs, using say a method of local improvements and then use it to find the hidden clique itself.
Notice that in this regime the densest subgraphs almost entirely contain $\mathcal{PC}$.

Summing this up we arrive at the following conjecture based on Theorem \ref{FM}.

\begin{conjecture}\label{mainconj}
Suppose $ \left( \log n \right)^5 \leq k=o(n)$ and  arbitrary $\epsilon \in (0,1)$.
\begin{itemize}
\item[(1)] If $ k=o\left(\sqrt{n}\right)$ then
\begin{itemize}

\item[(1i)] for any $\bar{k}=o\left(k^2 \log \left(\frac{n}{k^2}\right) \right)$ there is $\bar{k}$-Overlap Gap Property w.h.p. as $n \rightarrow + \infty$.
\item[(1ii)] for any $\bar{k}=\omega\left(k^2 \log \left(\frac{n}{k^2}\right) \right)$  there is no $\bar{k}$-Overlap Gap Property, but $d_{\bar{k},k}(G)(z)$ is decreasing as a function of $z$ w.h.p. as $n \rightarrow + \infty$. In particular, the near-optimal solutions of $\mathcal{D}_{\bar{k},k}(G)$ are uniformative about recovering $\mathcal{PC}$. 
\end{itemize}
\item[(2)] if $k=\omega\left(\sqrt{n}\right)$,
 \begin{itemize}
\item[(2i)] for any $\bar{k}=o\left(\frac{n^2}{k^2\log \left(\frac{k^2}{n}\right)}\right)$ there is $\bar{k}$-Overlap Gap Property w.h.p. as $n \rightarrow + \infty$.
\item[(2ii)] for any $\omega \left(\frac{n^2}{k^2\log \left(\frac{k^2}{n}\right) } \right)=\bar{k}=o\left(n \right),$   there is no $\bar{k}$-Overlap Gap Property and $d_{\bar{k},k}(G)(z)$ is increasing as a function of $z$ w.h.p. as $n \rightarrow + \infty$. In particular, the near-optimal solutions of $\mathcal{D}_{\bar{k},k}(G)$ are informative about recovering $\mathcal{PC}$.
\end{itemize}

Furthermore, in the regime that there is $\bar{k}$-Overlap Gap Property there are constants $0<D_1<D_2$ and $E>0$ such that for $u_1:=D_1  \ceil{\sqrt{\frac{\bar{k}}{\log \left(\frac{n}{\bar{k}}\right)}}}$ and $u_2:=D_2 \ceil{\sqrt{\frac{\bar{k}}{\log \left(\frac{n}{\bar{k}}\right)}}}$ the following are true.
\begin{itemize}
\item[(a)] For large enough $n$, $\floor{ C_0\frac{\bar{k}k}{n}}<u_1<u_2<\left(1-\epsilon\right)k$ and 

\item[(b)]
\begin{align}\label{slackConj2}
  \max_{z \in \mathcal{I} \cap \left[u_1,u_2\right]} d_{\bar{k},k}(G)(z) +\Omega\left( \frac{\bar{k}}{ \log \left(\frac{n}{\bar{k}}\right)} \right) \leq d_{\bar{k},k}(G)(\floor{C_0 \frac{\bar{k}k}{n}}) \leq d_{\bar{k},k}(G)\left(\left(1-\epsilon \right) k\right),
\end{align}w.h.p. as $n \rightarrow + \infty$.
\item[(c)] $d_{\bar{k},k}(G)\left(z\right)$ is an increasing function when $z \in [E \sqrt{\bar{k} \log \frac{n}{\bar{k}}},\left(1-\epsilon\right)k]$, w.h.p. as $n \rightarrow + \infty$.
\end{itemize} 

\end{itemize}

\end{conjecture}In the following two subsections we establish rigorously a part  of Conjecture~\ref{mainconj}.

\subsection{$\bar{k}$-Overlap Gap Property for $k=n^{0.0917}$}\label{OGPPP}

We now turn to the regime $k=o(\sqrt{n})$. In this regime Theorem \ref{FM} and Conjecture \ref{mainconj} suggests the presence of $\bar{k}$-OGP when $\bar{k}=o\left(k^2 \log \left( \frac{n}{k^2}\right)\right)$, which includes $\bar{k}=k$. We establish here the result that $\bar{k}$-OGP indeed holds as long as both $k,\bar{k}$ are less than $n^C$ for $C \sim 0.0917..$.

 \begin{theorem}\label{OGP}[$\bar{k}$-Overlap Gap Property]\text{                            }\\
 Suppose $ \left(\log n \right)^5  \leq k \leq \bar{k} =\Theta\left(n^C\right)$ for some $C>0$ with $0<C  < \frac{1}{2}-\frac{\sqrt{6}}{6} \sim 0.0917..$ and furthermore $\bar{k}=o\left(k^2 \log \left( \frac{n}{k^2}\right)\right)$. There exist constants $C_0>0$ and $0<D_1<D_2$ such that for $u_1:=D_1  \ceil{\sqrt{\frac{\bar{k}}{\log \left(\frac{n}{\bar{k}}\right)}}}$ and $u_2:=D_2 \ceil{\sqrt{\frac{\bar{k}}{\log \left(\frac{n}{\bar{k}}\right)}}}$ and large enough $n$ the following holds.
\begin{itemize}
\item[(a)]  $\ceil{ C_0\frac{\bar{k}k}{n}}<u_1<u_2<\frac{k}{2}$ and 

\item[(b)] $d_{\bar{k},k}(G)(z)$ is non-monotonic with
\begin{align}\label{slack22}
  \min\{d_{\bar{k},k}(G)(0),d_{\bar{k},k}(G)\left(\frac{k}{2}\right)\} -\max_{z \in \mathcal{I} \cap \left[u_1,u_2\right]} d_{\bar{k},k}(G)(z) =\Omega \left( \frac{\bar{k}}{ \log \left(\frac{n}{\bar{k}}\right)} \right)
\end{align}with high probability as $n \rightarrow + \infty$.
\end{itemize} 
 In particular, $\bar{k}$-Overlap Gap Property holds for the choice $\zeta_1=u_1,\zeta_2=u_2$ and $r_n= \Gamma_{\bar{k},k}\left(0\right)-\Theta \left( \frac{\bar{k}}{ \log \left(\frac{n}{\bar{k}}\right)} \right),$ with high probability as $n \rightarrow + \infty$.

\end{theorem}

The proof of the Theorem \ref{OGP} is in Section \ref{Sec:OGP}.

\subsection{Overlap Gap Property implies failure of an MCMC family}\label{sec:slowmix}

In this subsection we show that the presence of OGP as stated in Conjecture \ref{mainconj} and partially proven in Theorem \ref{OGP}, implies the provable failure of a family of Markov Chain Monte Carlo methods.

We let $G$ be sampled from the $G\left(n,\frac{1}{2},k \right)$ model and fix $\bar{k}$ with $k \leq \bar{k} \leq n$.

\textbf{The MCMC family} Let $\beta=\beta_n>0$ be a sequence of inverse temperatures, and the Gibbs distributions $\pi_{\beta}$ which is defined on the  $\bar{k}$-vertex subgraphs of $G$, with probability mass function, 
\begin{align} \label{Gibbs} 
\pi_{\beta}\left(S\right)=\frac{1}{Z} \exp \left( \beta |\mathrm{E}\left[S\right]|\right),
\end{align}
where the normalizing partition function constant is
\begin{align}\label{part}
Z:=\sum_{S}\exp \left( \beta |\mathrm{E}\left[S\right]|\right),
\end{align}
and the sum above is over all $\bar{k}$-subgraphs of $G$. Now we define the graph which supports
the Markov Chain underlying the MCMC methods. Let $\mathcal{G}_{\bar{k}}$ be the undirected graph on $\binom{n}{\bar{k}}$ vertices where 

\begin{itemize}
\item each vertex corresponds to a $\bar{k}$-subgraph of $G$ and
\item  each edge connects two nodes if and only if the Hamming distance of their corresponding $\bar{k}$-subgraphs is exactly equal to 2.
\end{itemize}

Let $X_0$ be a random $\bar{k}$-subgraph of $G$ we use for initialization. Then we define by $X_{t,\beta}$, or simply $X_t$ when $\beta$ is clear from the context, the nearest-neighbor Markov chain on $\mathcal{G}_{\bar{k}}$, which is reversible with respect to $\pi_{\beta}$ and initialized at $X_0$.

We describe the assumptions under which the results in this section hold.

\vspace{.1in}
\textbf{Assumptions 1:} Suppose $1 \leq k \leq \bar{k} \leq n$ satisfy,
\begin{itemize}
 \item either $k=o\left(\sqrt{n}\right)$ and $\bar{k}=o\left(k^2 \log \left(\frac{n}{k^2}\right)\right)$ 
\item or $k=\omega\left(\sqrt{n}\right)$ and $\bar{k}=o\left(\frac{n^2}{k^2 \log \left(\frac{k^2}{n}\right)}\right), \bar{k}=o\left(n\right).$ 
\end{itemize}
With respect to Conjecture \ref{mainconj} this corresponds to the cases (1i) and (2i) where OGP is conjectured to appear. Furthermore, note that when $\bar{k}=k$ the above parameter assumptions are satisfied as long as $k \leq n^{\frac{2}{3}}.$

\vspace{.1in}
\textbf{Assumption 2:}  Conjecture~\ref{mainconj} is valid. 
In particular, we assume that the OGP holds  (\ref{slackConj2}) holds.  
\vspace{.1in}

We use the exact same notation as in Conjecture \ref{mainconj}. Recall that case (1i) is partially established above as Theorem~\ref{OGP}.
We now start with our first result which 
shows that under the above assumptions a property called Free Energy Well (FEW)  holds in the landscape of the dense subgraphs. The onset 
of FEW was used to show slow mixing results in several prior models including~\cite{AukoshPCA} and~\cite{gamarnik2019overlap}.

Let us define three subsets of $\bar{k}$-subgraphs. 
 \begin{itemize}
\item $A_0$ is the set of all $\bar{k}$-subgraphs with overlap with the planted clique at most $ \floor*{D_1\sqrt{\frac{\bar{k}}{\log \left( \frac{n}{\bar{k}}\right)}}}$,
\item $A_1$ is the set of all $\bar{k}$-subgraphs with overlap with the planted in the interval 
$$
\left[ \ceil*{D_1 \sqrt{\frac{\bar{k}}{\log \left( \frac{n}{\bar{k}}\right)}}}, \ceil*{D_2 \sqrt{\frac{\bar{k}}{\log \left( \frac{n}{\bar{k}}\right)}}}
\right]
$$ and
\item $A_2$ is the set of all $\bar{k}$-subgraphs with overlap with the planted clique at least $\floor*{k/2}$.
\end{itemize}

We establish the FEW property. It basically says that the set $A_1$ is a ''well'' consuming an exponentially smaller amount of Gibbs mass than $A_0$ 
and $A_2$ and thus separating the two. The FEW property will be established as  a direct corollary of the OGP. 

\begin{proposition}\label{Prop1few}
Under Assumptions 1,2 the following holds.  Let $\beta=\omega\left( \left(\log \left(\frac{n}{\bar{k}}\right)\right)^{\frac{3}{2}}\right)$, then
 \begin{align}\label{few}
\min\{\pi_{\beta}\left(A_0\right),\pi_{\beta}\left(A_2\right)\} \geq \exp\left(\Omega\left(\beta \frac{\bar{k}}{\log \left(\frac{n}{\bar{k}}\right)}\right)\right) \pi_{\beta}\left(A_1\right)
\end{align}w.h.p. (with respect to the randomness of $G$), as $n \rightarrow + \infty$.
\end{proposition}
The proof of Proposition \ref{Prop1few} is  in Section \ref{secMix}.
Intuitively, the result says that the Gibbs mass on both ``low overlap" $\bar{k}$-vertex subgraphs and ``high overlap" $\bar{k}$-vertex subgraphs is larger by an exponentially-in-$\bar{k}$ multiplicative factor than the Gibbs mass of ``medium overlap" $\bar{k}$-vertex subgraphs. This naturally created bottlenecks for the MCMC methods. We formalize this with the following failure of MCMC result.

 Assume that $X_{0}$ is drawn from $\pi_{\beta}\left(\cdot|A_0 \cup A_1\right)$, that is according to $\pi_{\beta}$ conditioned on belonging to $A_0 \cup A_1$. Note that $A_0 \cup A_1$ is simply the set of $\bar{k}$-subgraphs of $G$ with overlap at most $\ceil*{D_2 \sqrt{\frac{\bar{k}}{\log \left( \frac{n}{\bar{k}}\right)}}}$ with the planted clique.
Now consider the stopping time corresponding to the first time that 
$X_{t,\beta}$ reaches a $\bar{k}$-vertex subgraph with overlap with the planted clique strictly bigger than $\ceil*{D_2 \sqrt{\frac{\bar{k}}{\log \left( \frac{n}{\bar{k}}\right)}}}$, that is  \begin{align} \tau_{\beta,\bar{k}}=\inf\{ t \in \mathbb{Z}_{>0} : X_{t,\beta} \not \in A_0 \cup A_1 | X_{0,\beta} \sim \pi_{\beta}\left(\cdot | A_0 \cup A_1\right)\}.\end{align}

We establish the following result.

\begin{theorem}\label{thmslow}
 Under Assumptions 1,2 the following is true.   If $\beta=\omega\left( \left(\log \left(\frac{n}{\bar{k}}\right)\right)^{\frac{3}{2}}\right)$, then
\begin{align}
\tau_{\beta,\bar{k}} =\exp\left(\Omega\left( \beta \frac{\bar{k}}{\log \frac{n}{\bar{k}}}\right)\right),
\end{align}w.h.p. as $n \rightarrow + \infty$.
\end{theorem}

The proof of Theorem \ref{thmslow} is given in Section \ref{secMix}.
Recall that when $\bar{k}=k$ Assumption 1 can be simplified to $k \leq n^{\frac{2}{3}}$. For this reason we obtain 
the following corollary of Theorem~\ref{thmslow}.

\begin{corollary}
 Under Assumptions 2 the following holds.  Let $k=o\left(n^{\frac{2}{3}}\right)$ and $\beta=\omega\left( \left(\log \left(\frac{n}{k}\right)\right)^{\frac{3}{2}}\right)$, then
\begin{align}
\tau_{\beta,k} =\exp\left(\Omega\left( \beta \frac{k}{\log \frac{n}{k}}\right)\right),
\end{align}w.h.p. as $n \rightarrow + \infty$.
\end{corollary}In words, the family of MCMC methods on the $k$-vertex subgraphs for any $k=o\left(n^{\frac{2}{3}}\right)$ require exponential-in-$k$ time to reach a $k$-subgraph with overlap with the planted clique at least equal to $D_2 \sqrt{\frac{k}{\log \left( \frac{n}{k}\right)}}$. Note that the negative result holds for much larger values of $k$ than the conjectured algorithmic threshold $k=O(\sqrt{n})$ and thus refutes the slow mixing of MCMC method as an ''evidence'' of the
hardness of the Hidden Clique Problem, as was originally proposed by Jerrum~\cite{JerrumClique}.

\subsection{$K$-Densest Subgraph Problem for $G \left(n,\frac{1}{2} \right)$}
\label{DenseER}
Of instrumental importance towards Theorem \ref{unionkkk} and Theorem \ref{OGP} is a new result on the value of the densest $K$-subgraph of a vanilla \ER graph $G_0\left(n,\frac{1}{2} \right)$. In this section we present this result. To the best of our knowledge it is the first such result for super-logarithmic-in-$n$ values of $K$ (see \cite{Bollobas18} and the Introduction of the present paper for details).

Let $1 \leq K \leq n$. We study the maximum number of edges of a subgraph of $G_0 \sim G(n,\frac{1}{2})$ with $K$ vertices, that is
\begin{equation}\label{ERdense}
d_{\mathrm{ER},K}(G_0):=\max_{A \subseteq V(G), |A|=K} |\mathrm{E}[A]|.  \end{equation}

We establish the following result.
\begin{theorem}\label{denseeee}Suppose $K = \Theta\left(n^{C}\right)$ for any constant $C \in (0,\frac{1}{2})$. For any fixed $\beta \in (0,\frac{3}{2})$ with $$\beta=\beta(C)> \max\{\frac{3}{2}-\left( \frac{5}{2}-\sqrt{6}\right)\frac{1-C}{C},0\}$$it holds,
\begin{align} \label{DenseSub} &h^{-1} \left(\log 2-\frac{ \log  \binom{n}{K}  }{\binom{K}{2}} \right)\binom{K}{2} - O\left(K^{\beta}\sqrt{ \log n} \right)  \leq  d_{\mathrm{ER},K}(G_0) \leq h^{-1} \left(\log 2-\frac{ \log  \binom{n}{K}  }{\binom{K}{2}} \right)\binom{K}{2},
\end{align} 
with high probability as $n \rightarrow + \infty$.

\end{theorem}
The proof of the theorem is given in Section  \ref{17}.

\begin{remark}
Let $C_{\rm crit}:=5/8-\sqrt{6}/4$ be the unique positive solution to $\frac{3}{2}-\left( \frac{5}{2}-\sqrt{6}\right)\frac{1-C}{C}=0$. Notice that Theorem \ref{denseeee} provides a qualitative different concentration result in the regime where $C \leq C_{\rm crit}$ and when $C > C_{\rm crit}.$ In the former case it implies that for any arbitrarily small constant $\beta>0$ (\ref{DenseSub}) holds, while in the latter case for (\ref{DenseSub}) to hold the exponent $\beta$ needs to be  larger than $\frac{3}{2}-\left( \frac{5}{2}-\sqrt{6}\right)\frac{1-C}{C}>0.$
\end{remark}
For any value of $C \in (0,1/2)$ we can choose some $0<\beta=\beta\left(C\right)<\frac{3}{2}$ so that Theorem \ref{denseeee}, and in particular (\ref{DenseSub}), holds for this value of $\beta$. Combining (\ref{DenseSub}) with a direct applications of the Taylor expansion of $h^{-1}$ (Lemma \ref{EntTaylor}) and the Stirling's approximation for $\binom{n}{K}$ we obtain the following asymptotic behavior of $d_{\mathrm{ER},K}(G)$, for any $K=\Theta\left(n^C\right), C \in \left(0,\frac{1}{2}\right).$
\begin{corollary}\label{CorD}Suppose $K = \Theta\left(n^{C}\right)$ for any fixed $C \in (0,\frac{1}{2})$. Then,
 \begin{align} \label{DenseSub2}   d_{\mathrm{ER},K}(G_0) =\frac{K^2}{4}+\frac{K^{\frac{3}{2}}\sqrt{\log  \left(\frac{n}{K}\right) }}{2} + o\left(K^{\frac{3}{2}}\right),
\end{align} 
with high probability as $n \rightarrow + \infty$. 
\end{corollary}

\section{Proof of Theorem \ref{denseeee}}
In this section we establish Theorem \ref{denseeee}. We first provide a proof techniques section and then establish in separate subsections the lower and upper bounds of (\ref{DenseSub}). Finally an intermediate subsection is devoted to certain key lemmas for the proof.
\label{17}

\subsection{Roadmap}
\label{road}
For $\gamma \in (\frac{1}{2},1)$ let $Z_{K,\gamma}$ the random variable that counts the number of $K$-vertex subgraphs of $G \sim G(n,\frac{1}{2})$ with edge density at least $\gamma $ (equivalently with number of edges at least $\gamma  \binom{K}{2}$), that is
\begin{equation} \label{dfn:z} 
Z_{K,\gamma}:=\sum_{A \subset V(G): |A|=K} 1\left( |E[A]| \geq \gamma \binom{K}{2}\right).
\end{equation}

Markov's inequality (on the left) and Paley-Zygmund inequality (on the right) give \begin{equation} \label{FSMM} \mathbb{E}\left[Z_{K,\gamma}\right] \geq \mathbb{P}\left[Z_{K,\gamma} \geq 1\right] \geq \frac{ \mathbb{E}\left[Z_{K,\gamma}\right]^2}{\mathbb{E}\left[Z^2_{K,\gamma}\right]}.\end{equation} (\ref{FSMM}) has two important implications.

First if for some $\gamma>0$, $$\lim_n \mathbb{E}\left[Z_{K,\gamma}\right]=0$$ then  (\ref{FSMM}) gives $Z_{K,\gamma}=0$ w.h.p. as $n \rightarrow + \infty$ and therefore the densest $K$-subgraph has at most $\gamma  \binom{K}{2}$ edges w.h.p. as $n \rightarrow + \infty$. This is called \textit{the first moment method} for the random variable $Z_{k,\gamma}.$

Second  if for some $\gamma>0$, $$\lim_n \frac{ \mathbb{E}\left[Z_{k,\gamma}\right]^2}{\mathbb{E}\left[Z^2_{K,\gamma}\right]}= 1$$ then $Z_{K,\gamma} \geq 1$ w.h.p. as $n \rightarrow + \infty$ and therefore the densest-$K$ subgraph has at least $\gamma  \binom{K}{2}$ edges w.h.p. as $n \rightarrow + \infty$. This is called \textit{the second moment method} for the random variable $Z_{k,\gamma}.$

Combining the two observations and a Taylor Expansion result described in Lemma \ref{lem:EntDense}, to establish Theorem  \ref{denseeee} it suffices to establish for some $\alpha \leq \beta(C)-\frac{1}{2}$ and $$\gamma=h^{-1} \left(\log 2-\frac{ \log  \binom{n}{K} -O \left(K^{\alpha}  \log n\right) }{\binom{K}{2}} \right),$$ that it holds $$\lim_n \mathbb{E}\left[Z_{K,\gamma}\right]=0, \lim_n \frac{ \mathbb{E}\left[Z_{K,\gamma}\right]^2}{\mathbb{E}\left[Z^2_{K,\gamma}\right]}= 1.$$

We establish the upper bound provided in Theorem \ref{denseeee} exactly in this way, by showing that for $\alpha=0$ and $\gamma=h^{-1} \left(\log 2-\frac{ \log  \binom{n}{K}  }{\binom{K}{2}} \right)$ it holds $\lim_n \mathbb{E}\left[Z_{K,\gamma}\right]=0$. We present this argument in Subsection \ref{upp}.

The lower bound appears much more challenging to obtain. A crucial difficulty is that by writing $Z_{K,\gamma}$ as a sum of indicators  as in (\ref{dfn:z}) and expanding $\mathbb{E}\left[Z^2_{K,\gamma}\right]$ we need to control various complicated ways that two dense $K$-subgraphs overlap. This is not an uncommon difficulty in the literature of second moment method applications where certain conditioning is usually necessary for the second moment method to provide tight results (see e.g.~\cite{Jiaming18}, \cite{gamarnikzadik}, \cite{Wu18}, \cite{BandeiraNotes}, \cite{zadik19} and references therein).

To control the ways dense subgraphs overlap we follow a similar, but not identical, path to \cite{Bollobas18} which analyzed the $K$-densest subgraph problem for $K=\Theta\left( \log n\right)$ and also used a conditioning technique. We do not analyze directly the second moment of $Z_{K,\gamma}$ but instead we focus on the second moment for another counting random variable that counts sufficiently dense subgraphs satisfying also an additional \textit{flatness condition.} The condition is established to hold with high probability under the \ER structure (Lemma \ref{flat}) and under this condition the dense subgraphs overlap in more ``regular" ways leading to an easier control of the second moment. More details and the analysis of the second moment method under the flatness condition are in Subsections \ref{DenseLem} and \ref{low}.

\subsection{Proof of the Upper Bound}

\label{upp}

Using (\ref{FSMM}) it suffices to show that for $\gamma :=h^{-1} \left(\log 2-\frac{ \log  \binom{n}{K} }{\binom{K}{2}} \right)$, $\mathbb{E}[Z_{K,\gamma}]=o(1)$.

We have by linearity of expectation and (\ref{dfn:z})
\begin{align}
\mathbb{E}\left[Z_{K,\gamma} \right] &= \binom{n}{K} \mathbb{P}\left[ |\mathrm{E}[A]| \geq  \gamma  \binom{K}{2}  \right],\text{for some } A \subseteq V(G), |A|=K \notag\\
&= \binom{n}{K} \mathbb{P}\left[ \mathrm{Bin}\left(\binom{K}{2},\frac{1}{2} \right) \geq  \gamma \binom{K}{2} \right] \label{Step1FMM}
\end{align}

Using the elementary inequality $\binom{n}{K} \leq n^{K}$ we have \begin{equation}  \label{simpleee} \frac{ \log \binom{n}{K} }{\binom{K}{2}}=O\left(\frac{\log n}{K}\right)=o(1)\end{equation} since by our assumption $\omega( \log n)=K$.

By Lemma \ref{EntTaylor} and (\ref{simpleee}) we have, $$\gamma=\frac{1}{2}+\Omega \left( \sqrt{ \frac{ \log \binom{n}{K} }{\binom{K}{2}} }\right)= \frac{1}{2}+o(1).$$ Therefore $\lim_n \gamma = \frac{1}{2}$ and by Stirling's approximation, $$ \left(\gamma-\frac{1}{2}\right) \sqrt{\binom{K}{2}}=\Omega\left( \sqrt{\log \binom{n}{K} } \right)=\Omega \left(\sqrt{K \log \frac{n}{K}} \right)=\omega(1). $$ Hence both assumptions of Lemma \ref{Bin} are satisfied and hence (\ref{Step1FMM}) implies
\begin{align}
\mathbb{E}\left[Z_{K,\gamma} \right] \leq \binom{n}{K} O\left(\exp\left(-\binom{K}{2} r(\gamma,\frac{1}{2}) -\Omega \left(\sqrt{K\log \frac{n}{K}} \right)\right)\right), \label{FMpf11}
\end{align}where recall that $r(\gamma,\frac{1}{2})$ is defined in (\ref{alpha}).
Now notice that for our choice of $\gamma$, $$r(\gamma,\frac{1}{2})=\log 2-h(\gamma)=\frac{ \log  \binom{n}{K} }{\binom{K}{2}}.$$ In particular using (\ref{FMpf11}) we conclude that 
\begin{align}
\mathbb{E}\left[Z_{K,\gamma} \right] = \exp\left(- \Omega \left(\sqrt{K \log \frac{n}{K}} \right)\right)=o(1).
\end{align} The completes the proof of the upper bound.

\subsection{$(\gamma,\delta)$-flatness and auxiliary lemmas} \label{DenseLem}

We start with appropriately defining the flatness condition mention in Subsection \ref{road}. Specifically, for $K \in \mathbb{N}$ we introduce a notion of a $(\gamma,\delta)$-flat $K$-vertex graph $G$, where $\gamma,\delta \in (0,1)$. This generalizes the corresponding definition from \cite[Section 3]{Bollobas18}.

 For $ 0 \leq \ell \leq K$ let
\begin{align} \label{Dk}
D_K(\ell,\delta):= \left\{
\begin{array}{ll} 
      \sqrt{2\gamma(2+\delta) \min \left(\binom{K}{2}-\binom{\ell}{2},\binom{\ell}{2}\right) \left( \log \binom{K}{\ell}+2\log K \right)}& 0 \leq \ell<\frac{2K}{3} \\
     \sqrt{2\gamma (1+\delta) \min \left(\binom{K}{2}-\binom{\ell}{2},\binom{\ell}{2} \right) \left( \log \binom{K}{\ell}+2\log K \right)} &  \frac{2K}{3} \leq \ell \leq K\\
\end{array} 
\right. 
\end{align}

\begin{definition}[$(\gamma,\delta)$-flat graph] \label{dfn:flat}
Call a $K$-vertex graph $G$, $(\gamma,\delta)$-flat if \begin{itemize}
\item $|E[G]|=\ceil*{\gamma \binom{K}{2}}$ and
\item  for all $A \subset V(G)$ with $\ell=|A| \in \{2,3,\ldots,K-1\}$ we have $|E[A]| \leq \ceil*{\gamma \binom{\ell}{2}}+D_K(\ell,\delta)$.
\end{itemize}
\end{definition} Notice that a $(\gamma,\delta)$-flat subgraph of $G \sim G(n,\frac{1}{2})$ has edge density approximately $\gamma$ and is constrained to do not have arbitrarily dense subgraphs. In particular, two $(\gamma,\delta)$-flat subgraphs of $G$ cannot overlap in ``extremely" dense subgraphs. This property leads to an easier control of the second moment of the random variable which counts the number of $(\gamma,\delta)$-flat subgraphs compared to the second moment of $Z_{K,\gamma}$ defined in Definition \ref{dfn:z}. Using the second moment method we establish the existence of an appropriate ($\gamma,\delta)$-flat subgraph leading to the desired lower bound stated in Theorem \ref{denseeee}. Even under the flatness restriction, the control of the second moment remains far from trivial and requires a lot of careful and technical computations. For this reason we devote the rest of this subsection on stating and proving four auxiliary lemmas. In the following subsection we provide the proof of the lower bound.

\begin{lemma}\label{gamma}
Let $\alpha \in (0,1)$. Suppose $K=\Theta(n^C)$ for $C \in (0,1)$. \\For any $\gamma$ satisfying $\gamma=h^{-1} \left(\log 2-\frac{ \log  \binom{n}{K} -O\left(K^{\alpha} \log n \right) }{\binom{K}{2}} \right)$ it holds $$\gamma=\frac{1}{2}+\left(1+o(1)\right)\sqrt{ \frac{ \log \frac{n}{K}}{K}}=\frac{1}{2}+ \Theta \left(\sqrt{ \frac{ \log n}{K}} \right).$$Furthermore,
$$r \left(\gamma,\frac{1}{2} \right)=\log 2- h\left(\gamma\right)=\left(1+o\left(1 \right)\right)\frac{ 2\log \frac{n}{K}}{K}=\Theta\left(\frac{ \log n}{K}\right).$$

\end{lemma}

\begin{proof}
We first observe that since $K=\Theta(n^C)$ for $C \in (0,1)$ by Stirling approximation we have $\log \binom{n}{K}=\left(1+o(1)\right)K \log \frac{n}{K}.$ Therefore, since $C<1$ and $\alpha<1$,  it also holds 
\begin{align*}
\frac{ \log  \binom{n}{K} -O\left(K^{\alpha} \log n \right) }{\binom{K}{2}} =\left(1+o\left(1\right)\right)\frac{K \log \frac{n}{K}}{\frac{K^2}{2}}=\left(1+o\left(1\right)\right)\frac{2   \log  \frac{n}{K}}{K}.
\end{align*}
Hence $\gamma$ satisfies
 \begin{equation} \label{Gammaaa} \gamma=h^{-1} \left(\log 2-\left(1+o\left(1\right)\right)\frac{2   \log  \frac{n}{K}}{K} \right).\end{equation} By Lemma \ref{EntTaylor} we have $h^{-1}\left( \log 2-\epsilon\right)=\frac{1}{2}+\left(\frac{1}{\sqrt{2}}+o\left(1\right)\right)\sqrt{\epsilon}$. Since $\frac{2   \log  \frac{n}{K}}{K}=o\left(1\right)$ we have that $$\gamma=\frac{1}{2}+\left(1+o(1)\right)\sqrt{ \frac{ \log \frac{n}{K}}{K}}=\frac{1}{2}+ \Theta \left(\sqrt{ \frac{ \log n}{K}} \right).$$
Furthermore by (\ref{Gammaaa}) we directly have $$r \left(\gamma,\frac{1}{2} \right)=\log 2- h\left(\gamma\right)=\left(1+o\left(1 \right)\right)\frac{ 2\log \frac{n}{K}}{K}=\Theta\left(\frac{ \log n}{K}\right).$$

\end{proof}

\begin{lemma}\label{lem:EntDense} Suppose $\omega( \log n)=K=o(\sqrt{n})$. Then for any fixed $\alpha \in (0,1)$, $$h^{-1} \left(\log 2-\frac{ \log  \binom{n}{K} -O \left(K^{\alpha} \log n\right) }{\binom{K}{2}} \right)\binom{K}{2}=h^{-1} \left(\log 2-\frac{ \log  \binom{n}{K}  }{\binom{K}{2}} \right)\binom{K}{2} - O\left(K^{\alpha+\frac{1}{2}} \sqrt{ \log n} \right).$$

\end{lemma}

\begin{proof}

Equivalently we need to show that 

$$h^{-1} \left(\log 2-\frac{ \log  \binom{n}{K} -O \left(K^{\alpha} \log n\right) }{\binom{K}{2}} \right)=h^{-1} \left(\log 2-\frac{ \log  \binom{n}{K}  }{\binom{K}{2}} \right) - O\left(K^{\alpha-\frac{3}{2}} \sqrt{\log n} \right).$$
Now from Lemma \ref{EntTaylor} we know that for $\epsilon=o\left(1\right)$, $h^{-1}\left(\log 2-\epsilon \right)=\frac{1}{2}+\Theta \left( \sqrt{\epsilon} \right).$ By Stirling approximation since $K=o\left(\sqrt{n}\right)$ we have $\binom{n}{K} =\Theta \left( \left(\frac{ne}{K}\right)^K\right)$. Using $\alpha \in (0,1)$, $$ \log  \binom{n}{K} =\Theta \left(K \log \left(\frac{ne}{K}\right)\right)=\omega \left(K^{\alpha} \log n \right).$$ Hence, $$\big{|}\frac{ \log  \binom{n}{K} -O \left(K^{\alpha} \log n\right) }{\binom{K}{2}} \big{|} = O \left( \frac{\log n}{K} \right)=o\left(1 \right).$$ Therefore by Lemma \ref{EntTaylor}
\begin{align*}
&h^{-1} \left(\log 2-\frac{ \log  \binom{n}{K} }{\binom{K}{2}} \right)-h^{-1} \left(\log 2-\frac{ \log  \binom{n}{K} -O\left(K^{\alpha-2} \log n \right) }{\binom{K}{2}} \right)\\
&=\Theta\left(\sqrt{\frac{ \log  \binom{n}{K} }{\binom{K}{2}} }-\sqrt{\frac{ \log  \binom{n}{K} -O(K^{\alpha}\log n)}{\binom{K}{2}} }\right)\\
&=O\left( \frac{K^{\alpha-2} \log n}{\sqrt{ \frac{ \log  \binom{n}{K} }{\binom{K}{2}}}}\right), \text{ using } \sqrt{a}-\sqrt{b}=\left(a-b\right)/\left(\sqrt{a}+\sqrt{b}\right)\\
&=O\left( K^{\alpha-\frac{3}{2}} \sqrt{\log n} \right).
\end{align*} The proof of the Lemma is complete.

\end{proof}

The lemma below generalizes Lemma 4 from \cite{Bollobas18}.
\begin{lemma}\label{flat}
Let $\gamma,\delta \in (0,1)$. Suppose $G'$ is an \ER $G\left(K,\frac{1}{2}\right)$ conditioned on having $\ceil*{\gamma \binom{K}{2}}$ edges. Then $G'$ is $(\gamma,\delta)$-flat (defined in Definition \ref{dfn:flat}) w.h.p. as $K \rightarrow + \infty$.
\end{lemma}
\begin{proof}  For any $C \subset V(G)$, let $e\left(C\right):=|E[C]|/\binom{|C|}{2}.$

Consider any $2 \leq \ell \leq K-1$ and any $C \subset V(G)$ with $|C|=\ell$. By identical reasoning we have from equation (4), page 6 in \cite{Bollobas18} that for any $r>0$, $$\mathbb{P}\left(|E\left[C\right]| \geq \gamma \binom{\ell}{2}+r \right) \leq \exp \left(-\frac{r(r-1)}{2\gamma \min(\binom{K}{2}-\binom{\ell}{2},\binom{\ell}{2})}\right).$$Therefore by union bound, \begin{align*} &\binom{K}{\ell}\mathbb{P}\left(|E\left[C\right]| \geq \gamma \binom{\ell}{2}+D_K\left(\ell,\delta\right)\right) \\
&\leq \binom{K}{\ell} \sum_{r=D_K\left(\ell,\delta\right)}^{\binom{\ell}{2}} \exp \left(-\frac{r(r-1)}{2\gamma \min(\binom{K}{2}-\binom{\ell}{2},\binom{\ell}{2})}\right) \\
& \leq \exp\left( \log \binom{K}{\ell}+ \log \binom{\ell}{2}  - \frac{(D_K\left(\ell,\delta\right)-1)^2}{2\gamma \min(\binom{K}{2}-\binom{\ell}{2},\binom{\ell}{2})}\right)\\
& \leq \exp\left( \log \binom{K}{\ell}+ 2\log K  - \frac{(D_K\left(\ell,\delta\right)-1)^2}{2\gamma \min(\binom{K}{2}-\binom{\ell}{2},\binom{\ell}{2})}\right).
\end{align*}Therefore plugging in the value for $D_K\left(\ell,\delta\right)$ we conclude that for $\ell<\frac{2K}{3}$, $$\binom{K}{\ell}\mathbb{P}\left(|E\left[C\right]| \geq \gamma \binom{\ell}{2}+D_K\left(\ell,\delta\right)\right) \leq \exp(-(1+\delta) \log \binom{K}{\ell} )$$ and for $\ell \geq \frac{2K}{3}$, $$\binom{K}{\ell}\mathbb{P}\left(|E\left[C\right]| \geq \gamma \binom{\ell}{2}+D_K\left(\ell,\delta\right) \right) \leq \exp(-\delta \log \binom{K}{\ell} ).$$Using union bound and the above two inequalities we have that  $G'$ is  not $(\gamma,\delta)$-flat with probability at most
\begin{align} & \sum_{\ell=2}^{K-1} \binom{K}{\ell}\mathbb{P}\left(|E\left[C\right]| \geq \gamma \binom{\ell}{2}+D_K\left(\ell,\delta\right)\right)  \notag \\
& \leq \sum_{\ell =1}^{\floor{\frac{2K}{3}}} \binom{K}{\ell}^{-1-\delta}+\sum_{ \ell=\ceil{\frac{2K}{3}}}^{K-1} \binom{K}{\ell}^{-\delta} \label{Bound} \end{align}Using now that for $\ell$ satisfying $\frac{2k}{3} \leq \ell \leq K-K^{\frac{\delta}{2}}$ we have $$\binom{K}{\ell}=\binom{K}{K-\ell} \geq \left( \frac{K}{K-\ell}\right)^{K-\ell} \geq 3^{K-\ell} \geq 3^{K^{\frac{\delta}{2}}}$$and otherwise if $\ell  \leq \frac{2K}{3}$, $\binom{K}{\ell} \geq \binom{K}{1}=K$ the right hand side of (\ref{Bound}) is at most
\begin{align*} 
&K K^{-1-\delta}+K3^{- K^{\frac{\delta}{2}}}+K^{\frac{\delta}{2}}K^{-\delta} \leq K^{-\delta}+K3^{- K^{\frac{\delta}{2}}}+K^{-\frac{\delta}{2}}
\end{align*}which is $o\left(1\right)$. The proof of the Lemma is complete.

\end{proof}

Assume $G \sim G\left(n,\frac{1}{2}\right)$ and $K \leq n$. For $2\leq \ell \leq K-1, 0 \leq L \leq \binom{\ell}{2}$ and $A,B \subset V(G)$ with $|A|=K,|B|=K$ and $|A \cap B|=\ell$ let \begin{equation}\label{dfnG} g_{\ell}(L):=\mathbb{P} \left( |\mathrm{E}[A]|=|\mathrm{E}[B]|=\ceil{\gamma \binom{K}{2}},|\mathrm{E}[|A \cap B|]|=L\right).\end{equation} 

\begin{lemma}\label{lam} 
For  $2\leq \ell \leq K-1$ and $\gamma \in (\frac{1}{2},1)$ let $\lambda:=\exp \left( \frac{2\gamma-1}{1-\gamma}+\frac{1}{\gamma \left[ \binom{K}{2}-\binom{\ell}{2}\right]} \right)$. Then
\begin{itemize}
\item[(1)] for any $r \geq 0$, $$\frac{ g_{\ell} ( \ceil*{\gamma \binom{\ell}{2}}+r)}{\mathbb{P} \left(|\mathrm{E}[A]|=\ceil*{\gamma \binom{K}{2}}\right)^2} \leq \lambda^r \exp \left( \binom{\ell}{2} r(\gamma,\frac{1}{2})+O(1) \right).$$
\item[(2)] for any $r \leq  0$, $$\frac{ g_{\ell} ( \ceil*{\gamma \binom{\ell}{2}}+r)}{\mathbb{P} \left(|\mathrm{E}[A]|=\ceil*{\gamma \binom{K}{2}}\right)^2} \leq \exp \left( \binom{\ell}{2} r(\gamma,\frac{1}{2})+O(1) \right).$$
\end{itemize}
\end{lemma}

\begin{proof}This follows from the proof of \cite[Lemma 6]{Bollobas18} for $p=\frac{1}{2}$ and minor adjustment in the choice of $\lambda$. The minor adjustment is justified by the second displayed equation on Page 9 in the aforementioned paper. In that equation if we apply the elementary inequality $1+x \leq e^x$ once for $x=\frac{2 \gamma -1}{1-\gamma}$ and once for $x=\frac{1}{\gamma  \left[ \binom{K}{2}-\binom{\ell}{2}\right]}$ we obtain the new choice of $\lambda$. With this modification, following the proof of \cite[Lemma 6]{Bollobas18}, \textit{mutatis mutandis}, gives the Lemma.
\end{proof}

%

\subsection{Proof of the Lower Bound}
\label{low}
We turn now to the lower bound of (\ref{DenseSub}).

For $\gamma \in (\frac{1}{2},1)$ we again define $Z_{K,\gamma}$ as in (\ref{dfn:z}). Furthermore for any $\delta>0$,  let $\hat{Z}_{K,\gamma,\delta}$ the random variable that counts the number of $(\gamma,\delta)$-flat $K$-vertex subgraphs of $G$;
\begin{equation} \label{dfn:zpr} 
\hat{Z}_{K,\gamma,\delta}:=\sum_{A \subset V(G): |A|=K} 1\left( A \text{ is } (\gamma,\delta)\text{-flat}\right).
\end{equation} Notice that clearly by definition of $(\gamma,\delta)$-flatness we have that for any choice of $K, \gamma $ and any $\delta>0$ almost surely \begin{equation} \label{Domin} Z_{K,\gamma} \geq \hat{Z}_{K,\gamma,\delta}.\end{equation}

We establish the following proposition.
\begin{proposition}\label{second}

Suppose that $K=\Theta(n^{C})$ for some constant $C \in (0,\frac{1}{2})$. Let any $\alpha \in (0,1)$ satisfying \begin{equation} \label{alpha2} \alpha > 1-\left( \frac{5}{2}-\sqrt{6}\right)\frac{1-C}{C} \end{equation}  and set $$\gamma=h^{-1} \left(\log 2-\frac{ \log  \binom{n}{K} -K^{\alpha}  \log n }{\binom{K}{2}} \right).$$Then there exists $\delta>0$ small enough such that \begin{equation} \label{Ultimate} \frac{\mathbb{E}\left[\left(\hat{Z}_{K,\gamma,\delta}\right) ^2\right]}{\mathbb{E}\left[\hat{Z}_{K,\gamma,\delta}\right]^2}=1+o\left(1\right).\end{equation} In particular, $Z_{K,\gamma} \geq \hat{Z}_{K,\gamma,\delta} \geq 1$ with high probability as $n \rightarrow + \infty$.
\end{proposition}

Using this proposition for $\alpha:=\beta(C)+\frac{1}{2}$ and the Taylor expansion argument from Lemma \ref{lem:EntDense} we conclude the desired lower bound of Theorem \ref{denseeee}.

\begin{proof}[Proof of Proposition \ref{second}]
Notice that $\hat{Z}_{K,\gamma,\delta} \geq 1$ with high probability as $n \rightarrow + \infty$ follows by (\ref{Ultimate}) using Paley-Zigmund inequality. Thus we focus on establishing (\ref{Ultimate}).

We begin by choosing $\delta>0$ to satisfy \begin{align} \label{property0} 1-C(2\alpha-1)+4(\sqrt{(1-\alpha)}+\delta)\sqrt{C(1-C)} -2(1-C)<0. \end{align} To establish the existence of such $\delta$ notice that  (\ref{alpha2}) by elementary algebra is equivalent with \begin{align*}
C(1-\alpha) <(\sqrt{\frac{3}{2}}-1)^2 (1-C)
\end{align*}
or \begin{align*}
\sqrt{C(1-\alpha)}+\sqrt{1-C} <\sqrt{\frac{3}{2} (1-C)}
\end{align*}which by squaring both sides yields   \begin{align*}
C(1-\alpha)+1-C+2\sqrt{(1-\alpha)}\sqrt{C(1-C)} <\frac{3}{2} (1-C)
\end{align*}
or equivalently by multiplying both sides by $2$ and rearranging
\begin{align*}  1-C(2\alpha-1)+4\sqrt{(1-\alpha)}\sqrt{C(1-C)} -2(1-C) < 0. \end{align*} Now, since $C \in (0,1)$, the last inequality implies the existence of some sufficiently small $\delta>0$ such that (\ref{property0}) holds.
 
For an arbitrary $K$-vertex subset $A \subseteq V(G)$ and linearity of expectation, (\ref{dfn:zpr}) gives \begin{align} 
\mathbb{E}[\hat{Z}_{K,\gamma,\delta}]& =\binom{n}{K} \mathbb{P}\left( A \text{ is } (\gamma,\delta)\text{-flat}\right) \notag \\
& = \left(1-o\left(1\right) \right) \binom{n}{K}\mathbb{P}\left( |\mathrm{E}[A]| = \ceil*{\gamma \binom{K}{2}}  \right) \text{, using Lemma \ref{flat}} \notag \\
& =\binom{n}{K}\exp\left(-\binom{K}{2} r(\gamma,\frac{1}{2})-\frac{1}{2} \log \binom{K}{2}+O\left(1\right)\right) , \text{ using Lemma \ref{Bin} }\notag\\
& = \exp \left(\log \binom{n}{K} -\binom{K}{2} r(\gamma,\frac{1}{2})-\frac{1}{2} \log K+O(1)\right). \label{FMC}
\end{align}Using that for our choice of $\gamma$, $$r(\gamma,\frac{1}{2}) =\frac{ \log  \binom{n}{K} -K^{\alpha}  \log n }{\binom{K}{2}}$$ we conclude that, \begin{align} \label{FMB} \mathbb{E}[\hat{Z}_{K,\gamma,\delta}] = \exp \left( K^{\alpha} \log n-\frac{1}{2} \log K +O(1) \right)=\exp \left(\Omega \left(K^{\alpha} \log n \right) \right),\end{align} since $K^{\alpha} =\Theta(n^{C \alpha})=\omega(1)$.

We now proceed to the second moment calculation. 
For $A \subset V(G)$ with $|A|=K$ define the events $E_A:=\{A \text{ is } (\gamma,\delta)\text{-flat}\}$ and $E'_A:=\{|E[A]|=\ceil*{\gamma \binom{K}{2}}\}$. Note $$\hat{Z}_{K,\gamma,\delta}=\sum_{A \subset V(G),|A|=K} 1(E_A).$$
For $\ell=|A\cap B|$ we have via standard expansion,
\begin{align*}
&\frac{\mathbb{E}[(\hat{Z}_{K,\gamma,\delta})^2]}{\mathbb{E}[\hat{Z}_{K,\gamma,\delta}]^2}-1\\
&=\frac{\mathbb{E}[(\hat{Z}_{K,\gamma,\delta})^2]-\mathbb{E}[\hat{Z}_{K,\gamma,\delta}]^2}{\mathbb{E}[\hat{Z}_{K,\gamma,\delta}]^2}\\&=\sum_{\ell=2}^{K} \binom{K}{\ell}\binom{n-K}{K-\ell} \binom{n}{K}^{-1} \frac{ \mathbb{P}\left(E_A \cap E_B\right)-\mathbb{P}\left(E_A\right)^2}{\mathbb{P}\left(E_A\right)^2}\\
&\sum_{\ell=2}^{K-1} \binom{K}{\ell}\binom{n-K}{K-\ell} \binom{n}{K}^{-1} \frac{ \mathbb{P}\left(E_A \cap E_B\right)-\mathbb{P}\left(E_A\right)^2}{\mathbb{P}\left(E_A\right)^2}+\frac{1-\mathbb{P}\left(E_A\right)}{\mathbb{E}[\hat{Z}_{K,\gamma,\delta}]},\\
\end{align*}which since $\mathbb{E}[\hat{Z}_{k,\gamma,\delta}]=\omega(1)$, using (\ref{FMB}),  it equals to
\begin{align*}
&\sum_{\ell=2}^{K-1} \binom{K}{\ell}\binom{n-K}{K-\ell} \binom{n}{K}^{-1} \frac{ \mathbb{P}\left(E_A \cap E_B\right)-\mathbb{P}\left(E_A\right)^2}{\mathbb{P}\left(E_A\right)^2}+o(1)\\
&\leq \sum_{\ell=2}^{K-1} \binom{K}{\ell}\binom{n-K}{K-\ell} \binom{n}{K}^{-1} \frac{ \mathbb{P}\left(E_A \cap E_B\right)}{\mathbb{P}\left(E_A\right)^2}+o(1)\\
&\leq \left(1+o\left(1\right) \right) \sum_{\ell=2}^{K-1} \binom{K}{\ell}\binom{n-K}{K-\ell} \binom{n}{K}^{-1} \frac{ \mathbb{P}\left(E_A \cap E_B\right)}{\mathbb{P}\left(E'_A\right)^2}+o(1),\text{ from Lemma \ref{flat}}.\\
\end{align*}Now for fixed $\ell \in \{2,3,\ldots,K-1\}$ and $(\gamma,\delta)$-flat $K$-subgraphs $A,B$ with $\ell=|A\cap B|$ we have from the definition of $(\gamma,\delta)$-flatness that the graph induced by $A \cap B$ contains at most $\ceil*{\gamma \binom{k}{2}}+D_K(\ell,\delta)$ edges. In particular,
\begin{align*}
\frac{\mathbb{P}\left(E_A \cap E_B\right)}{\mathbb{P}\left(E'_A\right)^2}&=\sum_{L=0}^{\ceil*{\gamma \binom{\ell}{2}}+D_K(\ell,\delta)}\frac{\mathbb{P}\left(E_A \cap E_B, E[A\cap B]=L\right)}{\mathbb{P}\left(E'_A\right)^2}\\
&\leq \sum_{L=0}^{\ceil*{\gamma \binom{\ell}{2}}+D_K(\ell,\delta)}\frac{\mathbb{P}\left(E'_A \cap E'_B, E[A\cap B]=L\right)}{\mathbb{P}\left(E'_A\right)^2},\text{ using that  } E_A \subseteq E'_A, E_B \subseteq E'_B\\
&=\sum_{L=0}^{\ceil*{\gamma \binom{\ell}{2}}+D_K(\ell,\delta)} \frac{g_{\ell}(L)}{\mathbb{P}\left(E'_A\right)^2},\text{ using notation (\ref{dfnG}}) \\
& \leq \sum_{L=0}^{\ceil*{\gamma \binom{\ell}{2}}+D_K(\ell,\delta)} \lambda^{D_K(\ell,\delta)} \exp \left(\binom{\ell}{2} r(\gamma,\frac{1}{2})+O(1) \right), \text{ using Lemma \ref{lam} and }\lambda \geq 1\\
& \leq \binom{\ell}{2} \lambda^{D_K(\ell,\delta)} \exp \left(\binom{\ell}{2} r(\gamma,\frac{1}{2})+O(1) \right)  \\
& = \exp \left(D_K(\ell,\delta) \log \lambda+\binom{\ell}{2} r(\gamma,\frac{1}{2})+O(\log \ell) \right).
\end{align*}
Therefore  we conclude \begin{align*} \label{main} \frac{\mathbb{E}[(\hat{Z}_{K,\gamma,\delta})^2]}{\mathbb{E}[\hat{Z}_{K,\gamma,\delta}]^2} \leq 1&+\sum_{\ell=2}^{K-1} \binom{K}{\ell}\binom{n-K}{K-\ell} \binom{n}{K}^{-1} \exp \left( D_K(\ell,\delta) \log \lambda+\binom{\ell}{2} r(\gamma,\frac{1}{2})+O(\log \ell)\right)\\
&+o(1).\end{align*} 

We proceed from now on in two steps to complete the proof. First we show that for some sufficiently small constant $\delta_1>0$,  \begin{equation} \label{main1} \sum_{\ell=2}^{\floor*{\delta_1 K}} \binom{K}{\ell}\binom{n-K}{K-\ell} \binom{n}{K}^{-1} \exp \left( D_K(\ell,\delta) \log \lambda+\binom{\ell}{2} r(\gamma,\frac{1}{2})+O(\log \ell)\right)=o(1).\end{equation}  In the second step we show that for the constant $\delta_1>0$ chosen in the first step, \begin{equation} \label{main2} \sum_{\ell=\ceil*{\delta_1 K}}^{K-1} \binom{K}{\ell}\binom{n-K}{K-\ell} \binom{n}{K}^{-1} \exp \left( D_K(\ell,\delta) \log \lambda+\binom{\ell}{2} r(\gamma,\frac{1}{2})+O\left(\log n\right)\right)=o(1).\end{equation}Note here that for these values of $\ell$ for the second step  we have replaced $O(\log \ell)$ with the equivalent bound $O\left(\log K \right)=O\left(\log n \right)$ since $K=\Theta(n^C)$ for $C>0$. 

\paragraph{First Step, proof of (\ref{main1}):}

For the combinatorial term we use a simple inequality derived from Stirling's approximation (see e.g. page 11 in \cite{Bollobas18}), \begin{align} \label{simpleBol} \binom{K}{\ell}\binom{n-K}{K-\ell} \binom{n}{K}^{-1} \leq \left(1+o(1)\right) \left( \frac{K^2}{n} \right)^{\ell}.\end{align}We now bound the terms in the exponent.  Plugging in the value of $\lambda$ from Lemma \ref{lam} we have $$ D_K(\ell,\delta) \log \lambda= \frac{D_K(\ell,\delta)}{\gamma \left[\binom{K}{2}-\binom{\ell}{2}\right]}+\frac{2 \gamma -1}{1-\gamma}D_K(\ell,\delta).$$By the definition of $D_K(\ell,\delta)$ (\ref{Dk}) we have \begin{align*}
\frac{D_K(\ell,\delta)}{\gamma \left[\binom{K}{2}-\binom{\ell}{2}\right]}=O\left(\sqrt{\frac{ \log \binom{K}{\ell}+ \log K}{\binom{K}{2}-\binom{\ell}{2}}}\right) \leq O\left(\sqrt{ \frac{K}{\binom{K}{2}-\binom{K-1}{2}}}\right)=O\left(1 \right),
\end{align*}since $\ell \leq \delta_1K \leq K-1$, assuming $\delta_1<1$. From Lemma \ref{gamma} we have $\gamma=\frac{1}{2}+\Theta \left(\sqrt{ \frac{\log n}{K}} \right)$. Furthermore, by  (\ref{Dk}), $K=\Theta\left(n^C\right)$ and $\binom{K}{\ell} \leq K^{\ell}$ we have $$D_K(\ell,\delta)=O\left(\sqrt{\ell^2 \left(\log \binom{K}{\ell}+ \log K \right)}\right)=O\left(\sqrt{\ell^3 \log n }\right).$$ Combining the two last equalities we conclude \begin{align} \label{St2} \frac{2 \gamma -1}{1-\gamma}D_K(\ell,\delta)=O \left( \frac{ \ell^{3/2}\log n}{\sqrt{K}} \right).\end{align}Finally, again by Lemma \ref{gamma}, $r(\gamma,\frac{1}{2})=\Theta( \frac{ \log n}{K} ) $ and therefore
\begin{align} \label{St3} \binom{\ell}{2} r(\gamma,\frac{1}{2})=O \left(   \frac{\ell^2 \log n}{K} \right).\end{align}
Combining  (\ref{simpleBol}), (\ref{St2}) and (\ref{St3}) we conclude that for any $\delta_1>0$, supposing $\ell<\delta_1 K$ we get

\begin{align}
&\binom{K}{\ell}\binom{n-K}{K-\ell} \binom{n}{K}^{-1} \exp \left( D_K(\ell,\delta) \log \lambda+\binom{\ell}{2} r(\gamma,\frac{1}{2})+O\left(\log \ell \right)\right) \label{firstbound1:OGP} \\
&=\exp \left[ -\ell\log \left(\frac{n}{K^2} \right) +O\left(\frac{ \ell^{3/2} \log n}{\sqrt{K}} \right)+O\left(  \frac{\ell^2 \log n}{K} \right)+O\left(\log \ell \right)         \right] \notag\\
&=\exp \left[ -\ell \log n \left( 1-2 C-O\left(\sqrt{\frac{ \ell}{K}}\right)-O\left(\frac{\ell}{K}\right) -O\left(\frac{\log \ell }{\ell \log n}\right) \right) \right],\text{ using }K=\Theta(n^{C}) \notag \\
& \leq \exp \left[ -\ell \log n \left( 1-2C-O\left(\sqrt{\delta_1}\right)-O\left(\delta_1\right) -O\left(\frac{\log \ell }{\ell \log n}\right) \right) \right], \text{ using } \ell \leq \delta_1 K \notag\\
& \leq \exp \left[ -\ell \log n \left( 1-2C-O\left(\sqrt{\delta_1}\right)-O\left(\delta_1\right) -O\left(\frac{1 }{ \log n}\right) \right) \right], \label{Bound1}
\end{align} where we have used $\log \ell \leq \ell$ for all $\ell \geq 1$. Since $C<\frac{1}{2}$ we choose $\delta_1>0$ small enough but constant such that for some $\delta_2>0$ and large enough $n$, \begin{align} \label{Bound2} 1-2 C-O\left(\sqrt{\delta_1} \right)-O\left(\delta_1 \right) -O\left(\frac{1}{ \log n} \right) >\delta_2.\end{align} Hence for this choice of constants $\delta_1,\delta_2>0$ if $\ell \leq \delta_1 K$ using (\ref{Bound1}) and (\ref{Bound2}) we conclude that the expression (\ref{firstbound1:OGP}) is at upper bounded by
\begin{align*}
 \exp \left( -\delta_2 \ell \log n  \right)=n^{-\delta_2\ell}. \end{align*}Therefore we have,
 \begin{align*} \sum_{\ell=2}^{\floor*{\delta_1 K}} \binom{K}{\ell}\binom{n-K}{K-\ell} \binom{n}{K}^{-1} \exp \left( D_K(\ell,\delta) \log \lambda+\binom{\ell}{2} r(\gamma,\frac{1}{2})+O\left(\log \ell \right)\right)& \leq \sum_{ \ell \geq 1}n^{-\delta_2\ell}=O\left(n^{-\delta_2} \right).\end{align*}This completes the proof of (\ref{main1}).

\paragraph{Second step, proof of (\ref{main2}):} For the second step we start by multiplying both numerator and denominator of the left hand side of (\ref{main2}) with the two sides of (\ref{FMC}); $\mathbb{E}\left[\hat{Z}_{K,\gamma,\delta}\right]=\binom{n}{K}\exp \left(-\binom{K}{2} r(\gamma,\frac{1}{2})+O\left(\log n \right)\right),$ to get that it suffices to show $$\frac{1}{\mathbb{E}[\hat{Z}_{K,\gamma,\delta}] }\sum_{\ell=\ceil*{\delta_1 K}}^{K-1} \binom{K}{\ell}\binom{n-K}{K-\ell} \exp \left( D_K(\ell,\delta) \log \lambda-\left( \binom{K}{2}-\binom{\ell}{2}\right) r(\gamma,\frac{1}{2})+O\left(\log n \right)\right)=o(1).$$Since by equation (\ref{FMB}) we have $\mathbb{E}[\hat{Z}_{k,\gamma,\delta}] \geq \exp\left(D_0K^{\alpha}\log n\right)$ for some universal constant $D_0>0$ and $K=\omega(1)$ it suffices that $$\sum_{\ell=\ceil*{\delta_1 K}}^{K-1} \binom{K}{\ell}\binom{n-K}{K-\ell}  \exp \left( D_K(\ell,\delta) \log \lambda-\left( \binom{K}{2}-\binom{\ell}{2}\right) r(\gamma,\frac{1}{2})-D_0K^{\alpha} \log n \right)=o(1).$$ Plugging in the value of $\lambda$ we have
\begin{align*}
&\sum_{\ell = \ceil*{\delta_1 K}}^{K-1} \binom{K}{\ell} \binom{n-K}{K-\ell}  \exp \left( D_K(\ell,\delta) \log \lambda-\left( \binom{K}{2}-\binom{\ell}{2}\right) r(\gamma,\frac{1}{2})- D_0 K^{\alpha} \log n \right) 
\end{align*} which is of the order
\begin{align*}
\sum_{\ell =\ceil*{\delta_1 K}}^{K-1} \binom{K}{\ell} \binom{n-K}{K-\ell}  \exp\left[ \frac{D_K(\ell,\delta)}{\gamma \left[ \binom{K}{2}-\binom{\ell}{2} \right]}+\frac{2\gamma -1}{1-\gamma}D_K(\ell,\delta)-\left(\binom{K}{2}-\binom{\ell}{2}\right) r(\gamma,\frac{1}{2})-D_0K^{\alpha}\sqrt{ \log n}\right]
\end{align*}By (\ref{Dk}) we have \begin{align}\label{TRIV1}
\frac{D_K(\ell,\delta)}{\gamma\left[\binom{K}{2}-\binom{\ell}{2}\right] }=O\left(\sqrt{\frac{ \log \binom{K}{\ell}+ \log K}{\binom{K}{2}-\binom{\ell}{2}}} \right) \leq O\left(\sqrt{ \frac{K}{\binom{K}{2}-\binom{K-1}{2}}}\right)=O\left(1 \right),
\end{align}since $\ell \leq K-1$. Furthermore by Lemma \ref{gamma},
\begin{align}
\left(\binom{K}{2}-\binom{\ell}{2}\right)r(\gamma, \frac{1}{2})& \geq  \left(\binom{K}{2}-\binom{\ell}{2}\right)\left(  1-o(1)\right)\frac{2 \log (\frac{n}{K})}{K} \notag\\
&=\left(\binom{K}{2}-\binom{\ell}{2}\right) \frac{ 2 \log (\frac{n}{K})}{K}-o\left((K-\ell) \log n\right)  \label{TRIV2}
\end{align}Hence, combining (\ref{TRIV1}) and (\ref{TRIV2}) we conclude that it suffices to show\begin{align}\label{ULTIMATE}
 \sum_{\ell =\ceil*{\delta_1 K}}^{K-1} \exp\left[ F(\ell) \right]=o(1)
\end{align}where $F(\ell)$ equals
 \begin{align}\label{Fell}
& \log ( \binom{K}{\ell}\binom{n-K}{K-\ell})+\frac{2\gamma -1}{1-\gamma}D_K(\ell,\delta)-\left(\binom{K}{2}-\binom{\ell}{2}\right) \frac{ 2 \log (\frac{n}{K})}{K}-D_0 K^{\alpha} \log n+o\left((K-\ell)\log n\right).
\end{align}
Now we separate three cases to study $F(\ell)$.

\paragraph{Case 1 (large values of $\ell$): } We assume $K-1 \geq \ell \geq K-c_1 K^{\alpha} \log n,$ where $c_1>0$ is a universal constant defined below.

 In this case we bound the combinatorial term using $\binom{K}{\ell}\leq K^{K-\ell} $ and $\binom{n-K}{K-\ell} \leq n^{K-\ell} $ to conclude \begin{align} \label{COMBI1} \binom{K}{\ell}\binom{n-K}{K-\ell} \leq K^{K-\ell} n^{K-\ell} =\exp\left( O\left[ \left(K-\ell \right) \log n \right]\right) .\end{align}
Furthermore,
\begin{align}
\frac{2\gamma-1}{1-\gamma} D_K(\ell,\delta) &= O\left( (2 \gamma -1) \sqrt{ \left(\binom{K}{2}-\binom{\ell}{2}\right) \left(\log \binom{K}{\ell}+ \log K\right)}\right), \text{using (\ref{Dk}) } \notag \\
&= O\left( (2 \gamma -1) \sqrt{ \left(K-\ell\right)\left(K+\ell-1\right)  \left(\log \binom{K}{\ell}+ \log K\right)} \right) \notag \\
& \leq O \left(\sqrt{\frac{ \log n}{K}} \sqrt{ (K-\ell)K \cdot (K-\ell)  \log K} \right), \text{ using Lemma \ref{gamma} and } \binom{K}{\ell} \leq K^{K-\ell} \notag \\
&\leq O\left[ \left(K-\ell\right) \log n\right], \label{BO2}
\end{align}
Therefore using (\ref{COMBI1}) and (\ref{BO2}) for $\ell$ with $K-1 \geq \ell \geq K-c_1 K^{\alpha}\log n$ we have 
\begin{align*}
F(\ell) & \leq O\left(\left(K-\ell \right)\log n\right)-D_0 K^{\alpha}\log n\\
&  \leq C\left(K-\ell \right)\log n-D_0 K^{\alpha}\log n,\text{ for some universal constant } C>0\\
&\leq (Cc_1-D_0 )K^{\alpha}\log n \text{,       by the assumption on } \ell\\
& \leq -\frac{D_0}{2}K^{\alpha}\log n, \text{ by choosing }c_1:=D_0/2C,
\end{align*}which gives

\begin{equation}\label{F1} \sum_{\ell=\ceil*{K-c_1 K^{\alpha}\log n}}^{K-1} \exp\left[F(\ell)\right] =O\left(\exp\left( \log K-\frac{D_0}{2}K^{\alpha}\log n \right)\right) =o\left(1\right)\end{equation}where the last equality is because $K=\omega(1)$.

\paragraph{Case 2 (moderate values of $\ell$) :} $ (1-\delta' ) K \leq \ell \leq K-c_1 K^{\alpha} \log n$, where $c_1>0$ is defined in Case 1 and $\frac{1}{3}>\delta'>0$ is a sufficiently small but constant positive number such that \begin{align} \label{property} 1-C(2\alpha-1)+4(\sqrt{(1-\alpha)}+\delta)\sqrt{C(1-C)} -2(1-\delta')(1-C)<-\delta'. \end{align} Note that such a $\frac{1}{3}>\delta'>0$ exists because of our choice of $\delta$ satisfying (\ref{property0}) and because $C<1$.

We start with the standard $\binom{K}{\ell} \leq \left(\frac{Ke}{K-\ell} \right)^{K-\ell}$ and $\binom{n-K}{K-\ell} \leq \left(\frac{(n-K)e}{K-\ell} \right)^{K-\ell}$ to conclude
\begin{align}
\log \left( \binom{K}{\ell}\binom{n-K}{K-\ell} \right) &\leq \left(K-\ell \right) \log \left(\frac{nKe^2}{(K-\ell)^2}\right) \label{COMBI3}\\
& \leq \left(1-C(2\alpha-1)+o(1)\right)\left(K-\ell\right) \log n, \label{COMBI2}
\end{align}where for the last step we used $K-\ell \geq \Omega\left(K^{\alpha}\right)$ and that $K=\Theta(n^{C})$.
Furthermore for this values of $\ell$ we have $\ell >\frac{2K}{3}$. Therefore from  (\ref{Dk}),
\begin{align}
D_K(\ell,\delta) &\leq (1+\delta)\sqrt{2\gamma \left( \binom{K}{2}-\binom{\ell}{2}\right) \log \left(2K\binom{k}{\ell}\right)} \notag\\
&\leq (1+\delta+o(1))\sqrt{ \left( \binom{K}{2}-\binom{\ell}{2}\right) \log \left(2K\binom{K}{\ell}\right)}, \text{ using Lemma \ref{gamma}} \notag\\
& \leq \left(1+\delta+o(1)\right)) \sqrt{\left( \binom{K}{2}-\binom{\ell}{2}\right)\left((K-\ell) \log (\frac{Ke}{K-\ell})+2\log K \right)}  \notag \\
& \leq \left(1+\delta+o(1) \right) \left(K-\ell \right) \sqrt{K \log \left(O\left(K^{1-\alpha}\right)\right)} , \text{since }   \binom{K}{2}-\binom{\ell}{2} \leq K(K-\ell), K-\ell \geq \Omega(K^{\alpha}) \notag \\
& \leq \left(\sqrt{1-\alpha}+\delta+o(1)\right) \left(K-\ell \right) \sqrt{K\log K}  \label{BO22}
\end{align}From Lemma \ref{gamma} we have
\begin{align*}
\frac{2 \gamma-1}{1-\gamma} =\left(4+o(1)\right) \sqrt{\frac{ \log \frac{n}{K}}{K}}.
\end{align*}Hence combining it with (\ref{BO22}),\begin{align}
\frac{2\gamma -1}{1-\gamma}D_K(\ell,\delta) & \leq \left(\sqrt{1-\alpha}+\delta+o(1) \right)4(K-\ell) \sqrt{\frac{ \log \frac{n}{K}}{K}}\sqrt{K \log K} \notag \\
&=4\left(\sqrt{(1-\alpha)}+\delta+o(1) \right)  (K-\ell)\sqrt{ \log (\frac{n}{K}) \log K} \label{BO3}
\end{align}Now by dropping the term $-D_0K^{\alpha} \log n<0$, $F(\ell)$ is at most
\begin{align*}\log ( \binom{K}{\ell}\binom{n-K}{K-\ell})+\frac{2\gamma -1}{1-\gamma}D_K(\ell,\delta)-\left(\binom{K}{2}-\binom{\ell}{2}\right) \frac{ 2 \log (\frac{n}{K})}{K}+o\left((K-\ell)\log n\right).
\end{align*}which using (\ref{COMBI2}), (\ref{BO3}) is at most $1+o\left(1\right)$ times
\begin{align*}
& \left(K-\ell\right)\left[ \left(1-C(2\alpha-1)\right)\log n+4\left(\sqrt{(1-\alpha)}+\delta \right)  \sqrt{ \log (\frac{n}{K}) \log K}   - \frac{2\left(\binom{K}{2}-\binom{\ell}{2}\right)\log \frac{n}{K} }{K(K-\ell)} +o( \log n)\right]\\
&\leq \left(K-\ell\right)\left[ \left(1-C(2\alpha-1)\right)\log n+4\left(\sqrt{(1-\alpha)}+\delta \right)  \sqrt{ \log (\frac{n}{K}) \log K}   -2 \left(1-\delta'\right)\log \frac{n}{K} +o( \log n)\right]\\
&=\left(K-\ell\right)\log n\left[ 1-C(2\alpha-1)+4\left(\sqrt{(1-\alpha)}+\delta \right) \frac{ \sqrt{ \log (\frac{n}{K}) \log K}}{\log n}   -2 \left(1-\delta'\right)\frac{\log \frac{n}{K}}{\log n} +o(1)\right], \\
\end{align*}where for the last inequality we used that for $\ell \geq (1-\delta')k$, $\binom{k}{2}-\binom{\ell}{2} \geq (1-\delta'-o\left(1\right))k(k-\ell)$. Using that $K=\Theta\left(n^C\right)$ we conclude,
\begin{align*}
F(\ell)\leq \left[ \left(1-C(2\alpha-1)\right)+4\left(\sqrt{(1-\alpha)}+\delta \right)\sqrt{C(1-C)}   -2 \left(1-\delta'\right)\left(1-C\right)  +o(1)\right]\left(K-\ell\right)\log n.
 \end{align*}From (\ref{property}) we know that for large $n$ $$\left(1-C(2\alpha-1)\right)+4\left(\sqrt{(1-\alpha)}+\delta \right)\sqrt{C(1-C)}   -2 \left(1-\delta'\right)\left(1-C \right)+o(1)<-\delta'.$$Therefore we conclude for all $\ell$ with $ (1-\delta' ) K \leq \ell \leq K-c_1 K^{\alpha} \log n$

$$F(\ell) \leq -\delta'  \left(K-\ell \right) \log n \leq -\Omega \left(K^{\alpha} \left( \log n \right)^2 \right).$$Hence,
\begin{equation}\label{F2}\sum_{\ell=\ceil*{(1-\delta')K}}^{\floor*{K-c_1 K^{\alpha}\log n}} \exp\left[F(\ell)\right] =O\left(K\exp\left(-\Omega(K^{\alpha} (\log n)^2\right)\right) =O\left(\exp\left( \log n-\Omega(K^{\alpha} (\log n)^2\right)\right)=o(1),\end{equation} where the last equality is because $K=\Theta(n^C)$ for $C>0$.

\paragraph{Case 3 (small values of $\ell$) :} $\delta_1 K \leq \ell \leq (1-\delta')K$ where $\delta'$ is defined in Case 2 and $\delta_1$ in Part 1.

Similar to (\ref{COMBI3}) we have
\begin{align}
\log \left( \binom{K}{\ell}\binom{n-K}{K-\ell} \right) &\leq \left(K-\ell \right) \log \left(\frac{nKe^2}{(K-\ell)^2} \right) \notag \\
& \leq \left(1+o(1) \right) \left(K-\ell \right) \log \frac{n}{K}, \label{COMBI4}\end{align} where we have used for the last inequality that $\ell=\Theta(K)$.

Furthermore using (\ref{Dk}) and Lemma \ref{gamma} we have
\begin{align*}
\frac{2 \gamma -1}{1-\gamma} D_K(\ell,\delta) &\leq  O\left(\sqrt{ \frac{ \log n}{K}}\sqrt{\left( \binom{K}{2}-\binom{\ell}{2}\right)\left(\log \binom{K}{\ell}+\log K \right)}\right)\\
& \leq O\left(\sqrt{ \frac{ \log n}{K}} \sqrt{\left( \binom{K}{2}-\binom{\ell}{2}\right)\left((K-\ell) \log (\frac{Ke}{K-\ell})+\log K \right)}\right) \\
& \leq  O\left(\sqrt{ \frac{ \log n}{K}} (K-\ell) \sqrt{K} \right), \text{ using } K-\ell =\Theta(K) \text{ and }  \binom{K}{2}-\binom{\ell}{2} \leq K(K-\ell)\\
&= o\left((K-\ell)\log n \right).
\end{align*}
Combining it with (\ref{COMBI4}) we have that $F(\ell)$ is at most
\begin{align} &\left(1+o(1)\right)\left[ (K-\ell) \log (\frac{n}{K})+o((K-\ell) \log n)-\frac{2\left(\binom{K}{2}-\binom{\ell}{2}\right)}{K} (\log \frac{n}{K})\right]  \notag\\
 &\leq^{(*)} \left(1+o(1)\right)(K-\ell)\left[ \log (\frac{n}{K})+o( \log n)-2\left(\frac{1+\delta_1}{2}\right)(\log \frac{n}{K})\right] \notag \\
& \leq \left(K-\ell \right)\log n \left(  1-C-2\left(\frac{1+\delta_1}{2}\right)\left(1-C\right)+o(1)\right),\text{ using } K=\Theta(n^c)  \notag \\
& = (K-\ell)\log n \left(-\delta_1\left(1-C\right)+o(1) \right),  \label{binomOGPeq2}
\end{align}where to derive (*) we use that for $\ell \geq \delta_1K$, $$ \binom{K}{2}-\binom{\ell}{2} \leq (\frac{1+\delta_1+o(1)}{2})K(K-\ell).$$ Since $\delta_1\left(1-C\right)>0$ we conclude from (\ref{binomOGPeq2}) that for all $\ell$ with $\delta_1 K \leq \ell \leq (1-\delta')K$ and large enough $n$,$$F(\ell) \leq -\Theta\left( K\log n \right)$$Hence,
\begin{equation}\label{F3}\sum_{\ell=\ceil*{\delta_1K}}^{\floor*{(1-\delta') K}} \exp\left[F(\ell)\right]   \leq O\left(K\exp\left[(\log n-\Theta( K \log n)\right)\right) \leq O\left(\exp\left( \log n-\Theta( K \log n)\right)\right) =o\left(1\right).\end{equation}

Combining (\ref{F1}), (\ref{F2}) and (\ref{F3}) we conclude the proof of (\ref{main2}). This completes the proof of the Proposition and of the Theorem.
\end{proof}
\section{Proofs for the First Moment Curve Bounds}
\label{25}

\subsection{Proof of first part of Proposition \ref{unionkkk}}
\begin{proof}[ Proof of first part of Proposition \ref{unionkkk}] If $z=\bar{k}=k$ then trivially $$d_{k,k}(G)(k)=|E[\mathcal{PC}]|=\binom{k}{2}=\Gamma_{k,k}(k)$$ almost surely.

Otherwise, we fix some $z \in \{\floor{\frac{k\bar{k}}{n}},\floor{\frac{k\bar{k}}{n}}+1,\ldots,k\}$. Since $z=\bar{k}=k$ does not hold, we have $z<\bar{k}$.
For $\gamma \in (0,1)$ we consider the counting random variable $$Z_{\gamma,z}:=|\{A \subseteq V(G) : |A|=\bar{k}, |A \cap \mathcal{PC}|=z, |E[A]| \geq \binom{z}{2}+\gamma_z \left( \binom{\bar{k}}{2}-\binom{z}{2} \right) \}|.$$ 
By Markov's inequality, $\mathbb{P}\left[ Z_{\gamma,z} \geq 1 \right] \leq \mathbb{E}\left[Z_{\gamma,z} \right]$. In particular, if for some $\gamma_z>0$ it holds \begin{equation} \label{goal?} \sum_{z=\floor{\frac{\bar{k}k}{n}}}^{k} \mathbb{E}\left[Z_{\gamma_z,z} \right]=o(1)\end{equation} we conclude using a union bound that for all $z$, $Z_{\gamma_z,z}=0$ w.h.p. as $n \rightarrow + \infty$ and in particular for all $z$, $d_{\bar{k},k}(G)(z) \leq  \binom{z}{2}+\gamma_z \left( \binom{\bar{k}}{2}-\binom{z}{2} \right)$, w.h.p. as $n \rightarrow + \infty$.
Therefore it suffices to show that for $\gamma_z :=h^{-1} \left(\log 2-\frac{ \log \left(\binom{k}{z} \binom{n-k}{\bar{k}-z} \right) }{\binom{\bar{k}}{2}-\binom{z}{2}} \right)$ (\ref{goal?}) holds. Notice that $\gamma_z$ is well-defined exactly because $z=\bar{k}=k$ does not hold. We fix this choice of $\gamma_z$ from now on.

Let us fix $z$. We start with bounding the expectation for arbitrary $\gamma>0$. By linearity of expectation we have 
\begin{align}
\mathbb{E}\left[Z_{\gamma_z,z} \right] &= \binom{k}{z} \binom{n-k}{\bar{k}-z} \mathbb{P}\left[ |\mathrm{E}[A]| \geq  \binom{z}{2}+\gamma_z \left( \binom{\bar{k}}{2}-\binom{z}{2} \right) \right],\text{for some } |A|=\bar{k}, |A \cap \mathcal{PC}|=z \notag\\
&= \binom{k}{z} \binom{n-k}{\bar{k}-z} \mathbb{P}\left[ \binom{z}{2}+\mathrm{Bin}\left(\binom{\bar{k}}{2}-\binom{z}{2} \right) \geq  \binom{z}{2}+\gamma_z \left( \binom{\bar{k}}{2}-\binom{z}{2} \right) \right] \notag\\
&= \binom{k}{z} \binom{n-k}{\bar{k}-z} \mathbb{P}\left[ \mathrm{Bin}\left(\binom{\bar{k}}{2}-\binom{z}{2} \right) \geq  \gamma_z \left( \binom{\bar{k}}{2}-\binom{z}{2} \right) \right] \label{Step1FM}
\end{align}

Using the elementary inequalities $$\binom{k}{z} \leq k^{k-z} \leq n^{\bar{k}-z}$$ and $$\binom{n-k}{\bar{k}-z} \leq n^{\bar{k}-z},$$ we conclude \begin{equation}  \label{simplee} \frac{ \log \left(\binom{k}{z} \binom{n-k}{\bar{k}-z} \right) }{\binom{\bar{k}}{2}-\binom{z}{2}}=O\left(\frac{\log n}{\bar{k}+z}\right)=o(1)\end{equation} by our assumption $\omega( \log n)=\bar{k}$.

By Lemma \ref{EntTaylor} and (\ref{simplee}) we have, $$\frac{1}{2}+\Omega \left( \sqrt{ \frac{ \log \left(\binom{k}{z} \binom{n-k}{\bar{k}-z} \right) }{\binom{\bar{k}}{2}-\binom{z}{2}} }\right) \leq \gamma_z \leq \frac{1}{2}+o(1).$$ Therefore $\lim_n \gamma_z = \frac{1}{2}$ and the elementary $ \binom{n-k}{\bar{k}-z} \geq  \left((n-k)/(\bar{k}-z)\right)^{\bar{k}-z}$ since $z<\bar{k}$, $$ \left(\gamma_z-\frac{1}{2}\right) \sqrt{\binom{\bar{k}}{2}-\binom{z}{2}}=\Omega\left( \sqrt{\log \left(\binom{k}{z} \binom{n-k}{\bar{k}-z} \right)} \right)=\Omega \left(\sqrt{\left(\bar{k}-z\right) \log \frac{n}{\bar{k}-z}} \right). $$ Hence both assumption of Lemma \ref{Bin} are satisfied  (notice that the Binomial distribution of interest is defined on population size $\binom{\bar{k}}{2}-\binom{z}{2}$) and hence (\ref{Step1FM}) implies
\begin{align}
\mathbb{E}\left[Z_{\gamma_z,z} \right] \leq \binom{k}{z} \binom{n-k}{\bar{k}-z} O\left(\exp\left(-\left(\binom{\bar{k}}{2}-\binom{z}{2} \right) r(\gamma,\frac{1}{2}) -\Omega \left(\sqrt{\left(\bar{k}-z\right) \log \frac{n}{\bar{k}-z}} \right)\right)\right) \label{FMpf1}
\end{align}
Now notice that for our choice of $\gamma_z$, $$r(\gamma,\frac{1}{2})=\log 2-h(\gamma)=\frac{ \log \left(\binom{k}{z} \binom{n-k}{\bar{k}-z} \right) }{\binom{\bar{k}}{2}-\binom{z}{2}}.$$ In particular using (\ref{FMpf1}) we conclude that for any $z$ 
\begin{align}
\mathbb{E}\left[Z_{\gamma_z,z} \right] = \exp\left(- \Omega \left(\sqrt{\left(\bar{k}-z\right) \log \frac{n}{\bar{k}-z}} \right)\right).
\end{align}Hence, 
\begin{align*}
&\sum_{z=\floor{\frac{\bar{k}k}{n}}}^{k} \mathbb{E}\left[Z_{\gamma_z,z} \right] \\
=& \sum_{z=\floor{\frac{\bar{k}k}{n}}}^{k} \exp\left(- \Omega \left(\sqrt{\left(\bar{k}-z\right) \log \frac{n}{\bar{k}-z}} \right)\right) \\
=&\text{ }\sum_{z=\min\{k,\bar{k}-\left( \left( \log n\right)^2 \right)\}}^{k}\exp\left(- \Omega \left(\sqrt{\left(\bar{k}-z\right) \log \frac{n}{\bar{k}-z}} \right)\right)\\
+&\sum_{z=\floor{ \frac{\bar{k}k}{n}}}^{\min\{k,\bar{k}-\left( \left( \log n\right)^2 \right)\}}\exp\left(- \Omega \left(\sqrt{\left(\bar{k}-z\right) \log \frac{n}{\bar{k}-z}} \right)\right)\\
 \leq &\text{ } (\log n )^2 \exp\left(- \Omega \left(\sqrt{ \log n }\right)\right)+ k \exp\left(- \Omega \left(\sqrt{\left( \log n\right)^{\frac{3}{2}}}\right) \right)\\
\leq & \text{ }\exp\left(- \Omega \left(\sqrt{ \log n }\right)\right)+ k \exp\left(- \Omega \left(\sqrt{\left( \log n\right)^{\frac{3}{2}}}\right) \right)\\
=& \text{ } o\left(1\right),
\end{align*}which is (\ref{goal?}) as we wanted.
\end{proof}
\subsection{Proof of second part of Proposition \ref{unionkkk}}
\begin{proof}[Proof of second part of Proposition \ref{unionkkk}]
The result follows from Theorem \ref{denseeee} by observing that $d_{\bar{k},k}(G)(0)$ corresponds to the number of edges of the $\bar{k}$-densest subgraph of a vanilla $G(n-k,\frac{1}{2})$ random graph.
\end{proof}

\section{Proofs for the First Moment Curve Monotonicity results}
\label{4710}

\subsection{Key lemmas}
\begin{lemma}\label{AFun} Suppose $1 \leq k \leq \bar{k} \leq n$ with $n \rightarrow +\infty$, $ \bar{k}=o\left( n \right)$ and $\epsilon \in (0,1)$ arbitrarily small constant. For $z \in [0,\left(1-\epsilon \right)k] \cap \mathbb{Z}$ let \begin{equation} \label{Adfn} A(z):=\log \left(\binom{k}{z} \binom{n-k}{\bar{k}-z} \right). \end{equation} 

Then for any $z \in [0,\left(1-\epsilon \right)k] \cap \mathbb{Z}$, 
 \begin{equation} \label{Fun2}A(z)=\Theta \left(\bar{k} \log \left(\frac{n}{\bar{k}}\right) \right).\end{equation} 
and
\begin{equation} \label{Fun3}A(z+1)-A(z)=\log \left(\frac{k\bar{k}}{\left(z+1\right)n}\right)-O\left(1 \right).\end{equation}
\end{lemma}

\begin{proof} First $$ A(z)=\log \left(\binom{k}{z}\right)+\log \left( \binom{n-k}{\bar{k}-z} \right).$$ Since, $\binom{k}{z}\leq 2^k$ we have $\log \left(\binom{k}{z}\right)=O\left(k\right)$. Hence,
\begin{equation}\label{Az} A(z)=\log \left( \binom{n-k}{\bar{k}-z} \right)+O\left(k\right).\end{equation}For any  $z \in [0,\left(1-\epsilon \right)k]$ since $k \leq \bar{k}$ we have $$\epsilon \bar{k} \leq \bar{k}-z \leq \bar{k}.$$ Hence, since for large $n$ we have $\bar{k}<\frac{n}{2}$, by standard monotonicity arguments on the binomial coefficients we have$$ \log \left( \binom{n-k}{\epsilon \bar{k}} \right) \leq \log \left( \binom{n-k}{\bar{k}-z} \right) \leq \log \left( \binom{n-k}{\bar{k}} \right)$$which using Stirling's approximation since $k \leq \bar{k} =o\left(n\right)$ and $\epsilon$ is a positive constant yields $$\log \left( \binom{n-k}{\bar{k}-z} \right)=\Theta \left(\bar{k} \log \left(\frac{n}{\bar{k}}\right) \right).$$ Combining this with (\ref{Az}) and $k \leq \bar{k} =o\left(n\right)$ we conclude (\ref{Fun2}).

For the final part, simple algebra and that $\frac{k-z}{k}=\Omega\left(1\right)$, $\frac{\bar{k}-z}{\bar{k}}=\Omega \left(1\right)$ for the $z$ of interest yields, \begin{align*} A\left(z+1\right)-A\left(z\right)&=\log \left( \frac{ \left(k-z\right)\left(\bar{k}-z\right)}{\left(z+1\right)\left(n-k-\bar{k}+z+1\right)}\right) \\
&=\log \left(\frac{k\bar{k}}{\left(z+1\right)n}\right)+\log\left(\frac{k-z}{k}\right)+\log\left(\frac{\bar{k}-z}{\bar{k}}\right)+\log\left(\frac{n}{n-k-\bar{k}+z+1}\right)\\
&=\log \left(\frac{k\bar{k}}{\left(z+1\right)n}\right)-O\left(1\right) -O\left(\frac{\bar{k}}{n}\right)\\
&=\log \left(\frac{k\bar{k}}{\left(z+1\right)n}\right)-O\left(1\right),\end{align*}which is (\ref{Fun3}).

\end{proof}

\begin{lemma}\label{PhiLemma}

Suppose $k \leq \bar{k} \leq n$ with $\left(\log n\right)^5\ \leq \bar{k}$. Then for any $z \in \mathbb{Z}_{>0}$ with $\floor{\frac{\bar{k}k}{n}} \leq z \leq k$,
\begin{align} \label{Tayl} |\Gamma_{\bar{k},k}(z)-\Phi_{\bar{k}}(z)|=O\left(1\right),\end{align} for
\begin{align} \label{FuncPhi}
\Phi_{\bar{k}}(z):=\frac{1}{2}\left( \binom{\bar{k}}{2}+\binom{z}{2} \right)&+\frac{1}{\sqrt{2}}\sqrt{  A(z)  \left( \binom{\bar{k}}{2}-\binom{z}{2} \right)} -\frac{1}{6\sqrt{2}}\sqrt{ \frac{  A(z)^{3} }{\binom{\bar{k}}{2}-\binom{z}{2} }}.
\end{align}
and $A(z)$ is defined in (\ref{Adfn}).

\end{lemma}

\begin{proof}
Let $$\epsilon:=\frac{ \log \left(\binom{k}{z} \binom{n-k}{\bar{k}-z} \right) }{\binom{\bar{k}}{2}-\binom{z}{2}}.$$ Combining the elementary inequalities $$\binom{k}{z}=\binom{k}{k-z}  \leq k^{k-z} \leq n^{\bar{k}-z}$$ and $$\binom{n-k}{\bar{k}-z}  \leq n^{\bar{k}-z}$$ with $$\bar{k} \geq \left( \log n\right)^5=\omega \left( \log n\right)$$ we conclude $$\epsilon =O \left(  \frac{ \left(\bar{k}-z\right) \log n }{ \left(\bar{k}-z\right)\left(\bar{k}+z-1\right)}\right) = O \left( \frac{ \log n}{\bar{k}}\right)=o\left(1\right).$$For our choice of $\epsilon$, $\Gamma_{\bar{k},k}$ can be simply expressed as $$\Gamma_{\bar{k},k}(z)=\binom{z}{2}+h^{-1}\left(\log 2-\epsilon\right)\left(  \binom{\bar{k}}{2}-\binom{z}{2}\right).$$ Since $\epsilon=o\left(1\right)$, Lemma \ref{EntTaylor} implies \begin{align} \label{Tayl} |\Gamma_{\bar{k},k}(z)-\Phi_{\bar{k}}(z)|=O\left(\sqrt{ \frac{  A(z)^{5} }{\left( \binom{\bar{k}}{2}-\binom{z}{2} \right)^3}}\right),\end{align} where $\Phi_{\bar{k}}(z)$ and $A(z)$ are defined in (\ref{FuncPhi}) and (\ref{Adfn}) respectively.

Using $\binom{\bar{k}}{z}\binom{n-k}{\bar{k}-z}  \leq n^{2\left(\bar{k}-z\right)}$ we have  $$A(z) \leq 2\left(\bar{k}-z\right) \log n.$$ Furthermore $\binom{\bar{k}}{2}-\binom{z}{2} \geq \frac{\bar{k}\left(\bar{k}-z\right)}{2}$. Hence, combining the last two inequalities, (\ref{Tayl}) can be simplified to 
\begin{align} \label{Tayl2} |\Gamma_{\bar{k},k}(z)-\Phi_{\bar{k}}(z)|=O\left(\sqrt{ \frac{ \left(\log n \right)^5 }{\bar{k} }}\right)=O \left(1\right),\end{align}where the last step is due to $\left(\log n \right)^5 \leq \bar{k}$. This concludes the proof of the Lemma.

\end{proof}
\begin{lemma}\label{DiscGap}
Suppose $k \leq \bar{k} \leq n$ with $\left(\log n\right)^5 \leq \bar{k}$ and $\epsilon>0$.
Then for some sufficiently large constant $C_0=C_0(\epsilon)>0$, if $\floor{C_0\frac{\bar{k}k}{n}} \leq z \leq \left(1-\epsilon \right) k$, \begin{align}\label{finalPhi10}
\Gamma_{\bar{k},k}(z+1)-\Gamma_{\bar{k},k}(z)=z\left(\frac{1}{2}-o\left(1\right)\right)-\Theta \left[  \sqrt{\frac{\bar{k}}{\log (\frac{n}{\bar{k}})}}\log \left(\frac{\left(z+1\right)n}{k\bar{k}} \right) \right]+O\left(1\right).
\end{align}

\end{lemma}
\begin{proof} First we choose $C_0>0$ large enough so that so that $\log \left(\frac{\left(z+1\right)n}{k\bar{k}} \right)$ dominates the constant additional factor in the right hand side of (\ref{Fun3}) and therefore for all $z$ of interest
\begin{equation}\label{BigC} A(z+1)-A(z)=\Theta\left(\log \left(\frac{k\bar{k}}{\left(z+1\right)n}\right)\right)=-\Theta\left(\log \left(\frac{\left(z+1\right)n}{k\bar{k}}\right)\right).
\end{equation}

%

In light of Lemma \ref{PhiLemma} we can prove (\ref{finalPhi10}) with $\Phi_{\bar{k}}(z+1)-\Phi_{\bar{k}}(z)$ (defined in (\ref{FuncPhi})) instead of $\Gamma_{\bar{k},k}(z+1)-\Gamma_{\bar{k},k}(z)$ at the expense only of $O\left(1\right)$ terms on the right hand sides. We write the difference $\Phi_{\bar{k}}(z+1)-\Phi_{\bar{k}}(z)$ as a summation of three parts.
\begin{align*}
\Phi_{\bar{k}}(z+1)-\Phi_{\bar{k}}(z)&=\underbrace{\frac{1}{2}\left( \binom{\bar{k}}{2}+\binom{z+1}{2} \right)-\frac{1}{2}\left( \binom{\bar{k}}{2}+\binom{z}{2} \right)}_{\text{First Part}}\\
&+\frac{1}{\sqrt{2}}\left(\underbrace{\sqrt{  A(z+1)  \left( \binom{\bar{k}}{2}-\binom{z+1}{2} \right)}-\sqrt{  A(z)  \left( \binom{\bar{k}}{2}-\binom{z}{2} \right)}}_{\text{Second Part}}\right)\\
&-\frac{1}{6\sqrt{2}} \left(\underbrace{\sqrt{ \frac{  A(z+1)^{3} }{\left( \binom{\bar{k}}{2}-\binom{z+1}{2} \right)}}-\sqrt{ \frac{  A(z)^{3} }{\left( \binom{\bar{k}}{2}-\binom{z}{2} \right)}}}_{\text{Third Part}}\right).\end{align*}

The first part can be straightforwardly simplified to $\frac{z}{2}$.

We write the second part as follows,
\begin{align}
&\sqrt{  A(z+1)  \left( \binom{\bar{k}}{2}-\binom{z+1}{2} \right)}-\sqrt{  A(z)  \left( \binom{\bar{k}}{2}-\binom{z}{2} \right)} \notag\\
=&\left(\sqrt{  A(z+1)}-\sqrt{ A(z)}\right) \sqrt{ \binom{\bar{k}}{2}-\binom{z+1}{2} }+\sqrt{  A(z) }\left(\sqrt{   \binom{\bar{k}}{2}-\binom{z+1}{2} }-\sqrt{   \binom{\bar{k}}{2}-\binom{z}{2} }\right) \notag \\
=&\left(\frac{  A(z+1)- A(z)}{\sqrt{A(z+1)}+\sqrt{ A(z)}}\right) \sqrt{ \binom{\bar{k}}{2}-\binom{z+1}{2} }-\sqrt{  A(z) }\frac{\binom{z+1}{2}-\binom{z}{2}}{\sqrt{   \binom{\bar{k}}{2}-\binom{z+1}{2} }+\sqrt{   \binom{\bar{k}}{2}-\binom{z}{2} }} \notag \\
=&\left(\frac{  A(z+1)- A(z)}{\sqrt{A(z+1)}+\sqrt{ A(z)}}\right) \sqrt{ \binom{\bar{k}}{2}-\binom{z+1}{2} }-\sqrt{  A(z) }\frac{z}{\sqrt{   \binom{\bar{k}}{2}-\binom{z+1}{2} }+\sqrt{   \binom{\bar{k}}{2}-\binom{z}{2} }}. \notag \\
\end{align}Since $z \leq (1-\epsilon)k \leq (1-\epsilon)\bar{k} $ applying (\ref{Fun2}) from Lemma \ref{AFun}, the last quantity is of the order
\begin{align}
& \Theta \left[ \left(\frac{  A(z+1)- A(z)}{\sqrt{\bar{k} \log (\frac{n}{\bar{k}})}}\right) \bar{k} \right]-\Theta\left[\sqrt{ \bar{k}  \log (\frac{n}{\bar{k}}) }\frac{z}{\bar{k}}\right], \text{using } \binom{\bar{k}}{2}-\binom{z}{2}=\Theta\left(\left(\bar{k}\right)^2\right) \notag \\
=& \Theta \left[ \left(\frac{  \left(A(z+1)- A(z)\right) \sqrt{\bar{k}} }{\sqrt{\log (\frac{n}{\bar{k}})}}\right)  \right]-\Theta\left[\frac{\sqrt{   \log (\frac{n}{\bar{k}}) } z}{\sqrt{\bar{k}}}\right] \notag \\
=&-\Theta \left[ \sqrt{\frac{   \bar{k} }{\log (\frac{n}{\bar{k}})} } \log \left(\frac{\left(z+1\right)n}{k\bar{k}} \right)\right]-o\left(z\right). \label{Case1} 
\end{align}where for the last equality we used (\ref{BigC}) and $\bar{k}=\omega \left( \log n \right)$.

For the third part we write,
\begin{align*}
&\sqrt{ \frac{  A(z+1)^{3} }{\left( \binom{\bar{k}}{2}-\binom{z+1}{2} \right)}}-\sqrt{ \frac{  A(z)^{3} }{\binom{\bar{k}}{2}-\binom{z}{2} }}\\
= &\frac{  A(z+1)^{\frac{3}{2}}-A(z)^{\frac{3}{2}} }{\sqrt{ \binom{\bar{k}}{2}-\binom{z+1}{2} }}+A(z)^{\frac{3}{2}} \left(\frac{1}{\sqrt{\binom{\bar{k}}{2}-\binom{z+1}{2} }}-\frac{1}{\sqrt{\binom{\bar{k}}{2}-\binom{z}{2} }}\right)
\end{align*}
Using  $a^{\frac{3}{2}}-b^{\frac{3}{2}}=\left(a^3-b^3\right)/\left(a^{\frac{3}{2}}+b^{\frac{3}{2}}\right)$ and $a^3-b^3=\left(a-b\right)\left(a^2+b^2+ab\right)=O\left( \left(a-b\right) \left(a^2+b^2\right)\right)$ for $a,b \in \mathbb{R}, $ we have that the quantity is at most
\begin{align*}
&O\left[ \frac{\left(A(z+1)-A(z)\right)\left(A(z+1)^2+A(z)^2\right)}{(A(z+1)^{\frac{3}{2}}+A(z)^{\frac{3}{2}})\sqrt{ \binom{\bar{k}}{2}-\binom{z+1}{2} }} \right]
+O\left[A(z)^{\frac{3}{2}} \left(\frac{\sqrt{ \binom{\bar{k}}{2}-\binom{z}{2} }-\sqrt{\binom{\bar{k}}{2}-\binom{z+1}{2} }}{\sqrt{\left( \binom{\bar{k}}{2}-\binom{z+1}{2} \right)\left( \binom{\bar{k}}{2}-\binom{z}{2} \right)}}\right) \right]
\end{align*}
which by Lemma \ref{AFun} and (\ref{BigC}) is at most
\begin{align*}
&O \left[ \frac{\log \left(\frac{\left(z+1\right)n}{k\bar{k}}\right) \sqrt{\bar{k} \log \left( \frac{n}{\bar{k}}\right)}}{\sqrt{ \binom{\bar{k}}{2}-\binom{z+1}{2} }} \right]+O \left[ \left(\bar{k} \log \left( \frac{n}{\bar{k}}\right)\right)^{\frac{3}{2}}  \left(\frac{\sqrt{ \binom{\bar{k}}{2}-\binom{z}{2} }-\sqrt{ \binom{\bar{k}}{2}-\binom{z+1}{2} }}{\sqrt{\left( \binom{\bar{k}}{2}-\binom{z+1}{2} \right)\left( \binom{\bar{k}}{2}-\binom{z}{2} \right)}}\right) \right] \\
=&O \left[ \frac{ \left(\log \left(\frac{n}{\bar{k}}\right)\right)^{\frac{3}{2}}}{\sqrt{\bar{k}}}\right]+O \left[ \frac{\left( \binom{z+1}{2}-\binom{z}{2}\right) \left(\log \left(\frac{n}{k}\right)\right)^{\frac{3}{2}}}{\left(\bar{k}\right)^{\frac{3}{2}}} \right],
\end{align*} where for the last equality we used the elementary $\sqrt{a}-\sqrt{b}=\left( a-b \right)/\left(\sqrt{a}+\sqrt{b}\right)$ and that $z \leq (1-\epsilon)\bar{k}$. Finally, the last displayed quantity is at most
\begin{align*}
&O \left[ \frac{ \left(\log \left(\frac{n}{\bar{k}}\right)\right)^{\frac{3}{2}}}{\sqrt{\bar{k}}}\right]+O \left[ \frac{z \left(\log \left(\frac{n}{k}\right)\right)^{\frac{3}{2}}}{\left(\bar{k}\right)^{\frac{3}{2}}} \right]\\
=&O\left[ \frac{ \left(\log \left(\frac{n}{\bar{k}}\right)\right)^{\frac{3}{2}}}{\sqrt{\bar{k}}} \right], \text{using } z \leq \bar{k}\\
=& o \left(1 \right),
\end{align*}since  $\bar{k}=\omega\left( \log^3 n\right)$ by our assumptions.

Combining the three parts gives
\begin{align}\label{finalPhi1}
\Phi_{\bar{k}}(z+1)-\Phi_{\bar{k}}(z)=z\left(\frac{1}{2}-o(1)\right)-\Theta \left[   \sqrt{ \frac{\bar{k} }{\log (\frac{n}{\bar{k}})}} \log \left( \frac{\left(z+1\right)n}{k\bar{k}} \right) \right]+o\left(1\right).
\end{align} 
which based on Lemma \ref{PhiLemma} implies (\ref{finalPhi10}).

The proof of the Lemma is complete.
\end{proof}

\begin{lemma}\label{Tn} 
Suppose $k \leq \bar{k} \leq n$ with $\left(\log n\right)^5 \leq \bar{k}$ and $\epsilon>0$.
Let \begin{equation}\label{Tn2} T_n:=\sqrt{\frac{\bar{k}}{\log\left(\frac{n}{\bar{k}}\right)}}\log \left( \sqrt{\frac{\bar{k}}{\log\left(\frac{n}{\bar{k}}\right)}} \left(\frac{\bar{k}k}{n}\right)^{-1}\right).\end{equation}
For some sufficiently large constant $C_0=C_0(\epsilon)>0$ and sufficiently large enough values of $n$ the following monotonicity properties hold in the discretized interval $$\mathcal{I}=\mathcal{I}_{C_0}=[\floor{C_0 \frac{\bar{k}k}{n}}, (1-\epsilon) k] \cap \mathbb{Z}.$$
\begin{itemize}
\item[(1)] If $T_n=o\left(\frac{\bar{k}k}{n}\right)$ then $\Gamma_{\bar{k},k}$ is monotonically increasing on $\mathcal{I}$.
\item[(2)] If $T_n=\omega \left(k\right)$ then $\Gamma_{\bar{k},k}$ is monotonically decreasing on $\mathcal{I}$.
\item[(3)] If $\omega\left(\frac{\bar{k}k}{n}\right)=T_n=o\left(k\right)$ then $\Gamma_{\bar{k},k}$ is non-monotonous on $\mathcal{I}$ with the property that for some constants $0<D_1<D_2$, $u_1:=D_1  \ceil{\sqrt{\frac{\bar{k}}{\log \left(\frac{n}{\bar{k}}\right)}}}$ and $u_2:=D_2 \ceil{\sqrt{\frac{\bar{k}}{\log \left(\frac{n}{\bar{k}}\right)}}}$ the following are true.
\begin{itemize}
\item[(a)]  $\floor{ C_0\frac{\bar{k}k}{n}}<u_1<u_2<(1-\epsilon) k$ and 

\item[(b)]
\begin{align}\label{slack}
  \max_{z \in \mathcal{I} \cap \left[u_1,u_2\right]} \Gamma_{\bar{k},k}(z) +\Omega\left( \frac{\bar{k}}{ \log \left(\frac{n}{\bar{k}}\right)} \right) \leq \Gamma_{\bar{k},k}(\floor{C \frac{\bar{k}k}{n}}) \leq \Gamma_{\bar{k},k}\left((1-\epsilon) k\right).
\end{align}
\end{itemize}
\end{itemize}
\end{lemma}

\begin{proof}
We start with the case $T_n=o\left(\frac{\bar{k}k}{n}\right)$ which can be equivalently written as $$\sqrt{\frac{\bar{k}}{\log\left(\frac{n}{\bar{k}}\right)}}\left(\frac{\bar{k}k}{n}\right)^{-1 }\log \left( \sqrt{\frac{\bar{k}}{\log\left(\frac{n}{\bar{k}}\right)}} \left(\frac{\bar{k}k}{n}\right)^{-1}\right)=o\left(1\right)$$ or using part (a) of Lemma \ref{Calc}, \begin{equation}\label{Case1Tn}
\sqrt{\frac{\bar{k}}{\log\left(\frac{n}{\bar{k}}\right)}}\left(\frac{\bar{k}k}{n}\right)^{-1 }=o\left(1\right).
\end{equation}

Using (\ref{finalPhi10}) from Lemma \ref{DiscGap} we have that for some universal constant $C_1>0$ and large enough $n$,\begin{align*}
\Gamma_{\bar{k},k}(z+1)- \Gamma_{\bar{k},k}(z) &\geq \frac{z}{4}-C_1\sqrt{ \frac{\bar{k}}{\log \left(\frac{n}{\bar{k}}\right)} }  \log \left(\frac{n\left(z+1\right)}{\bar{k}k}\right)-O\left(1\right)\\
&= \frac{ k\bar{k}}{4n}\log \left(\frac{n\left(z+1\right)}{\bar{k}k}\right) \left( \frac{\frac{nz}{\bar{k}k}}{\log \left(\frac{n\left(z+1\right)}{\bar{k}k}\right)} - 4C_1 \sqrt{\frac{\bar{k}}{\log\left(\frac{n}{\bar{k}}\right)}}\left(\frac{\bar{k}k}{n}\right)^{-1 } \right)-O\left(1\right).
 \end{align*} The second term in the parenthesis in the last displayed quantity is $o\left(1\right)$ from (\ref{Case1Tn}). Now notice that since $\bar{k}=\omega\left( \log n \right)$, from (\ref{Case1Tn}) we have \begin{align} \label{imp} \frac{\bar{k}k}{n}=\omega\left(1\right).\end{align}\ Therefore choosing $C_0>0$ large enough we have that $z \geq \floor{C_0\frac{\bar{k}k}{n}}$ implies that the first term in the parenthesis can be made to be at least $1$.  Finally, the multiplicative term outside the parenthesis satisfies $$\frac{k\bar{k}}{4n}\log \left(\frac{n\left(z+1\right)}{\bar{k}k}\right) \geq\frac{k\bar{k}}{4n}  \log e^4 =\frac{k\bar{k}}{n}$$  by choosing $z+1 \geq \floor{\frac{C_0k\bar{k}}{n}}$ for say $C_0>e^4$. Hence, indeed for some sufficiently large $C_0>0$ if $z \in \mathcal{I}_{C_0}$, $$\Gamma_{\bar{k},k}(z+1)- \Gamma_{\bar{k},k}(z) \geq  \frac{k\bar{k}}{n} \left(1-o\left(1\right) \right)-O\left(1\right)$$ which according to (\ref{imp}) implies that  for some sufficiently large $C_0>0$ for $n$ large enough if $z \in \mathcal{I}_{C_0}$, $$\Gamma_{\bar{k},k}(z+1) \geq \Gamma_{\bar{k},k}(z),$$ that is the curve is increasing.

We now turn to Part (2) where $T_n=\omega\left(k \right)$ which can be equivalently written as $$\sqrt{\frac{\bar{k}}{\log\left(\frac{n}{\bar{k}}\right)}}\left(\frac{\bar{k}k}{n}\right)^{-1 }\log \left( \sqrt{\frac{\bar{k}}{\log\left(\frac{n}{\bar{k}}\right)}} \left(\frac{\bar{k}k}{n}\right)^{-1}\right)=\omega\left(\frac{n}{\bar{k}}\right)$$ or using that $\bar{k}=o\left(n\right)$ and part (c) of Lemma \ref{Calc}, \begin{equation}\label{Case2Tn2}
\sqrt{\frac{\bar{k}}{\log\left(\frac{n}{\bar{k}}\right)}}\left(\frac{\bar{k}k}{n}\right)^{-1 }=\omega\left(\frac{\frac{n}{\bar{k}} }{\log \left(\frac{n}{\bar{k}} \right)} \right).
\end{equation}which simplifies to 
\begin{equation}\label{Case2Tn}
\sqrt{\bar{k}\log\left(\frac{n}{\bar{k}}\right)}=\omega\left(k  \right).
\end{equation}

Now using (\ref{finalPhi10}) from Lemma \ref{DiscGap} we have that for some universal constants $U_1>0$ and large enough $n$,

\begin{align}
\Gamma_{\bar{k},k}(z+1)- \Gamma_{\bar{k},k}(z) &\leq \frac{3\left(z+1\right)}{4}-U_1\sqrt{ \frac{\bar{k}}{\log \left(\frac{n}{\bar{k}}\right)} }  \log \left(\frac{n\left(z+1\right)}{\bar{k}k}\right)+O\left(1\right) \notag\\
&= \frac{ 3k\bar{k}}{4n}\log \left(\frac{n\left(z+1\right)}{\bar{k}k}\right) \left( \frac{\frac{n\left(z+1\right)}{\bar{k}k}}{\log \left(\frac{n\left(z+1\right)}{\bar{k}k}\right)} - \frac{4}{3}U_1 \sqrt{\frac{\bar{k}}{\log\left(\frac{n}{\bar{k}}\right)}}\left(\frac{\bar{k}k}{n}\right)^{-1 } \right)+O\left(1\right). \label{Case2Diff}
 \end{align} 
Recall that for $x>e$, $x/\log x$ is increasing from elementary reasoning. Therefore if $C_0>e$ using \begin{align} \label{Ce} \frac{n\left(z+1\right)}{\bar{k}k} \geq e \end{align} and the trivial $\frac{n\left(z+1\right)}{\bar{k}k} \leq \frac{n}{\bar{k}}$, we have $$\frac{\frac{n\left(z+1\right)}{\bar{k}k}}{\log \left(\frac{n\left(z+1\right)}{\bar{k}k}\right)} \leq  \frac{\frac{n}{\bar{k}}}{\log \left(\frac{n}{\bar{k}}\right)}.$$ Hence, by (\ref{Case2Tn2}) we conclude $$\frac{\frac{n\left(z+1\right)}{\bar{k}k}}{\log \left(\frac{n\left(z+1\right)}{\bar{k}k}\right)}= o\left(\sqrt{\frac{\bar{k}}{\log\left(\frac{n}{\bar{k}}\right)}}\left(\frac{\bar{k}k}{n}\right)^{-1 } \right).$$Therefore indeed the term inside the parenthesis in (\ref{Case2Diff}) is at most $-U_1 \sqrt{\frac{\bar{k}}{\log\left(\frac{n}{\bar{k}}\right)}}\left(\frac{\bar{k}k}{n}\right)^{-1 }$ for large enough $n$, which allows to conclude that  for large enough $n$  (\ref{Case2Diff}) implies for all $z \in \mathcal{I}_C$, \begin{align} \Gamma_{\bar{k},k}(z+1)- \Gamma_{\bar{k},k}(z) & \leq -\frac{3}{4}U_1\sqrt{ \frac{\bar{k}}{\log \left(\frac{n}{\bar{k}}\right)} }  \log \left(\frac{n\left(z+1\right)}{\bar{k}k}\right)+O\left(1\right) \notag\\
& \leq -\frac{3}{4}U_1\sqrt{ \frac{\bar{k}}{\log \left(\frac{n}{\bar{k}}\right)} }  +O\left(1\right) \label{IneqCe}\end{align}where we have used $\log \left(\frac{n\left(z+1\right)}{\bar{k}k}\right) \geq 1$ according to (\ref{Ce}). Using now that $\bar{k}=\omega \left( \log n \right)$ we conclude based on (\ref{IneqCe}), that indeed for some sufficiently large $C_0>0$ and large enough $n$, if $z \in \mathcal{I}_{C_0}$, $\Gamma_{\bar{k},k}(z+1) \leq \Gamma_{\bar{k},k}(z)$, that is the curve is decreasing.

We finally turn to Part (3) where $\omega\left(\frac{\bar{k}k}{n}\right)=T_n=o\left(k \right).$ By similar arguments as for (\ref{Case1Tn}) and (\ref{Case2Tn}) we conclude that in this case
\begin{equation}\label{Case3Tn1}
\sqrt{\frac{\bar{k}}{\log\left(\frac{n}{\bar{k}}\right)}}=\omega\left(\frac{\bar{k}k}{n}\right)
\end{equation} and
\begin{equation}\label{Case3Tn2}
\sqrt{\bar{k}\log\left(\frac{n}{\bar{k}}\right)}=o\left(k  \right).
\end{equation}

Notice that because of (\ref{Case3Tn1}) and (\ref{Case3Tn2}) we have that for any choice of $D_1,D_2>0$ and sufficiently large $n$, $$ \floor{C_0\frac{\bar{k}k}{n}}<u_1=D_1\ceil{\sqrt{ \frac{\bar{k}}{\log \left(\frac{n}{\bar{k}}\right)}}}<u_2=D_2\ceil{\sqrt{ \frac{\bar{k}}{\log \left(\frac{n}{\bar{k}}\right)}}}<(1-\epsilon)k.$$
It suffices now to establish (\ref{slack}) as non-monotonicity is directly implied by it.

By definition of $\Gamma_{\bar{k},k}$ to establish for large $n$ \begin{align} \label{GOAL} \Gamma_{\bar{k},k}\left((1-\epsilon) k\right) \geq \Gamma_{\bar{k},k}(\floor{C_0\frac{\bar{k}k}{n}})\end{align} it suffices to establish that for large $n$,  $$\binom{k}{2} +h^{-1} \left(\log 2-\frac{ \log \left( \binom{k}{ \left(1-\epsilon\right)k} \binom{n-k}{\bar{k}-(1-\epsilon) k}\right)  }{\binom{\bar{k}}{2}-\binom{(1-\epsilon)  k}{2}} \right)\left(\binom{\bar{k}}{2}-\binom{(1-\epsilon) k}{2}\right) $$ is bigger than $$\binom{\floor{C_0\frac{\bar{k}k}{n}}}{2}+h^{-1} \left(\log 2-\frac{ \log  \binom{k}{k-\floor{C_0\frac{\bar{k}k}{n}}}\binom{n-k}{\bar{k}}  }{\binom{\bar{k}}{2} -\binom{\floor{C\frac{\bar{k}k}{n}}}{2} }\right)\left(\binom{\bar{k}}{2}-\binom{\floor{C\frac{\bar{k}k}{n}}}{2}\right).$$  Since $\bar{k}=\omega \left( \log n \right)$ both the arguments of $h^{-1}$ in the displayed equations are  $\log 2-o\left(1\right)$. Hence by Lemma \ref{EntTaylor} it suffices for large $n$ to prove $$\frac{1}{2} \left(\binom{\bar{k}}{2}+\binom{(1-\epsilon) k}{2}\right)+\Theta \left( \sqrt{\log \left( \binom{k}{ \left(1-\epsilon\right)k} \binom{n-k}{\bar{k}-(1-\epsilon) k}\right) \left(\binom{\bar{k}}{2}-\binom{(1-\epsilon)k}{2}\right)} \right)$$ is bigger than
$$\frac{1}{2} \left(\binom{\bar{k}}{2}+\binom{\floor{C_0\frac{\bar{k}k}{n}}}{2}\right)+\Theta \left( \sqrt{ \log \left[\binom{k}{\floor{C_0\frac{\bar{k}k}{n}}}\binom{n-k}{\bar{k}-\floor{C_0\frac{\bar{k}k}{n}}}\right] \left(\binom{\bar{k}}{2}-\binom{\floor{C_0\frac{\bar{k}k}{n}}}{2}\right)} \right).$$

Since by (\ref{Case3Tn1}) and (\ref{Case3Tn2}) we have $k=\omega\left(\frac{\bar{k}k}{n}\right)$ it suffices \begin{align*} k^2=\omega \left(  \sqrt{ \log \left[\binom{k}{\floor{C_0\frac{\bar{k}k}{n}}}\binom{n-k}{\bar{k}-\floor{C_0\frac{\bar{k}k}{n}}}\right]\binom{\bar{k}}{2}}-\sqrt{ \log  \left[\binom{k}{ \left(1-\epsilon\right)k}\binom{n-k}{\bar{k}-\left(1-\epsilon \right)k} \right]\left(\binom{\bar{k}}{2}-\binom{k}{2}\right)}  \right).\end{align*} Using that $\left(1-\epsilon \right)k \leq k$ and that for large $n$, $\bar{k} \leq \frac{n}{2}$ by standard monotonicity arguments we have $$\binom{k}{ \left(1-\epsilon\right)k}\binom{n-k}{\bar{k}-\left(1-\epsilon \right)k}  \geq \binom{n-k}{\bar{k}-k}.$$ Hence it suffices to show \begin{align*} k^2=\omega \left(  \sqrt{ \log \left[\binom{k}{\floor{C_0\frac{\bar{k}k}{n}}}\binom{n-k}{\bar{k}-\floor{C_0\frac{\bar{k}k}{n}}}\right]\binom{\bar{k}}{2}}-\sqrt{ \log  \left[\binom{n-k}{\bar{k}-k} \right]\left(\binom{\bar{k}}{2}-\binom{k}{2}\right)}  \right).\end{align*} Now since for large $n$, $\bar{k}<\frac{n-k}{2}$, using the elementary $$\binom{k}{\floor{C_0\frac{\bar{k}k}{n}}} \binom{n-k}{\bar{k}-\floor{C_0\frac{\bar{k}k}{n}}}  \leq 2^k \binom{n-k}{\bar{k}} \leq e^{O\left(\bar{k} \log (\frac{n-k}{\bar{k}})\right)}$$ and $$\binom{n-k}{\bar{k}-k} \geq \left(\frac{\left(n-\bar{k}\right)}{\bar{k}}\right)^{\bar{k}-k}$$ we conclude that it suffices to have
$$k^2=\omega\left( \left(\left(\bar{k}\right)^{\frac{3}{2}}-\sqrt{\bar{k}}\left(\bar{k}-k\right)\right) \log \left(\frac{n-k}{\bar{k}} \right) \right)=\omega\left (\sqrt{\bar{k}}k\log \left(\frac{n-k}{\bar{k}} \right) \right),$$which follows directly from (\ref{Case3Tn2}). This establishes (\ref{GOAL}) for large enough $n$.


%

Now using (\ref{finalPhi10}) from Lemma \ref{DiscGap} to conclude that for some universal constants $U_1>0$, large enough $n$ and such $z$,

\begin{align}\label{Case3Tn3}
\Gamma_{\bar{k},k}(z+1)- \Gamma_{\bar{k},k}(z) &\leq z-U_1\sqrt{ \frac{\bar{k}}{\log \left(\frac{n}{\bar{k}}\right)} }  \log \left(\frac{n\left(z+1\right)}{\bar{k}k}\right)+O\left(1\right).
 \end{align} Using $z+1 \geq \floor{C_0\frac{\bar{k}k}{n}}$ for $C_0>e$ and focusing only on $\floor{C_0\frac{\bar{k}k}{n}} \leq z \leq \frac{U_1}{2}\sqrt{ \frac{\bar{k}}{\log \left(\frac{n}{\bar{k}}\right)} }$ (existence of such $z$ follows by (\ref{Case3Tn1}) ) we conclude for any such $z$ and large enough $n$,
\begin{align}\label{Case3Tn4}
\Gamma_{\bar{k},k}(z+1)- \Gamma_{\bar{k},k}(z) &\leq z-U_1\sqrt{ \frac{\bar{k}}{\log \left(\frac{n}{\bar{k}}\right)} }+O\left(1\right) \leq -\frac{U_1}{2}\sqrt{ \frac{\bar{k}}{\log \left(\frac{n}{\bar{k}}\right)} }
 \end{align}where we used the fact that $\bar{k}=\omega\left( \log n\right)$. Now set $$D_1:=\frac{U_1}{4},D_2:=\frac{U_1}{2}.$$ Fix any $Z \in \mathcal{I}$ with $D_1\sqrt{ \frac{\bar{k}}{\log \left(\frac{n}{\bar{k}}\right)} }  \leq Z \leq D_2\ceil{\sqrt{ \frac{\bar{k}}{\log \left(\frac{n}{\bar{k}}\right)} } }$. Focus on $z \in \mathcal{I}$ with $ \floor{C_0 \frac{\bar{k}k}{n}} \leq z \leq Z-1$. (\ref{Case3Tn1}) yields that the the number of such $z$'s for large $n$ is at least $ \frac{D_1}{2}\sqrt{ \frac{\bar{k}}{\log \left(\frac{n}{\bar{k}}\right)}}$. By telescopic summation of (\ref{Case3Tn4}) over these $z$ we have
\begin{align}\label{Case3Tn5}
\Gamma_{\bar{k},k}\left(Z\right)+D_1^2 \frac{\bar{k}}{\log \left(\frac{n}{\bar{k}}\right)}  &\leq \Gamma_{\bar{k},k}(\floor{C_0 \frac{\bar{k}k}{n}}).
 \end{align}Since $Z$ was arbitrary we conclude that,

\begin{align} \label{slackFin2}  \max_{z \in \mathcal{I} \cap \left[D_1\ceil{\sqrt{ \frac{\bar{k}}{\log \left(\frac{n}{\bar{k}}\right)} }}, D_2\ceil{\sqrt{ \frac{\bar{k}}{\log \left(\frac{n}{\bar{k}}\right)} } }\right]} \Gamma_{\bar{k},k}(z) +\Omega\left( \frac{\bar{k}}{ \log \left(\frac{n}{\bar{k}}\right)} \right) \leq \Gamma_{\bar{k},k}(\floor{C_0 \frac{\bar{k}k}{n}}).\end{align}

Equations (\ref{slackFin2}) and (\ref{GOAL}) imply (\ref{slack}). The proof of the Lemma is complete.

\end{proof}

\subsection{Proof of Theorem \ref{FM}}
\begin{proof}[Proof of Theorem \ref{FM}]

We start with the case where $k=o\left(\sqrt{n}\right).$  Notice that $k=o\left(\sqrt{n}\right)$ together with $\bar{k}=o\left(n\right)$ trivially imply
$$\frac{ \frac{n}{\bar{k}}}{ \log \frac{n}{\bar{k}}}=\omega \left(\frac{k^2}{n}\right)$$which can be written equivalently as $$\frac{n}{k^2}=\omega\left( \frac{\bar{k}}{n} \log \left(\frac{n}{\bar{k}}\right) \right)$$  or $$\sqrt{ \frac{\bar{k}}{\log \left(\frac{n}{\bar{k}}\right)}}=\omega \left( \frac{\bar{k}k}{n} \right)$$ or $$\left( \frac{\bar{k}k}{n} \right)^{-1}\sqrt{ \frac{\bar{k}}{\log \left(\frac{n}{\bar{k}}\right)}}=\omega \left(1 \right).$$ Using part (b) of Lemma \ref{Calc} we have $$\left(\frac{\bar{k}k}{n}\right)^{-1} \sqrt{ \frac{\bar{k}}{\log \left(\frac{n}{\bar{k}}\right)}} \log \left(\left(\frac{\bar{k}k}{n}\right)^{-1} \sqrt{ \frac{\bar{k}}{\log \left(\frac{n}{\bar{k}}\right)}}  \right)=\omega \left( 1 \right)$$or \begin{equation} \label{FmM1} T_n=\omega\left(\frac{\bar{k}k}{n}\right),\end{equation}where $T_n$ is defined in equation (\ref{Tn2}).

First, we consider the subcase where $\bar{k}=o\left(\frac{k^2}{\log \left(\frac{n}{k^2}\right)} \right)$. In that case, we have
$$\frac{n}{\bar{k}}=\omega \left( \frac{n}{k^2} \log \left( \frac{n}{k^2} \right) \right)$$which since $k^2=o\left(n\right)$ which according to part (d) of Lemma \ref{Calc} implies \begin{equation} \label{FmStep1}\frac{\frac{n}{\bar{k}}}{ \log \left( \frac{n}{\bar{k}} \right)}=\omega \left( \frac{n}{k^2}\right) \end{equation} which is equivalent with 
$$k=\omega \left(\sqrt{\bar{k}\log \left(\frac{n}{\bar{k}}\right)}\right)$$ or $$\frac{\frac{n}{\bar{k}}}{\log \left(\frac{n}{\bar{k}}\right)}=\omega\left(\left(\frac{\bar{k}k}{n}\right)^{-1} \sqrt{ \frac{\bar{k}}{\log \left(\frac{n}{\bar{k}}\right)}}\right)$$ or$$\frac{n}{\bar{k}}=\omega \left(\left(\frac{\bar{k}k}{n}\right)^{-1} \sqrt{ \frac{\bar{k}}{\log \left(\frac{n}{\bar{k}}\right)}} \log \left(\left(\frac{\bar{k}k}{n}\right)^{-1} \sqrt{ \frac{\bar{k}}{\log \left(\frac{n}{\bar{k}}\right)}}  \right)\right).$$ The last equality can be rewritten $$\frac{n}{\bar{k}}=\omega \left(\left(\frac{\bar{k}k}{n}\right)^{-1} T_n \right),$$ where $T_n$ is defined in equation (\ref{Tn2}), which now simplifies to \begin{equation} \label{Fm1} T_n=o\left(k\right).\end{equation}

Combining (\ref{FmM1}) with (\ref{Fm1}), according to Part (3) of Lemma \ref{Tn} we conclude the desired non-monotonicity result in that subcase.

Second, we consider the subcase $\bar{k}=\omega\left(\frac{k^2}{\log \left(\frac{n}{k^2}\right)} \right)$. In tha case, following similar to the derivation of (\ref{Fm1}) by simply the order of comparison (in more detail, reversing the $o$-notation with the $\omega$-notation and applying the other direction of part (d) of Lemma \ref{Calc}), we conclude that in this case $T_n=\omega \left(k \right)$. In particular, according to Part (2) of Lemma \ref{Tn} we can conclude that the curve is decreasing in that subcase.

We turn now to the case $k=\omega \left(\sqrt{n}\right)$. In that case, together with $\bar{k}=o\left(n\right)$ we have $$\frac{\frac{n}{\bar{k}}}{ \log \left( \frac{n}{\bar{k}} \right)}=\omega \left( \frac{n}{k^2}\right) $$ which is exactly (\ref{FmStep1}). Following the identical derivation following (\ref{FmStep1}) we conclude that (\ref{Fm1}) holds in this case.

First, we consider the subcase $\bar{k}=o \left( \frac{ \frac{n^2}{k^2}}{\log \left(\frac{n}{k^2}\right)}\right)$ which can be written equivalently as $$\frac{n}{\bar{k}}=\omega\left( \frac{k^2}{n} \log \left( \frac{k^2}{n} \right) \right).$$ Since $k^2=\omega \left(n\right)$ according to part (d) of Lemma \ref{Calc} we have$$\frac{k^2}{n}=o \left( \frac{\frac{n}{\bar{k}}}{\log \left(\frac{n}{\bar{k}} \right)} \right)$$ or $$\frac{\bar{k}k}{n}=o\left(\sqrt{\frac{\bar{k}}{\log \left( \frac{n}{\bar{k}} \right)}}  \right)$$or $$\left(\frac{\bar{k}k}{n}\right)^{-1}\sqrt{\frac{\bar{k}}{\log \left( \frac{n}{\bar{k}} \right)}}= \omega \left(1  \right)$$ which according to part (b) of Lemma \ref{Calc} gives $$\left(\frac{\bar{k}k}{n}\right)^{-1} \sqrt{ \frac{\bar{k}}{\log \left(\frac{n}{\bar{k}}\right)}} \log \left(\left(\frac{\bar{k}k}{n}\right)^{-1} \sqrt{ \frac{\bar{k}}{\log \left(\frac{n}{\bar{k}}\right)}}  \right)=\omega \left( 1 \right).$$ The last equality simplifies to (\ref{FmM1}). In particular, in this regime both (\ref{FmM1}) and (\ref{Fm1}) are hence according to Part (3) of Lemma \ref{Tn} we can conclude the desired non-monotonicity result in this subcase.

Second, we consider the subcase $\bar{k}=\omega \left( \frac{ \frac{n^2}{k^2}}{\log \left(\frac{n}{k^2}\right)}\right)$. Following similar reasoning to the derivation of (\ref{FmM1}) under the assumption $\bar{k}=o \left( \frac{ \frac{n^2}{k^2}}{\log \left(\frac{n}{k^2}\right)}\right)$ one can establish $T_n=o \left(\frac{\bar{k}k}{n}\right)$ from $\bar{k}=\omega \left( \frac{ \frac{n^2}{k^2}}{\log \left(\frac{n}{k^2}\right)}\right)$ imply (in more detail, we need to reverse the $o$-notation with the $\omega$-notation at all places and apply the other direction of part (b) of Lemma \ref{Calc})). Hence, using Part (1) of Lemma \ref{Tn} allows us to conclude that the curve is increasing in this subcase.

Finally, we observe that Lemma \ref{DiscGap} together with the fact that $z \leq k$ implies $\log \left( \frac{\left(z+1\right)n}{k\bar{k}}\right)=O\left( \log \frac{n}{k} \right)$ readily imply part (c) of Theorem \ref{FM} for a sufficiently large constant $E>0$.

This completes the proof of Theorem \ref{FM}.

\end{proof}

\section{Proof of the Presence of the Overlap Gap Property}
\label{Sec:OGP}
\begin{proof}[Proof of Theorem \ref{OGP}]

First, we apply Theorem \ref{FM} for $\epsilon=\frac{1}{2}$ and we denote by $C_0=C_0\left(\frac{1}{2}\right)>0$ the constant implied by Theorem \ref{FM}. Notice that  since under our assumptions both $\bar{k},k$ are $o\left( \sqrt{n} \right)$, $\bar{k}k=o\left(n\right)$, and therefore for large $n$, $\floor{C_0 \frac{\bar{k}k}{n} }=0$. In particular, the interval containg the overlap sizes of interest simplifies to $$\mathcal{I}=\left[0,\frac{k}{2}\right] \cap \mathbb{Z}.$$ Furthermore according to our parameter assumptions on $\bar{k},k,n$ we are in the case (1i) of Theorem \ref{FM} where $\Gamma_{\bar{k},k}(z),z \in \mathcal{I}$ is non-monotonic and satisfies (\ref{slack2}). Specifically, let $D_1,D_2,u_1,u_2$ as in Theorem \ref{FM} so that for large enough $n$, \begin{align} \label{order} \floor{C_0 \frac{\bar{k}k}{n} }=0 <u_1<u_2<\frac{k}{2} \end{align} and (\ref{slack2}) holds.

We first establish (\ref{slack22}) for $D_1,D_2,u_1,u_2$ as chosen above. By Proposition \ref{unionkkk} we know that for all $z \in \mathcal{I}$, $d_{\bar{k},k}(G)(z) \leq \Gamma_{\bar{k},k}(z)$, w.h.p. as $n \rightarrow + \infty$. Combining with (\ref{slack2}) we have that for some constant $c_0>0$:\begin{align} \label{slackOGP} \Gamma_{\bar{k},k}(0) \geq  \max_{z \in \mathcal{I}\cap \left[u_1,u_2\right]} d_{\bar{k},k}(G)(z) +c_0\frac{\bar{k}}{ \log \left(\frac{n}{\bar{k}}\right)},\end{align}w.h.p. as $n \rightarrow + \infty$. Hence, to establish (\ref{slack22}) from (\ref{slackOGP}) it suffices to establish that \begin{align} \label{slackG} \min \{ d_{\bar{k},k}(G)(0),d_{\bar{k},k}(G)(\frac{k}{2})\}   \geq   \Gamma_{\bar{k},k}(0) -o\left(\frac{\bar{k}}{ \log \left(\frac{n}{\bar{k}}\right)}\right),\end{align} w.h.p. as $n \rightarrow + \infty$.  Indeed, combining (\ref{slackG}) with (\ref{slackOGP}) implies  \begin{align} \label{slackGG} \min \{ d_{\bar{k},k}(G)(0),d_{\bar{k},k}(G)(\frac{k}{2})\}   \geq   \max_{z \in \mathcal{I}\cap \left[u_1,u_2\right]} d_{\bar{k},k}(G)(z) +\frac{c_0}{2}\frac{\bar{k}}{ \log \left(\frac{n}{\bar{k}}\right)},\end{align} w.h.p. as $n \rightarrow + \infty$ which implies (\ref{slack22}). 

We first prove \begin{align} \label{slackG1}d_{\bar{k},k}(G)(0) \geq   \Gamma_{\bar{k},k}(0) -o\left(\frac{\bar{k}}{ \log \left(\frac{n}{\bar{k}}\right)}\right),\end{align} w.h.p. as $n \rightarrow + \infty$. Notice that the exponent $C=\log k / \log n$ satisfies $$C <\frac{1}{2}-\frac{\sqrt{6}}{6}=1-\frac{1}{6-2\sqrt{6}}$$ and as it can straightforwardly checked it also satisfies $\frac{3}{2}-\left(\frac{5}{2}-\sqrt{6}\right)  \frac{1-C}{C}<1.$ Therefore some $\beta(C)>0$ satisfies $$\frac{3}{2}-\left(\frac{5}{2}-\sqrt{6}\right)  \frac{1-C}{C}<\beta(C)<1.$$  Part (2) of Theorem \ref{unionkkk} gives for this value of $\beta=\beta(C)$
\begin{align} \label{slackG01}  d_{k,\bar{k}}(G)(0) \geq \Gamma_{\bar{k},k}(0)-O\left(\left(\bar{k}\right)^{\beta} \sqrt{  \log n} \right)
\end{align} w.h.p. as $n \rightarrow + \infty$. Since $\beta<1$,  we have $\left(\bar{k}\right)^{\beta} \sqrt{  \log n}=o\left( \frac{\bar{k}}{\log (n/\bar{k})}\right)$. Hence, using (\ref{slackG01}), we can directly conclude
(\ref{slackG1}).

We now proceed with proving \begin{align} \label{slackG23}d_{\bar{k},k}(G)\left(\frac{k}{2} \right)\geq   \Gamma_{\bar{k},k}(0),\end{align} w.h.p. as $n \rightarrow + \infty$. Note that (\ref{slackG23}) combined with (\ref{slackG1}) imply (\ref{slackG}). First, denote by $G_0:=G \setminus \mathcal{PC}$ the graph obtained by deleting from $G$ the vertices of $\mathcal{PC}$ and notice that $G_0$ is simply distributed as an \ER model $G_0 \sim G\left(n-k,\frac{1}{2}\right)$. Second, we fix an arbitrary $\frac{k}{2}$-vertex subgraph  $S_1$ of $\mathcal{PC}$ and then optimize over the $N:=\binom{n-k}{\bar{k}-\frac{k}{2}}$ different $(\bar{k}-\frac{k}{2})$-vertex subgraphs $S_2$ of $G_0$ to get
 \begin{align} \label{trivc} d_{\bar{k},k}(G)\left(\frac{k}{2}\right) \geq \max_{S_2} |\mathrm{E}\left(S_1 \cup S_2\right)|=\binom{\frac{k}{2}}{2}+\max_{S_2} \{|\mathrm{E}(S_1,S_2)|+|\mathrm{E}(S_2)|\}, \end{align} where we used the fact that $S_1$ is a subset of the planted clique and by $\mathrm{E}(S_1,S_2)$ we denote to the set of edges with one endpoint in $S_1$ and one endpoint in $S_2$. We now index the subsets $S_2$ by $S^i,i=1,2,\ldots,N$ and set $X_i=|\mathrm{E}(S_1,S^i)|$ and $Y_i=|\mathrm{E}(S^i)|$. It is straightforward to see because of the distribution of $G_0$ that \begin{itemize}
 \item[(1)] for each $i \in [N]$, $X_i \sim \mathrm{Bin}\left( \binom{\bar{k}-k/2}{2},\frac{1}{2}\right)$
\item[(2)] $Y_i, i \in [N]  $ are i.i.d. $\mathrm{Bin}\left((\bar{k}-\frac{k}{2})\frac{k}{2},\frac{1}{2}\right)$
\item[(2)]  the sequence $X_i,  i \in [N]$ is independent from the sequence $Y_j , j \in [N]$ and 
\item[(4)] $\max_{i \in [N]} \{X_i  \}=d_{\mathrm{ER},\bar{k}-\frac{k}{2}}(G_0)$, where $d_{\mathrm{ER},K}(\cdot )$ is defined for any $K \in [|V(G_0)|]$ in  (\ref{ERdense}).
\end{itemize} Hence, combining (\ref{trivc}) and the above four observations with Lemma \ref{lem:SumMax} we have
\begin{align} d_{\bar{k},k}(G)\left(\frac{k}{2}\right) & \geq \binom{\frac{k}{2}}{2}+\max_{i=1,2,\ldots,N} \{X_i\}+ \max\{\frac{(\bar{k}-\frac{k}{2})\frac{k}{2}}{2}- \sqrt{(\bar{k}-\frac{k}{2})\frac{k}{2}\log \log  N},0\}\notag\\
&\geq \binom{\frac{k}{2}}{2}+d_{\mathrm{ER},\bar{k}-\frac{k}{2}}(G_0)+\max\{\frac{(\bar{k}-\frac{k}{2})\frac{k}{2}}{2}- \sqrt{\bar{k}k\log \log  N},0\} \notag \\
&=\binom{\frac{k}{2}}{2}+d_{\mathrm{ER},\bar{k}-\frac{k}{2}}(G_0)+\max\{\frac{(\bar{k}-\frac{k}{2})\frac{k}{2}}{2}-O\left( \sqrt{\bar{k}k \log n} \right),0\}, \label{FirstDense1}
\end{align}where for the last equality we have used that $ N=\binom{n-\bar{k}}{\bar{k}-k} \leq 2^{n-\bar{k}} \leq 2^n$ and therefore $\log \log  N = O \left( \log n \right).$

 Since by our assumption $\bar{k}=\Theta\left(n^C\right)$ for $C<\frac{1}{2}$ and $k \leq  \bar{k}$ we have  $ \frac{\bar{k}}{2} \leq \bar{k}-\frac{k}{2} \leq \bar{k}$ and therefore $\bar{k}-\frac{k}{2}=\Theta \left(n^C\right)$. Hence Theorem \ref{denseeee} can be applied, according to which for any $\beta>0$ with $\frac{3}{2}-\left(\frac{5}{2}-\sqrt{6}\right)  \frac{1-C}{C}<\beta<1$ it holds,
$$d_{\mathrm{ER},\bar{k}-\frac{k}{2}}(G_0) \geq h^{-1} \left(\log 2-\frac{ \log  \binom{n-k}{\bar{k}-\frac{k}{2}}  }{\binom{\bar{k}-\frac{k}{2}}{2}} \right)\binom{\bar{k}-\frac{k}{2}}{2}-O\left(\left(\bar{k}\right)^{\beta }\sqrt{\log n}\right).$$ Hence, using (\ref{FirstDense1}), \begin{equation} \label{2Step2} d_{\bar{k},k}(G)\left( \frac{k}{2}\right) \geq \binom{\frac{k}{2}}{2}+h^{-1} \left(\log 2-\frac{ \log  \binom{n-k}{\bar{k}-\frac{k}{2}}  }{\binom{\bar{k}-\frac{k}{2}}{2}} \right)\binom{\bar{k}-\frac{k}{2}}{2}+\frac{\left(\bar{k}-\frac{k}{2}\right)k}{2}-O\left( \sqrt{\bar{k}k \log n}+\left(\bar{k}\right)^{\beta}\sqrt{\log n}\right),\end{equation} w.h.p. as $n \rightarrow + \infty$.

Using Lemma \ref{EntTaylor} for Taylor expanding $h^{-1}$ the lower bound of (\ref{2Step2}) simplifies and yield that  $d_{\bar{k},k}(G)\left( \frac{k}{2}\right) $ is at least \begin{align*}\frac{1}{2} \binom{\bar{k}-\frac{k}{2}}{2}+\binom{\frac{k}{2}}{2}+\frac{\left(\bar{k}-\frac{k}{2}\right)k}{2}+\Theta \left( \sqrt{ \log  \left[\binom{n-k}{\bar{k}-\frac{k}{2}} \right]\binom{\bar{k}-\frac{k}{2}}{2}} \right) -O\left( \sqrt{\bar{k}k \log n} +\left(\bar{k}\right)^{\beta}\sqrt{\log n}\right)\end{align*} which since $\beta<1$ and $k \leq \bar{k}$ is at least \begin{align} \label{LBOGP} \frac{1}{2} \binom{\bar{k}-\frac{k}{2}}{2}+\binom{\frac{k}{2}}{2}+\frac{\left(\bar{k}-\frac{k}{2}\right)k}{2}+\Theta \left( \sqrt{ \log  \left[\binom{n-k}{\bar{k}-\frac{k}{2}} \right]\binom{\bar{k}-\frac{k}{2}}{2}} \right) -O\left( \bar{k} \sqrt{ \log n} \right).\end{align} Furthermore, Lemma \ref{EntTaylor} implies \begin{align} \label{Gam} \Gamma_{\bar{k},k}(0)= \frac{1}{2} \binom{\bar{k}}{2}+\Theta \left( \sqrt{ \log \left[  \binom{n-k}{\bar{k}}\right] \binom{\bar{k}}{2}} \right). \end{align} Hence,  to establish (\ref{slackG23}) using (\ref{LBOGP}), (\ref{Gam}) it suffices to show that 
$$\frac{1}{2} \binom{\bar{k}-\frac{k}{2}}{2}+\binom{\frac{k}{2}}{2}+\frac{\left(\bar{k}-\frac{k}{2}\right)\frac{k}{2}}{2}+\Theta \left( \sqrt{ \log  \left[\binom{n-k}{\bar{k}-\frac{k}{2}} \right]\binom{\bar{k}-\frac{k}{2}}{2}} \right)$$ is bigger than $$\frac{1}{2} \binom{\bar{k}}{2}+\Theta \left( \sqrt{ \log \left[  \binom{n-k}{\bar{k}} \right] \binom{\bar{k}}{2}} \right)+\omega\left( \bar{k}\sqrt{ \log n} \right).$$ By direct computation we have $$\frac{1}{2} \binom{\bar{k}-\frac{k}{2}}{2}+\binom{k/2}{2}+\frac{\left(\bar{k}-\frac{k}{2}\right)\frac{k}{2}}{2}-\frac{1}{2} \binom{\bar{k}}{2}=\frac{k^2}{16}-O\left(\bar{k}\right).$$Hence, it suffices to have \begin{align} \label{gooalOGP2} k^2=\omega \left(\sqrt{ \log  \left[\binom{n-k}{\bar{k}}\right] \binom{\bar{k}}{2}} - \sqrt{ \log  \left[\binom{n-k}{\bar{k}-k}  \right] \binom{\bar{k}-k}{2}} \right) +\omega \left(\bar{k}\sqrt{\log n}\right).\end{align} Now using the elementary $ \binom{n-k}{\bar{k}} \leq (\frac{ne}{\bar{k}})^{\bar{k}}$ and $\binom{n-k}{\bar{k}-k} \geq (\frac{(n-\bar{k})}{\bar{k}})^{\bar{k}-k}$ and the fact that $\bar{k} =\Theta \left( n^C\right)$ for $C<1/2$ we conclude for (\ref{gooalOGP2}) to hold, it suffices to have $$k^2=\omega( \left((\bar{k})^{\frac{3}{2}}-(\bar{k}-k)^{\frac{3}{2}}\right) \sqrt{ \log n})+\omega\left(\bar{k}\sqrt{\log n}\right).$$ Using the elementary inequality, implied by mean value theorem, that for $0< a \leq b$, $a^{\frac{3}{2}}-b^{\frac{3}{2}}\leq \frac{3}{2} \sqrt{a} \left(a-b \right)$ it suffices  $$k^2=\omega \left(\sqrt{\bar{k}}k \sqrt{\log n}\right)+\omega\left(\bar{k}\sqrt{\log n}\right)$$ which now follows directly from our assumptions $k^2=\omega \left(\bar{k} \log \frac{n}{k^2} \right)$ and $k \leq \bar{k}= n^C$ for $C<1/2$. The proof of (\ref{slackG23}) and therefore of (\ref{slackGG}) and (\ref{slack22}) is complete.

We now show how (\ref{slackOGP}), (\ref{slackG1}) and (\ref{slackG23}) established above imply the presence of OGP w.h.p. as $n \rightarrow + \infty$. We set $$\zeta_1:=u_1,\zeta_2:=u_2 \text{ and } r:= \Gamma_{\bar{k},k}(0) -\frac{c_0}{2}\frac{\bar{k}}{ \log \left(\frac{n}{\bar{k}}\right)}.$$ 
We start with the second property of $\bar{k}$-OGP. By the definition of $\zeta_1,\zeta_2,r$ and (\ref{slackOGP}) we have $$\max_{z \in \mathcal{I}\cap \left[\zeta_1,\zeta_2\right]} d_{\bar{k},k}(G)(z) <r,$$w.h.p. as $n \rightarrow + \infty$. Using now the definition of $d_{\bar{k},k}(G)(z)$ the last displayed equality directly implies that there is no $\bar{k}$-vertex subset $A$ with $|\mathrm{E}[A]| \geq r$ with $|A\cap \mathcal{PC}| \in [\zeta_1,\zeta_2]$, as we wanted.

For the first part, notice that (\ref{slackG1}), (\ref{slackG23}) and the definition of $r$ imply $$\min \{ d_{\bar{k},k}(G)(0) ,d_{\bar{k},k}(G)(\frac{k}{2})\} >r,$$w.h.p. as $n \rightarrow + \infty$. Using the definitions of $d_{\bar{k},k}(G)(0) ,d_{\bar{k},k}(G)(\frac{k}{2})$ respectively we conclude the existence of a $\bar{k}$-vertex subset $A$ with $|\mathrm{E}[A]| \geq r$ and $|A\cap \mathcal{PC}|=0$ and of a $\bar{k}$-vertex subset $B$ with $|\mathrm{E}[B]| \geq r$ and $|B\cap \mathcal{PC}|=\frac{k}{2}$, w.h.p. as $n \rightarrow + \infty$. Since (\ref{order}) holds, we conclude the first property of $\bar{k}$-OGP. This completes the proof of the presence of $\bar{k}$-OGP and of Theorem \ref{OGP}.

\end{proof}

\section{OGP implies Failure of an MCMC family}\label{secMix}

\begin{proof}[Proof for Proposition \ref{Prop1few}]
We use the Conjecture \ref{mainconj} for for the cases (1i) and (2i) and $\epsilon=\frac{1}{2}$.
Using our parameter assumptions from Part (a) in Theorem \ref{FM} we have that for large enough $n$, $$\floor*{C_0 \frac{k\bar{k}}{n}} \leq  \floor*{\sqrt{D_1\frac{\bar{k}}{\log \left( \frac{n}{\bar{k}}\right)}}}.$$  In particular, the  $\bar{k}$-subgraphs with overlap $\floor*{C_0 \frac{k\bar{k}}{n}}$ are  included in $A_0$.
Furthermore, by definition, the  $\bar{k}$-subgraphs with overlap $\floor*{k/2}$ are  included in $A_2$.
Now using the definition of the Gibbs measure $\pi_{\beta}$ we can lower bound almost surely $\pi_{\beta}\left(A_0\right)$ by the mass of the densest $\bar{k}$-subgraph with overlap $\floor*{C_0 \frac{k\bar{k}}{n}}$ and $\pi_{\beta}\left(A_2\right)$ by the mass of the densest $\bar{k}$-subgraph with overlap $\floor*{\frac{k}{2}}$. Hence, by the definition of the functions $d_{\bar{k},k}\left(G\right)(z)$ we have
\begin{align}\label{FirstStep}
\min\{\pi_{\beta}\left(A_0\right),\pi_{\beta}\left(A_2\right)\}  \geq \frac{1}{Z}\exp\left(\min\{d_{\bar{k},k}\left(G\right)\left(C_0\frac{\bar{k}k}{n}\right),d_{\bar{k},k}\left(G\right)\left(\frac{k}{2}\right)\}\right),
\end{align}where $Z$ is defined in (\ref{part}).

On the other hand again by the definition of the Gibbs measure and the definition of the functions $d_{\bar{k},k}\left(G\right)(z)$ we can upper bound each mass of every element of $A_1$  almost surely by $\frac{1}{Z}    \exp\left(  \max_{z \in \mathcal{I} \cap \left[u_1,u_2\right]} d_{\bar{k},k}(G)(z) \right)$, where $u_1:=\floor*{D_1\sqrt{\frac{\bar{k}}{\log \left( \frac{n}{\bar{k}}\right)}}}$ and $u_2:=\floor*{D_2\sqrt{\frac{\bar{k}}{\log \left( \frac{n}{\bar{k}}\right)}}}$. Hence, we conclude
\begin{align}\label{SecondStep}
\pi_{\beta}\left(A_1\right) &\leq |A_1| \frac{1}{Z}    \exp\left(  \max_{z \in \mathcal{I} \cap \left[u_1,u_2\right]} d_{\bar{k},k}(G)(z) \right) \notag \\
& \leq \binom{n}{\bar{k}} \frac{1}{Z}    \exp\left(  \max_{z \in \mathcal{I} \cap \left[u_1,u_2\right]} d_{\bar{k},k}(G)(z) \right) \notag \\
& \leq \frac{1}{Z}    \exp\left( k \log \left(\frac{ne}{\bar{k}}\right)+ \max_{z \in \mathcal{I} \cap \left[u_1,u_2\right]} d_{\bar{k},k}(G)(z) \right),
\end{align} where we have used the elementary inequalities $|A_1| \leq \binom{n}{\bar{k}} \leq \left(\frac{ne}{\bar{k}}\right)^{\bar{k}}.$

Now using (\ref{slackConj2}) we  further conclude 
\begin{align}\label{ThirdStep}
\pi_{\beta}\left(A_1\right) \leq \frac{1}{Z}    \exp\left( k \log n-\beta \Omega\left( \frac{\bar{k}}{\log \frac{n}{\bar{k}}}\right)+\min\{d_{\bar{k},k}\left(G\right)\left(C_0\frac{\bar{k}k}{n}\right),d_{\bar{k},k}\left(G\right)\left(\frac{k}{2}\right)\}\right)
\end{align} 
Combining (\ref{FirstStep}) and (\ref{ThirdStep}) we have
\begin{align}\label{FourthStep}
\min\{\pi_{\beta}\left(A_0\right),\pi_{\beta}\left(A_2\right)\}  \geq   \exp\left( -k \log \left(\frac{ne}{\bar{k}}\right)+\beta \Omega\left( \frac{\bar{k}}{\log \frac{n}{\bar{k}}}\right) \right)\pi_{\beta}\left(A_1\right).
\end{align}Finally, since $\beta=\omega\left( \left(\log \left(\frac{n}{\bar{k}}\right)\right)^{\frac{3}{2}}\right)$ and $\bar{k}=o\left(n\right)$ we have that 
\begin{align}\label{FifthStep}-k \log \left(\frac{ne}{\bar{k}}\right)+ \Omega\left(\beta  \frac{\bar{k}}{\log \frac{n}{\bar{k}}}\right)=\Omega\left(\beta  \frac{\bar{k}}{\log \frac{n}{\bar{k}}}\right).
\end{align}Combining (\ref{FourthStep}) and (\ref{FifthStep}) concludes the proof of (\ref{few}).

\end{proof}

Towards proving Theorem \ref{thmslow} we establish the following proposition. 
\begin{proposition}\label{Prop2few}

Suppose that the free energy well property (\ref{few}) holds. Then for some constant $c_0>0$, and any $T>0$,
\begin{align}
\mathbb{P}\left( \tau_{\beta} \leq T\right) \leq T\exp\left(-\Omega\left( \beta \frac{\bar{k}}{\log \frac{n}{\bar{k}}}\right)\right)
\end{align}w.h.p. as $n \rightarrow + \infty$.
\end{proposition}

Note that Theorem \ref{thmslow} readily follows by combining Proposition \ref{Prop1few} and \ref{Prop2few}. For this reason for the rest of the section we establish Proposition \ref{Prop1few}.

\begin{proof}[Proof of Proposition \ref{Prop2few}]
We first consider the Markov Chain $\tilde{X}_t$ defined on $A_0 \cup A_1$, which is $X_t$  reflected on the boundary of $A_0 \cup A_1$, $A:=\partial \left(A_0 \cup A_1\right)$. Note that $A$ is simply the set of $\bar{k}$-subgraphs with overlap exactly equal to $\ceil*{D_2 \sqrt{\frac{\bar{k}}{\log \left( \frac{n}{\bar{k}}\right)}}}$. To be more precise if $Q(x,y)$ is the transition matrix of $X_t$, $\tilde{X}_t$ has transition matrix $\tilde{Q}$ where $\tilde{Q}(x,y)=Q(x,y)$ if $x \in \left(A_0 \cup A_1\right) \setminus A$ and $y \in A_0 \cup A_1$, and for $x \in A$, \[ 
\tilde{Q}(x,y)= \left\{
\begin{array}{ll}
     Q(x,y), & y \in A_0 \cup A_1 \\
    0, & \text{ otherwise}
\end{array} 
\right. 
\]

Now using the detailed balanced equations, $\tilde{X}_t$ is reversible with respect to $ \pi_{\beta}\left(\cdot | A_0 \cup A_1\right)$. In particular by coupling the initialization of the Markov chains; $\tilde{X}_{0,\beta}=X_{0,\beta} \sim \pi_{\beta}\left(\cdot | A_0 \cup A_1\right),$ we have that for any $t \leq \tau_{\beta}$, almost surely \begin{align} \label{coal} \tilde{X}_{t,\beta}=X_{t,\beta}\end{align} and \begin{align} \label{det} \tilde{X}_{t,\beta} \sim \pi_{\beta}\left(\cdot | A_0 \cup A_1\right).\end{align}     Furthermore, by definition of $\tau_{\beta}$,
 \begin{align} \label{stop} X_{\tau_{\beta}-1} \in A:=\partial \left(A_0 \cup A_1\right).\end{align}

Combining the above we have,

\begin{align}
\mathbb{P}\left( \tau_{\beta} \leq T\right) & \leq \sum_{i=1}^T \mathbb{P}\left( \tau_{\beta} =i \right) \notag \\
& \leq \sum_{i=1}^T \mathbb{P}\left( \tau_{\beta} =i,  X_{i-1} \in A\right), \text{ by }(\ref{stop}) \notag\\
& = \sum_{i=1}^T \mathbb{P}\left( \tau_{\beta} =i,  \tilde{X}_{i-1} \in A\right) \text{by } (\ref{coal}) \notag \\
& \leq \sum_{i=1}^T \mathbb{P}\left(  \tilde{X}_{i-1} \in A\right) \notag\\
& =T \mathbb{P}\left(  \tilde{X}_{0} \in A\right), \text{by } (\ref{det}) \notag\\
& =T \pi_{\beta}\left(  A | A_0\cup A_1\right), \text{by } (\ref{det}) \label{finalslow1}.
\end{align}

Finally notice that 
\begin{align}
\pi_{\beta}\left(  A | A_0\cup A_1\right) & = \frac{ \pi_{\beta}\left(A\right)}{\pi_{\beta}\left(A_0\cup A_1\right)} \notag \\
& \leq \frac{ \pi_{\beta}\left(A_1\right)}{\pi_{\beta}\left(A_0\right)} \notag \\
& \leq \exp\left(-\Omega\left( \beta \frac{\bar{k}}{\log \frac{n}{\bar{k}}}\right)\right), \text{by } (\ref{few}) \label{finalslow2}
\end{align}

Combining (\ref{finalslow1}) and (\ref{finalslow2}) completes the proof of Proposition \ref{Prop2few}.
\end{proof}

\newpage

%
%
%
%
%

\section{Conclusion and future directions}

Our paper studies the OGP for the planted clique problem. We focus on the way dense subgraphs of the observed graph $G\left(n,\frac{1}{2},k\right)$ overlap with the planted clique and offer first moment evidence of a fundamental OGP phase transition at $k=\Theta\left(\sqrt{n}\right)$. We establish parts of the conjectured OGP phase transition by showing that (1) for any $k \leq n^{0.0917}$ OGP appears and (2) for any $k>\sqrt{n \log n}$ a simple local method can exploit the smooth local structure of dense subgraphs and succeed. All of our results are for overparametrized $\ok$-vertex dense subgraphs, where $\ok \geq k$. Introducing this additional free parameter is essential for establishing our results.

Our paper prompts to multiple future research directions.

\begin{itemize}
\item[(1)] The first and most relevant future direction is establishing the rest parts of Conjecture \ref{mainconj}. We pose this as the main open problem out of this work.

\item[(2)] Our result on the value of the $K$-densest subgraph of an \ER model $G\left(n,\frac{1}{2}\right)$ shows tight concentration of the first and second order behavior of the quantity $d_{\mathrm{ER},K}(G_0)$ defined in (\ref{ERdense}), and applies for any $K \leq n^{0.5-\epsilon}$, for $\epsilon>0$.

 Improving on the third order bounds established in Corollary \ref{CorD} is of high interest. If the third order term can be proven to be $o\left(K\right)$ for any $K \leq n^{0.5-\epsilon}$ (currently established only for $K \leq n^{0.0917}$) then the first part 1(a) of Conjecture \ref{mainconj} can be established by following the arguments of this paper. 

Studying the $K$-densest subgraph problem for higher values of $K$ appears also an interesting mathematical quest. According to Corollary \ref{CorD} the second order term behaves as $\Theta\left(K^{\frac{3}{2}}\right)$ (up-to-$\log n$ terms) when $K \leq n^{\frac{1}{2}-\epsilon}$ but for $K=n$ it is easy to see that the second order term behaves as $O\left(K\right)$ (up-to-$\log n$ terms). For which critical value of $K$ does the transition happen? We conjecture the transition to happen at $K=n^{\frac{2}{3}}.$ The identification of the exact constants related to the second order term is also interesting. When $K=\Theta \left(n\right)$ the constant is naturally expected to be related to the celebrated Parisi formula (see e.g. \cite{Aukosh2018} for similar results for the random MAX-CUT problem and  \cite{SubhaOpt} for a general technique).

\item[(3)] In this paper, we use the overparametrization technique as a way to study the landscape of the planted clique problem.  Overparametrization has been used extensively in the literature for smoothening ``bad" landscapes, but predominantly in the context of deep learning. To the best of our knowledge this is the first time it is used to study computational-statistical gaps. Without overparametrization the first moment curve obtains a phase transition at the peculiar threshold $k=n^{\frac{2}{3}}$, far above the conjectured onset of algorithmic hardness threshold $k=\sqrt{n}$. \textit{Can the technique of overparametrization be used to study the OGP phase transition of other computational-statistical gaps?} An interesting candidate would be the 3-tensor PCA problem. In this problem, a landscape property called free-energy wells, which is similar to OGP, seems to be appearing in a different place from the conjectured algorithmic threshold (see \cite{AukoshPCA} and the discussion therein). It would be very interesting if the free energy wells-algorithmic gap could close using the overparametrization technique.

%
%

\end{itemize}

\section*{Acknowledgments}
D.G.   acknowledges the support from the Office of Naval Research Grant N00014-17-1-2790.
I.Z. would like to thank Arian Maleki for a helpful discussion on the use of overparametrization during the creation of the paper.

\bibliographystyle{alpha}
\bibliography{biblioCOLT19}

\newcommand{\etalchar}[1]{$^{#1}$}
\begin{thebibliography}{ACOGM17}

\bibitem[Abb17]{Abbe17}
Emmanuel Abbe.
\newblock Community detection and stochastic block models: recent developments.
\newblock {\em arXiv}, 2017.

\bibitem[ACO08]{achlioptas2008algorithmic}
Dimitris Achlioptas and Amin Coja-Oghlan.
\newblock Algorithmic barriers from phase transitions.
\newblock In {\em Foundations of Computer Science, 2008. FOCS'08. IEEE 49th
  Annual IEEE Symposium on}, pages 793--802. IEEE, 2008.

\bibitem[ACOGM17]{Coja17}
Peter Ayre, Amin Coja-Oghlan, Pu~Gao, and Noela Muller.
\newblock The satisfiability threshold for random linear equations.
\newblock {\em arXiv}, 2017.

\bibitem[ACORT11]{AchlioptasCojaOghlanRicciTersenghi}
D.~Achlioptas, A.~Coja-Oghlan, and F.~Ricci-Tersenghi.
\newblock On the solution space geometry of random formulas.
\newblock {\em Random Structures and Algorithms}, 38:251--268, 2011.

\bibitem[AKKT02]{achlioptas2002two}
Dimitris Achlioptas, Jeong~Han Kim, Michael Krivelevich, and Prasad Tetali.
\newblock Two-coloring random hypergraphs.
\newblock {\em Random Structures \& Algorithms}, 20(2):249--259, 2002.

\bibitem[AKS98]{alon1998finding}
Noga Alon, Michael Krivelevich, and Benny Sudakov.
\newblock Finding a large hidden clique in a random graph.
\newblock {\em Random Structures and Algorithms}, 13(3-4):457--466, 1998.

\bibitem[BAGJ18]{AukoshPCA}
Gerard Ben~Arous, Reza Gheissari, and Aukosh Jagannath.
\newblock Algorithmic thresholds for tensor pca.
\newblock {\em arXiv}, 2018.

\bibitem[BBH18]{BreslerCom}
Matthew Brennan, Guy Bresler, and Wasim Huleihel.
\newblock Reducibility and computational lower bounds for problems with planted
  sparse structure.
\newblock {\em Conference on Learning Theory (COLT)}, 2018.

\bibitem[BBSV18]{Bollobas18}
Paul Balister, Bela Bollobas, Julian Sahasrabudhe, and Alexander Veremyev.
\newblock Dense subgraphs in random graphs.
\newblock {\em arXiv Preprint}, 2018.

\bibitem[BHK{\etalchar{+}}16]{SamClique}
Boaz Barak, Samuel Hopkins, Jonathan Kelner, Pravesh Kothari, Ankur Moitra, and
  Aaron Potechin.
\newblock A nearly tight sum-of-squares lower bound for the planted clique
  problem.
\newblock {\em 57th Annual Symposium on Foundations of Computer Science
  (FOCS)}, 2016.

\bibitem[BMR{\etalchar{+}}18]{Jiaming18}
J.~Banks, C.~Moore, Vershynin R., N.~Verzelen, and J.~Xu.
\newblock Information-theoretic bounds and phase transitions in clustering,
  space pca, and submatrix localization.
\newblock {\em IEEE Transactions on Information Theory}, 2018.

\bibitem[Bol01]{BollBook}
Bela Bollobas.
\newblock {\em Random Graphs}.
\newblock Cambridge studies in advanced mathematics, 2001.

\bibitem[BPW18]{BandeiraNotes}
Afonso~S Bandeira, Amelia Perry, and Alexander~S. Wein.
\newblock Notes on computational-statistical gaps: predictions using
  statistical physics.
\newblock {\em Arxiv preprint arXiv:1803.11132.pdf}, 2018.

\bibitem[BR13]{berthet2013complexity}
Quentin Berthet and Philippe Rigollet.
\newblock Complexity theoretic lower bounds for sparse principal component
  detection.
\newblock In {\em Conference on Learning Theory}, pages 1046--1066, 2013.

\bibitem[CGPR17]{GamarnikSpin}
Wei-Kuo Chen, David Gamarnik, Dmitry Panchenko, and Mustazee Rahman.
\newblock Suboptimality of local algorithms for a class of max-cut problems,
  2017.

\bibitem[CLR17]{CaiGaps}
Tony Cai, Tengyuan Liang, and Alexander Rakhlin.
\newblock Computational and statistical boundaries for submatrix localization
  in a large noisy matrix.
\newblock {\em The Annals of Statistics}, 2017.

\bibitem[COE11]{coja2011independent}
A.~Coja-Oghlan and C.~Efthymiou.
\newblock On independent sets in random graphs.
\newblock In {\em Proceedings of the Twenty-Second Annual ACM-SIAM Symposium on
  Discrete Algorithms}, pages 136--144. SIAM, 2011.

\bibitem[COHH16]{coja2016walksat}
Amin Coja-Oghlan, Amir Haqshenas, and Samuel Hetterich.
\newblock Walksat stalls well below the satisfiability threshold.
\newblock {\em arXiv preprint arXiv:1608.00346}, 2016.

\bibitem[DGGP14]{PeresClique}
Yael Dekel, Ori Gurel-Gurevich, and Yuval Peres.
\newblock Finding hidden cliques in linear time with high probability.
\newblock {\em Combinatorics, Probabability and Computing}, 2014.

\bibitem[DM13]{DeshpandeMontanari}
Y.~Deshpande and A.~Montanari.
\newblock Finding hidden cliques of size $\sqrt{N/e}$ in nearly linear time.
\newblock {\em arxiv.org/abs/1304.7047}, 2013.

\bibitem[FGR{\etalchar{+}}17]{VitalyClique}
Vitaly Fedman, Elena Grigorescu, Lev Reyzin, Santosh Vempala, and Ying Xiao.
\newblock Statistical algorithms and a lower bound for detecting planted
  cliques.
\newblock {\em Journal of the ACM (JACM)}, 2017.

\bibitem[FR10]{FeigeClique}
Uriel Feige and Dorit Ron.
\newblock Finding hidden cliques in linear time.
\newblock In {\em AofA}, 2010.

\bibitem[GJ19]{gamarnik2019overlap}
David Gamarnik and Aukosh Jagannath.
\newblock The overlap gap property and approximate message passing algorithms
  for $ p $-spin models.
\newblock {\em arXiv preprint arXiv:1911.06943}, 2019.

\bibitem[GL16]{gamarnik2016finding}
David Gamarnik and Quan Li.
\newblock Finding a large submatrix of a gaussian random matrix.
\newblock {\em arXiv preprint arXiv:1602.08529}, 2016.

\bibitem[GM75]{Grimmet75}
G.~Grimmett and C.~McDiarmid.
\newblock On colouring random graphs.
\newblock {\em Mathematical Proceedings of the Cambridge Philosophical
  Society}, 1975.

\bibitem[GSa]{gamarnik2014limits}
David Gamarnik and Madhu Sudan.
\newblock Limits of local algorithms over sparse random graphs.
\newblock {\em Annals of Probability. {\rm To appear}}.

\bibitem[GSb]{gamarnik2014performance}
David Gamarnik and Madhu Sudan.
\newblock Performance of sequential local algorithms for the random nae-k-sat
  problem.
\newblock {\em SIAM Journal on Computing. {\rm To appear}}.

\bibitem[GZ17a]{gamarnikzadik}
David Gamarnik and Ilias Zadik.
\newblock High dimensional linear regression with binary coefficients: Mean
  squared error and a phase transition.
\newblock {\em Conference on Learning Theory (COLT)}, 2017.

\bibitem[GZ17b]{gamarnikzadik2}
David Gamarnik and Ilias Zadik.
\newblock Sparse high dimensional linear regression: Algorithmic barrier and a
  local search algorithm.
\newblock {\em arXiv Preprint}, 2017.

\bibitem[Het16]{hetterich2016analysing}
Samuel Hetterich.
\newblock Analysing survey propagation guided decimation on random formulas.
\newblock {\em arXiv preprint arXiv:1602.08519}, 2016.

\bibitem[Jer92]{JerrumClique}
Mark Jerrum.
\newblock Large cliques elude the metropolis process.
\newblock {\em Random Structures \& Algorithms}, 1992.

\bibitem[JS18]{Aukosh2018}
Aukosh Jagannath and Subhabrata Sen.
\newblock On the unbalanced cut problem and the generalized
  sherrington-kickpatrick model.
\newblock {\em arXiv Preprint}, 2018.

\bibitem[Kar72]{KarpNP}
Richard~M Karp.
\newblock Reducibility among combinatorial problems.
\newblock {\em Complexity of Computer Computations}, 1972.

\bibitem[Kla00]{Tail00}
Bernhard Klar.
\newblock Bounds on tail probabilities of discrete distributions.
\newblock {\em Probability in the Engineering and Informational Sciences},
  2000.

\bibitem[KMRT{\etalchar{+}}07]{krzakala2007gibbs}
F.~Krzaka{\l}a, A.~Montanari, F.~Ricci-Tersenghi, G.~Semerjian, and
  L.~Zdeborov{\'a}.
\newblock Gibbs states and the set of solutions of random constraint
  satisfaction problems.
\newblock {\em Proceedings of the National Academy of Sciences},
  104(25):10318--10323, 2007.

\bibitem[Ku{\v c}95]{KuceraClique}
Lud{\v e}k Ku{\v c}era.
\newblock Expected complexity of graph partitioning problems.
\newblock {\em Discrete Applied Mathematics}, 1995.

\bibitem[LMZ18]{TengyuNeural}
Yuanzhi Li, Tengyu Ma, and Hongyang Zhang.
\newblock Algorithmic regularization in over-parameterized matrix sensing and
  neural networks with quadratic activations.
\newblock {\em Conference on Learning Theory (COLT)}, 2018.

\bibitem[MMZ05]{mezard2005clustering}
M.~M{\'e}zard, T.~Mora, and R.~Zecchina.
\newblock Clustering of solutions in the random satisfiability problem.
\newblock {\em Physical Review Letters}, 94(19):197205, 2005.

\bibitem[Mon18]{montanari2018optimization}
Andrea Montanari.
\newblock Optimization of the sherrington-kirkpatrick hamiltonian.
\newblock {\em arXiv preprint arXiv:1812.10897}, 2018.

\bibitem[MRT11]{montanari2011reconstruction}
Andrea Montanari, Ricardo Restrepo, and Prasad Tetali.
\newblock Reconstruction and clustering in random constraint satisfaction
  problems.
\newblock {\em SIAM Journal on Discrete Mathematics}, 25(2):771--808, 2011.

\bibitem[RV14]{rahman2014local}
Mustazee Rahman and Balint Virag.
\newblock Local algorithms for independent sets are half-optimal.
\newblock {\em arXiv preprint arXiv:1402.0485}, 2014.

\bibitem[RXZ19]{zadik19}
Galen Reeves, Jiaming Xu, and Ilias Zadik.
\newblock The all-or-nothing phenomenon in sparse linear regression.
\newblock {\em arXiv Preprint}, 2019.

\bibitem[Sen]{SubhaOpt}
Subhabrata Sen.
\newblock Optimization on sparse random hypergraphs and spin glasses.
\newblock {\em Random Structures and Algorithms, {\rm to appear}.}

\bibitem[SS17]{ShamirNN}
Itay Safran and Ohad Shamir.
\newblock Spurious local minima are common in two-layer relu neural networks.
\newblock {\em arXiv}, 2017.

\bibitem[Sub18]{subag2018following}
Eliran Subag.
\newblock Following the ground-states of full-rsb spherical spin glasses.
\newblock {\em arXiv preprint arXiv:1812.04588}, 2018.

\bibitem[Tal10]{TalagrandBook}
M.~Talagrand.
\newblock {\em Mean Field Models for Spin Glasses: Volume {I}: Basic Examples}.
\newblock Springer, 2010.

\bibitem[VBB18]{AfonsoNeural}
Luca Venturi, Afonso~S Bandeira, and Joan Bruna.
\newblock Spurious valleys in two-layer neural network optimization landscapes.
\newblock {\em arXiv Preprint arxiv:1802.06384}, 2018.

\bibitem[WBP16]{BertherRIP}
Tengyao Wang, Quentin Berthet, and Yaniv Plan.
\newblock Average-case hardness of rip certification.
\newblock {\em Neural Information Processing Systems (NeurIPS)}, 2016.

\bibitem[WX18]{Wu18}
Yihong Wu and Jiaming Xu.
\newblock Statistical problems with planted structures: Information theoretical
  and computational limits.
\newblock {\em arXiv Preprint}, 2018.

\bibitem[XHM18]{ArianEM}
Ji~Xu, Daniel Hsu, and Arian Maleki.
\newblock Benefit of over-parameterization with em.
\newblock {\em Neural Information Processing Systems (NeurIPS)}, 2018.

\end{thebibliography}

\appendix

\newpage
\section{Auxilary lemmas}

\begin{lemma}\label{lem:SumMax}
Let $M,N \in \mathbb{N}$ with $N \rightarrow + \infty$. Let $X_1,X_2,\ldots,X_N$ abritrary correlated random variables and $Y_1,Y_2,\ldots,Y_N$ i.i.d. $\mathrm{Bin}\left(M,\frac{1}{2}\right)$, all living in the same probability space. We also assume $\left(Y_i\right)_{i=1,2,\ldots,N}$ are independent of $\left(X_i\right)_{i=1,2,\ldots,N}$. Then $$\max_{i=1,2,\ldots,N} \{X_i+Y_i\} \geq \max_{i=1,2,\ldots,N} \{ X_i \}+\max\{\frac{M}{2}-\sqrt{M \log \log N},0\},$$ w.h.p. as $N \rightarrow + \infty$.

\end{lemma}

\begin{proof}
It suffices to show that for $i^*:=\arg \max_{i=1,2,\ldots,N} X_i$, $$Y_{i^*} \geq \frac{M}{2}-\sqrt{M \log \log N}$$ w.h.p. as $N \rightarrow + \infty$. The result now easily follows from standard Chernoff bound and independence between $\left(Y_i\right)_{i=1,2,\ldots,N}$ and $i^*$.
\end{proof}

For the following two lemmas recall that $h$ is defined in (\ref{BinEnt}) and for $\gamma \in (\frac{1}{2},1)$, $r(\gamma,\frac{1}{2})$ is defined in (\ref{alpha}).
\begin{lemma}\label{Bin}
Let $N \in \mathbb{N}$ growing to infinity and $\gamma=\gamma_N >\frac{1}{2}$ with $\lim_N \gamma_N=\frac{1}{2}$ and $\lim_N \left(\gamma_N-\frac{1}{2}\right) \sqrt{N}=+\infty$.

Then for $X$ following $\mathrm{Bin}(N,\frac{1}{2})$ and $N \rightarrow + \infty$ it holds

$$\mathbb{P} \left( X=  \ceil*{\gamma N} \right)=\exp \left(-N r(\gamma,\frac{1}{2})-\frac{ \log N}{2}+O\left(1\right) \right)$$and
$$\mathbb{P} \left( X \geq  \ceil*{\gamma N} \right)=\exp \left(-N r(\gamma,\frac{1}{2}) -\Omega\left(\log \left(\left(\gamma-\frac{1}{2}\right) \sqrt{N} \right)\right)\right),$$where $r\left(\gamma,\frac{1}{2}\right)$ is defined in (\ref{alpha}).
\end{lemma}

\begin{proof}
We have by Stirling approximation 
\begin{equation}\label{Stirling} 
\binom{N}{\ceil{\gamma N}}=\exp \left( N h \left(\gamma\right)-\frac{1}{2} \log \left(N \gamma \left(1-\gamma\right) \right)+O \left(1\right)  \right)
\end{equation}
In particular, using $r\left(\gamma,\frac{1}{2}\right)=h\left(\frac{1}{2}\right)-h\left(\gamma  \right)=\log 2-h\left(\gamma  \right)$ and that $\gamma=\frac{1}{2}+o_N\left(1\right)$ we conclude
$$\mathbb{P} \left( X=  \ceil*{\gamma N} \right)=\binom{N}{\ceil{\gamma N}}\frac{1}{2^N}=\exp(-N r(\gamma,\frac{1}{2})-\frac{1}{2}\log N+O\left(1\right))$$ 

Now using standard binomial coefficient inequalities (see e.g. Proposition 1(c) in \cite{Tail00}) we have that for any $1 \leq k \leq N/2$,
\begin{align*}
\mathbb{P} \left( X \geq  \ceil*{\frac{N}{2}+k} \right) \leq \frac{\frac{N}{2}+k}{2k+1}\mathbb{P} \left( X=  \ceil*{\frac{N}{2}+k} \right).
\end{align*}Hence for large enough $N$ if we set $k=\left(\gamma-\frac{1}{2}\right)N$ we have,
\begin{align*}
\mathbb{P} \left( X \geq  \ceil*{\gamma N} \right) &\leq \left(\frac{\gamma}{2\gamma-1}+o\left(1\right)\right)\mathbb{P} \left( X=  \ceil*{\gamma N} \right)\\
&=\left(\frac{\gamma}{2\gamma-1}+o\left(1\right)\right)\exp(-N r(\gamma,\frac{1}{2})-\frac{1}{2}\log N+O\left(1\right))\\
&=\left(1+o\left(1\right)\right) \left(\frac{\gamma}{2\gamma-1}\right)\exp(-N r(\gamma,\frac{1}{2})-\frac{1}{2}\log N+O\left(1\right)), \text{ since } \lim_N \gamma_N =\frac{1}{2}>0\\
&= \exp \left(-N r(\gamma,\frac{1}{2})+\log \left(\frac{2\gamma}{(2\gamma-1) \sqrt{N}}\right)+O(1) \right)\\
&=\exp \left(-N r(\gamma,\frac{1}{2})-\Omega\left(\log \left(\left(\gamma-\frac{1}{2}\right) \sqrt{N} \right)\right)\right).
\end{align*}

The proof of the Lemma \ref{Bin} is complete.
\end{proof}

\begin{lemma}\label{EntTaylor}
For $\epsilon=\epsilon_n \rightarrow 0$,  it holds $$h^{-1}\left(\log 2-\epsilon\right)=\frac{1}{2}+\frac{1}{\sqrt{2}}\sqrt{\epsilon}-\frac{1}{6 \sqrt{2}}\epsilon^{\frac{3}{2}}+ O\left(\epsilon^{\frac{5}{2}}\right).$$
\end{lemma}

\begin{proof}
 Let $\Phi(x):=\sqrt{\log 2- h\left(\frac{1}{2}+x\right)}, x \in [0,\frac{1}{2}]$. We straightforwardly calculate that for the sequence of derivatives at $0$, $a_i:=\Phi^{(i)}(0)$, $i \in \mathbb{Z}_{\geq 0}$ it holds $a_0=0,a_1=\sqrt{2},a_2=0,a_3= 2 \sqrt{2}$ and  $a_4=0.$

 Notice that for all $\epsilon \in (0,\log 2)$ and $\Phi^{-1}$ the inverse of $\Phi$,  $$h^{-1}\left(\log 2-\epsilon\right)=\frac{1}{2}+\Phi^{-1}(\sqrt{\epsilon}).$$
Lemma follows if we establish that Taylor expansion of $\Phi^{-1}$ around $y=0$ is given by
\begin{equation} \label{TaylGoal}
\Phi^{-1}(y)=\frac{1}{\sqrt{2}}y-\frac{1}{6 \sqrt{2}}y^3+ O\left(y^5\right).
\end{equation}Clearly $\Phi^{-1}(0)=0$.  For $b_i:=\left(\Phi^{-1}\right)^{(i)}(0)$, $i \in \mathbb{Z}_{\geq 0}$ by standard calculations using the Lagrange inversion theorem we have
$b_0=0$,$$b_1=\frac{1}{a_1}=\frac{1}{\sqrt{2}},$$ $$b_2=-\frac{a_2}{2a_1}=0,$$ $$b_3=\frac{1}{2\sqrt{2}}\left[-\frac{a_3}{a_1}+3\left( \frac{a_2}{a_1}\right)^2\right]=-\frac{1}{\sqrt{2}}$$and $$b_4=\frac{1}{4}\left[-\frac{a_4}{a_1}+\frac{10}{3} \frac{a_2a_3}{a_1^2}-60 \frac{a_2}{a_1}\right]=0.$$From this point, Taylor expansion yields that for small $y$ $$\Phi^{-1}(y)=b_0+b_1y+\frac{b_2}{2}y^2+\frac{b_3}{6}y^3+\frac{b_4}{24}y^4+ O\left(y^5\right)$$which given the values of $b_i,i=0,1,2,3,4$ yields (\ref{TaylGoal}). The proof of the Lemma is complete.
\end{proof}

The following elementary calculus properties are used throughout the proof sections.

\begin{lemma}\label{Calc}
Suppose $\left(a_n\right)_{n \in \mathbb{N}}, \left(b_n\right)_{n \in \mathbb{N}}$ are two sequences of positive real numbers. The following are true.

\begin{itemize}

\item[(a)] The sequence $a_n \log a_n $ converges to zero if and only if $a_n$ converges to zero.

\item[(b)]  The sequence $a_n \log a_n $ diverges to infinity if and only if $a_n$ diverges to infinity.
\item[(c)] Suppose $b_n$ diverges to infinity. Then 
$a_n \log a_n= \omega \left(b_n \right)$ if and only if $a_n = \omega \left(\frac{b_n}{\log b_n} \right).$

\item[(d)] Suppose $b_n$ diverges to infinity. Then 
$a_n = \omega \left(b_n \log b_n \right)$ if and only if $\frac{a_n}{ \log a_n } = \omega \left(b_n \right).$

\end{itemize}
\end{lemma}

\begin{proof}
Both properties (a), (b) follow in a straightforward way from the continuity of the mapping $$ x \in  (0,\infty) \rightarrow x \log x \in \mathbb{R},$$ and the limiting behaviors $$\lim_{x \rightarrow 0} x \log x =0, \lim_{x \rightarrow + \infty} x \log x =+ \infty.$$

Regarding property (c): For the one direction, assume \begin{align} \label{calc1} \lim_n \frac{a_n \log a_n }{b_n} = + \infty \end{align} and $c_n$ is defined by $a_n=\frac{c_n b_n}{  \log b_n}$. It suffices to show $c_n$ diverges to infinity. By (\ref{calc1}) we know \begin{align} \label{calc2} \lim_n c_n \frac{ \log \left( \frac{c_n b_n}{ \log b_n } \right)}{ \log b_n}= + \infty.\end{align} Assuming $\liminf_n c_n< + \infty$ it follows since $\lim_n b_n= +\infty$ that \begin{align*} \liminf_n c_n \frac{ \log \left( \frac{c_n b_n}{ \log b_n } \right)}{ \log b_n} \leq  \liminf_n c_n \frac{ \log c_n+ \log  b_n }{ \log b_n} \leq \liminf_n c_n < \infty,\end{align*}a direct contradiction with (\ref{calc2}). This completes the proof of this direction.

For the other direction, assume \begin{align} \label{calc3} \lim_n \frac{a_n  \log b_n}{b_n} = + \infty \end{align} and $c_n$ is defined by $a_n \log a_n =c_n b_n$. It suffices to show $c_n$ diverges to infinity. By (\ref{calc3}) we know \begin{align} \label{calc4} \lim_n c_n \frac{ \log \left( \frac{a_n \log a_n}{ c_n } \right)}{ \log a_n}= + \infty.\end{align} Note that since $b_n$ diverges to infinity (\ref{calc3}) implies that $a_n$ diverges to infinity as well. Assuming $\liminf_n c_n< + \infty$ it follows since $\lim_n a_n= +\infty$ that \begin{align*} \label{calc3} \liminf_n c_n \frac{ \log \left( \frac{a_n \log a_n}{ c_n }\right)}{ \log a_n} \leq  \liminf_n c_n \frac{ \log a_n+ \log \log a_n- \log  c_n }{ \log a_n} \leq \liminf_n c_n < \infty,\end{align*}a direct contradiction with (\ref{calc4}). This completes the proof of this direction.

Property (d) follows by similar reasoning as in the case of property (c).
\end{proof}
\end{spacing}

\end{document}